\documentclass[12pt]{article}
\usepackage{amsmath,amssymb,amsthm,eucal, mathtools, mathrsfs}
\usepackage[T1]{fontenc}
\usepackage{color}
\usepackage[all]{xy}
\usepackage[pdftex, draft=false]{hyperref} 
\usepackage{slashed}
\title{\vspace*{-1pc}%
 Existence and uniqueness of the Levi-Civita connection on noncommutative differential forms}

\author{Bram Mesland\S$^*$, Adam Rennie\dag
\thanks{email: 
\texttt{b.mesland@math.leidenuniv.nl}, \texttt{renniea@uow.edu.au}
}
\\[3pt]
\S Mathematisch Instituut, Universiteit Leiden, Netherlands
\\[3pt]
\dag School of Mathematics and Applied Statistics, University of Wollongong\\
Wollongong, Australia
}

\advance\topmargin by -\headheight
\advance\topmargin by -\headsep
\evensidemargin=\oddsidemargin
\makeatletter
\def\section{\@startsection{section}{1}{\z@}{-3.5ex plus -1ex minus
  -.2ex}{2.3ex plus .2ex}{\large\bf}}
\def\subsection{\@startsection{subsection}{2}{\z@}{-3.25ex plus -1ex
  minus -.2ex}{1.5ex plus .2ex}{\normalsize\bf}}
\makeatother

\numberwithin{equation}{section} 

\theoremstyle{plain} 
\newtheorem{thm}{Theorem}[section]
\newtheorem{lemma}[thm]{Lemma}
\newtheorem{prop}[thm]{Proposition}
\newtheorem{corl}[thm]{Corollary}

\theoremstyle{definition} 
\newtheorem{defn}[thm]{Definition}
\newtheorem*{ass*}{Standing assumption}
\newtheorem{example}[thm]{Example} 
\theoremstyle{remark} 
\newtheorem{rmk}[thm]{Remark}
\newtheorem{rmks}[thm]{Remarks}

\DeclareMathOperator{\Dom}{Dom}   
\DeclareMathOperator{\End}{End}   

\newcommand{\B}{\mathcal{B}}  
\newcommand{\C}{\mathbb{C}}   
\newcommand{\Y}{\mathcal{Y}}

\newcommand{\D}{\mathcal{D}}  
\renewcommand{\d}{\mathrm{d}_\Psi} 
\newcommand{\dee}{\mathrm{d}} 
\renewcommand{\H}{\mathcal{H}}  
\newcommand{\N}{\mathbb{N}}   
\newcommand{\op}{\circ}       
\newcommand{\ox}{\otimes}     
\newcommand{\R}{\mathbb{R}}   
\newcommand{\REnd}{\overrightarrow{\textnormal{End}}^{*}}
\newcommand{\LEnd}{\overleftarrow{\textnormal{End}}^{*}}
\newcommand{\X}{\mathcal{X}}  
\newcommand{\Z}{\mathbb{Z}}   
\newcommand{\nablar}{\overrightarrow{\nabla}}
\newcommand{\nablal}{\overleftarrow{\nabla}}

\newcommand{\bra}[1]{\langle#1|} 
\newcommand{\ket}[1]{|#1\rangle} 
\newcommand{\ketbra}[2]{\lvert#1\rangle\langle#2\rvert} 
\newcommand{\pairing}[2]{\langle #1\mathbin{|}#2\rangle} 
\newcommand{\stroke}{\mathbin|}     
\newcommand{\alphar}{\overrightarrow{\alpha}}
\newcommand{\alphal}{\overleftarrow{\alpha}}
\newcommand{\dt}{\mathrm{d}_{\theta}}
\def\pairL_#1(#2|#3){{}_{#1}(#2\stroke#3)} 
\def\pairR(#1|#2)_#3{(#1\stroke#2)_{#3}} 
\def\scal<#1|#2>{\langle#1\stroke#2\rangle} 

\newbox\ncintdbox \newbox\ncinttbox 
	\setbox0=\hbox{$-$}
	\setbox2=\hbox{$\displaystyle\int$}
	\setbox\ncintdbox=\hbox{\rlap{\hbox
		to \wd2{\hskip-.125em \box2\relax\hfil}}\box0\kern.1em}
	\setbox0=\hbox{$\vcenter{\hrule width 4pt}$}
	\setbox2=\hbox{$\textstyle\int$}
	\setbox\ncinttbox=\hbox{\rlap{\hbox
		to \wd2{\hskip-.175em \box2\relax\hfil}}\box0\kern.1em}

\allowdisplaybreaks
\begin{document}

\maketitle

\vspace{-2pc}

\begin{abstract}
We combine Hilbert module and algebraic techniques
to  give necessary and sufficient conditions for the existence of an Hermitian torsion-free connection on the bimodule of differential one-forms of a first order differential calculus.  In the presence of the extra structure of a bimodule connection, we give sufficient conditions for uniqueness.

We prove that any $\theta$-deformation of a compact Riemannian manifold admits a unique Hermitian torsion-free bimodule connection and provide an explicit construction of it.
Specialising to classical Riemannian manifolds yields a novel construction of the Levi-Civita connection on the cotangent bundle.
\end{abstract}

\tableofcontents
\parskip=4pt
\parindent=0pt

\section{Introduction}
\label{sec:intro}
The fundamental theorem of Riemannian geometry asserts that there exists a unique metric compatible and torsion-free connection on the bimodule of differential one-forms of a Riemannian manifold $(M,g)$.
In this paper we extend this result to a class of second-order differential calculi 
$$
\B\xrightarrow{\dee} \Omega^{1}_{\dee}\xrightarrow{\dee}\Lambda^{2}_{\dee}
$$ 
over a noncommutative algebra $\B$. We require that $\Omega^{1}_{\dee}$ carries an Hermitian $\B$-valued inner product and $\Lambda^{2}_{\dee}\subset \Omega^{1}_{\dee}\ox_{\B}\Omega^{1}_{\dee}$ is the range of a projection $1-\Psi$, a notion which first appeared in \cite{BGJ1,BGJ2} and is also used in \cite[p574]{BMBook} and \cite[Section 4.1]{M22}. 
These assumptions are equivalent to the existence of a projection annihilating the junk tensors, introduced by Connes, \cite[Chapter VI]{BRB}.
In case the second-order calculus consists of projective modules, the existence of $\Psi$ is always guaranteed, and merely gives a realisation of two-forms as two-tensors. This is advantageous for several reasons. 

First of all, suppose that $\mathcal{C}$ is an algebra and $\Omega^{1}_{\dee}\subset \mathcal{C}$ is a subset. Then the multiplication map allows one to represent the two tensors inside $\mathcal{C}$ as well. An embedding of the two-forms inside the two tensors then gives representations of curvature tensors inside tensor powers of $\mathcal{C}$. This structure is present for instance on the module of one-forms of a spectral triple $(\mathcal{B},\H,\D)$, in which case $\mathcal{C}=\mathbb{B}(\H)$.
We have used this viewpoint to define Ricci and scalar curvatures as well as prove Weitzenb\"{o}ck formulae for spectral triples in \cite{MRCurve}. We observe that in \cite[p574]{BMBook} and \cite[Section 4.1]{M22}, the definition of Ricci curvature relies on an inclusion of the two-forms into the two-tensors, analogous to using the projection $1-\Psi$.

Secondly, and more directly relevant to the present paper, if $\Omega^{1}_{\dee}$ carries a $\B$-valued inner product, then so do the two-tensors. When $\Lambda^{2}_{\dee}$ is realised inside the two-tensors, this allows for a direct comparison between connections and the exterior derivative. 


Despite the strong motivation from spectral triples to define two-forms as a submodule rather than as a quotient, this paper applies to any first order calculus satisfying the various definitions we introduce. All our constructions are compatible with, and inspired by, Connes' noncommutative differential geometry. The fundamental structure we exploit and impose assumptions on is a $*$-representation of the universal differential forms, such as that arising from a spectral triple.

The existence of two-forms $\Lambda^2_\dee$ allows us to define the torsion and curvature of connections 
$\nabla:\Omega^1_\dee\to\Omega^1_\dee\ox_\B\Omega^1_\dee$. Many previous approaches, such as \cite{BGJ2,BGJ1}, relied on the presence of a flip map to emulate some classical constructions, such as antisymmetrisation. While
 $2\Psi-1$ on $\Omega^1_\dee\ox_\B\Omega^1_\dee$ behaves formally as a flip map, we rather rely on $\Omega^1_\dee$ being a $\dag$-bimodule, mirroring the operator adjoint on one-forms of a spectral triple. This gives an anti-linear relation between left and right module structures.

We relate the existence of Hermitian torsion-free connections on $\Omega^{1}_{\dee}$ to the two-projection problem in Hilbert $C^{*}$-modules \cite{MRtwoprojns} and obtain a necessary and sufficient condition for the existence of such connections and provide an explicit construction. 
Using noncommutative braidings and bimodule connections, \cite{BMBook,CO,DHLS,DMM,DM,GKOSS,LRZ}, we also obtain a sufficient condition for uniqueness. The braiding for which we have a bimodule connection is given by $2\Psi-1$ in many examples, but crucially there are examples \cite{MRPods} where the braiding $\sigma$ and $2\Psi-1$ are distinct, and $\sigma^2\neq 1$.

As usual, the space of connections on $\Omega^{1}_{\dee}$ is an affine space modelled on 
$\textnormal{Hom}_{\B}(\Omega^{1}_{\dee}, (\Omega^1_{\dee})^{\ox2})$. Using an analogue of raising and lowering indices arising from the $\dag$-structure, the latter bimodule can be identified with the module of 3-tensors $(\Omega^1_{\dee})^{\ox3}$. Thus, after a choice of Grassmann connection, given by any choice of frame for the one-forms, the construction of an Hermitian torsion-free connection boils down to identifying a specific element $A$ in $(\Omega^1_{\dee})^{\ox3}$.

Our construction of the required three-tensor reveals the  relation with the two-projection problem in Hilbert $C^{*}$-modules \cite{MRtwoprojns}, involving the projections $P=\Psi\otimes 1$ and $Q=1\otimes\Psi$. We are able to identify a necessary and sufficient condition, relating the projections $P,Q$ to the differential and $\dag$-structure, in order for Hermitian torsion-free connections to exist. Consequently, this condition is satisfied for compact Riemannian manifolds, and yields an entirely novel construction of the Levi-Civita connection in that case. 

Starting with a frame $(\omega_j)$ for the one-forms, we set $W=\sum_j\omega_j\ox\dee(\omega_j^\dag)$
and then the connection is determined by the Grassmann connection of $(\omega_j)$ and the connection form
\begin{equation}
A=-\sum_{n=0}^\infty (PQ)^n(W+PW^\dag)=-\sum_{n=0}^\infty(QP)^n(W^\dag+QW).
\label{eq:conn-form-intro}
\end{equation}

The formula \eqref{eq:conn-form-intro} determines $A$ only up to an element of ${\rm Im}(P)\cap{\rm Im}(Q)$. 
For so-called bimodule connections, see  \cite[Section 8.1]{BMBook}, which are compatible with the $\dag$-structure on the module of one-forms, the situation improves, and we obtain a sufficient condition for uniqueness. We show that relating the bimodule connection concept to the $\dag$-structure is imposing a reality condition, so that such bimodule connections map real one-forms to real two-forms. 

The bimodule connection condition is automatically satisfied on a compact Riemannian manifold, and accounts for a new uniqueness proof in that context. 
Moreover, we show that our entire setup applies to arbitrary $\theta$-deformations of toric Riemannian manifolds, endowing each of them with a unique Hermitian torsion-free bimodule connection. The case of free actions was covered in \cite{BGJ2}, and we remove the freeness condition.

Within Connes' noncommutative geometry there have been many instances of “Levi-Civita” connections in recent years [AW17a, AW17b, BGJ21, BGJ20, LNW94, LRZ09, R13], with varying hypotheses. The complemented junk of \cite{BGJ2,BGJ1} is a crucial ingredient for us.  Our methods apply to all of the examples in \cite{BGJ2,BGJ1}, and extend the applicability of the theory to non-centred bimodules. The  flip map on centred bimodules used in \cite{BGJ2,BGJ1}, gives a representation of the permutation group $S_3$ on three tensors. In our existence proofs we can deal with representations of $\Z\rtimes\Z_2$ on the three tensors, with  $(2P-1)$ and $(2Q-1)$ as generators, which do not factor through $S_3$.

In companion works we will present applications and examples. In \cite{MRCurve} we present definitions of Ricci and scalar curvature, along with a Weitzenb\"{o}ck formula, and apply them to $\theta$-deformations. In \cite{MRPods} we study connections, curvature and Weitzenb\"{o}ck for the Podle\'s sphere. For the latter example we use the discussion in Appendix \ref{sec:j-and-c} where we give an alternative method for finding Hermitian torsion-free connections in the presence of a frame consisting of closed forms. The relationships described between connections and junk two-tensors are used in \cite{MRPods} to determine the bimodule of junk-tensors.

\subsection*{Structure of the paper}

In Section \ref{sec:junk-diff} we review differential forms, 
$\dag$-modules and connections, introducing the new notion of conjugate pairs of connections. Section \ref{sec:junk} describes junk
and the construction of an exterior derivative $\d:\Omega^1_\dee(\B)\to \Lambda^2_\dee(\B):=(1-\Psi)T^2_\dee(\B)$. 

In Section \ref{subsec:RNDS} we recall what it means for a connection to be torsion-free. We then present a definition of torsion tensor, and show that
for an Hermitian connection the vanishing of the torsion tensor is equivalent to the connection being torsion-free.
We then give necessary and sufficient conditions for the existence of Hermitian torsion-free connections on differential one-forms. 
This section also recalls the two projection theory required for our constructions, and the relation to representations of the infinite dihedral group. We also characterise the junk submodule in terms of the image of connections on exact one-forms.

In Section \ref{sec:unique} we discuss uniqueness of Hermitian torsion-free connection by combining the notions of conjugate connections and bimodule connections, and note that these concepts have appeared before, see eg \cite[Section 8.1]{BMBook}. In Section \ref{sec:examples} we treat the case of $\theta$-deformations in detail, and note that this section includes the case of classical manifolds. Section \ref{sec:optimism} briefly discusses  issues of indefinite metrics and non-unitality/non-compactness. In Appendix \ref{sec:j-and-c} we outline the relationship between connections and junk two-tensors.

{\bf Acknowledgements} The authors thank the Erwin Schr\"{o}dinger 
Institute, 
Vienna, for hospitality and support during the 
production of this work. 
BM thanks the University of Wollongong for hospitality at 
an early stage of this project. AR thanks the Universiteit  Leiden for hospitality in 2022 and 2024. 
We thank Francesca Arici, Alan Carey, Giovanni Landi and Walter van Suijlekom for important 
discussions, and Alexander Flamant for careful reading of the manuscript, as well as the referee for their suggestions which have improved the manuscript. 

\vspace{-10pt}

\section{Noncommutative differential forms and connections}
\label{sec:junk-diff}

We introduce differential forms and their representations. 
Then we consider ``$\dag$-bimodules'' or ``$*$-bimodules'', and connections on them.
An excellent reference for this section is Landi's book \cite{Landi}, and more modern details can be found in \cite[Chapter 1]{BMBook}, \cite{Sch}.


The class of algebras we consider arises from examples of the noncommutative geometry of complex $*$-algebras. While most of our constructions are purely algebraic, there are two instances where we need some analysis. The first is in Section \ref{subsec:dag} to guarantee existence of frames for modules over our algebra, and the second is in Theorem \ref{thm:smo-intersect}. For this reason we recall a natural class of dense $*$-subalgebras of $C^{*}$-algebras.
See \cite[Section 3.1]{Blackadar} and \cite[Section 3.3.1]{BMBook}.
\begin{defn}
\label{def: local} 
Let $B$ be a $C^{*}$-algebra. A $*$-subalgebra $\mathcal{B}\subset B$ is \emph{local} if $\mathcal{B}$ is dense in $B$ and if for all $n\in\mathbb{N}$ the $*$-subalgebra $M_{n}(\mathcal{B})\subset M_{n}(B)$ is spectral invariant.
\end{defn}
The assumption that $\B\subset B$ is local holds for many algebras $\B$, such as smooth functions on a manifold. 
The notion of local algebra is not necessary for  our basic definitions, but later we will restrict to local algebras, and so work in this context throughout.

\subsection{Noncommutative differential forms}
\label{subsec:diff-junk}

We begin by recalling the notion of universal differential forms for a noncommutative associative algebra $\B$. Useful references are \cite{BHMS,Lo,Landi}.
\begin{defn}
\label{defn:forms}
The universal differential one-forms over the local algebra $\B$ are defined as the
kernel of the multiplication map
\[
\Omega^1_u(\B):=\ker(m:\,\B\ox\B\to \B)
\]
where $\ox$ is the algebraic tensor product. The universal
$k$-forms are then 
\[
\Omega^k_u(\B):=\Omega^{1}_{u}(\B)^{\otimes k}
=\Omega^1_u(\B)\ox_\B\Omega^1_u(\B)\ox_\B\cdots\ox_\B\Omega^1_u(\B)\quad
k\ \mbox{factors},
\]
and $\Omega_u^*(\B)=\oplus_{n\geq 0}\Omega^n_u(\B)$.
\end{defn}
Universal one-forms are linear combinations of forms $a\delta(b)$
with $a,b\in\B$, and
$
\delta(b):=1\ox b-b\ox 1.
$
Observe that $\delta(ab)=a\delta(b)+\delta(a)b$. As $\B$ is a $*$-algebra, the universal forms $\Omega^{*}_{u}(\B)$ form a graded $*$- algebra when equipped with the obvious concatenation product and the involution defined 
on universal one-forms by
$$
\Big(\sum_ja_j\delta(b_j)\Big)^\dag:=-\sum_j\delta(b_j^*)a_j^*
:=-\sum_j\delta(b_j^*a_j^*)+\sum_jb_j^*\delta(a_j^*),
$$
and extended to $\Omega^{*}_{u}(\B)$ via
$
(\omega_{1}\otimes \cdots \otimes \omega_{k})^{\dag}
:=\omega_{k}^{\dag}\otimes \cdots \otimes \omega_{1}^{\dag}.
$
In addition $\Omega^*_u(\B)$ is a differential graded algebra for the map
$$
\delta:\sum_ja^0_j\delta(a^1_j)\cdots\delta(a^k_j)
\mapsto\sum_j\delta(a^0_j)\delta(a^1_j)\cdots\delta(a_j^k).
$$
The universal feature of $\Omega^1_u(\B)$ is that whenever we have a $\B$-bimodule $M$, and a bimodule derivation ${\rm d}:\B\to M$, there exists a bimodule map $\pi:\Omega^1_u(\B)\to M$ such that
\[
\xymatrix{\B\ar[r]^{\rm d}\ar[d]_\delta& M\\
\Omega^1_u(\B)\ar[ur]_\pi &}
\]
commutes. The data $(M,{\rm d})$ is called a first order calculus for $\B$. We need a little more than this, encoding an `adjoint' structure on one-forms. 
\begin{defn}\cite[Definition 1.4]{BMBook}
A first order differential structure $(\Omega^1_\dee(\B),\dag)$ for the local algebra $\B$ is a first order calculus $(M,{\rm d})$ for $\B$ such that on $\Omega^1_\dee(\B):=\pi(\Omega^1_u(\B))$ there is a $\C$-anti-linear map $\dag:\Omega^1_{{\rm d}}(\B)\to \Omega^1_{{\rm d}}(\B)$
such that for all $a,b\in \B$ and $\omega\in \Omega^1_{{\rm d}}(\B)$ we 
have $\dag(a\omega b)=b^*\omega^\dag a^*$. 
\end{defn}
The  surjective homomorphism of first order calculi $\pi:\Omega^1_u(\B)\to \Omega^1_{{\rm d}}(\B)$ is then a homomorphism of $\dag$-modules when $\dee(a)=-\dee(a^*)$, which we assume unless stated otherwise.

Our definition is called a first order $*$-calculus in \cite[Definition 1.4]{BMBook}. Later we will wed this definition to a pre-Hilbert module structure to take full advantage of the $*$-structure. The main sources of first order differential structure we are interested in are  spectral triples and unbounded Kasparov modules, \cite{BRB,CPR11}.

\begin{example}
\label{eg:spec-trip}
Let $(\B,\H,\D)$ be a spectral triple for the local algebra $\B$. Then we can define the one-forms
\[
\Omega^{1}_{\D}(\B):=\Big\{ \sum_{finite} a_i^{0}[\D,a^{1}_{i}]: a^{j}_{i}\in\B\Big\}\subset\mathbb{B}(\H)
\]
and the derivation $\dee:\B\to\mathbb{B}(\H)$, $b\mapsto \dee(b):=[\D,b]$. The operator adjoint provides
$\dag:\Omega^{1}_{\D}(\B)\to\Omega^{1}_{\D}(\B)$. 

The assumption that $\B$ is local is not a restriction: if the $*$-algebra $\B$ is complete in the norm $\|b\|_{\D}:=\|b\|+\|[\D,b]\|$, then it is local, \cite[Proposition 3.12]{BC91}. In general, we may enlarge the dense $*$-subalgebra $\B$ by taking its closure in the norm $\|\cdot\|_{\D}$, which is still a dense $*$-algebra of the $C^{*}$-closure of $\B$, and so we may, without loss of generality, assume that $\B$ was local to begin with.


There are many examples of spectral triples throughout the literature.
Relevant examples for this work include Riemannian manifolds and their $\theta$-deformations \cite{CL},  the Podle\'{s} sphere \cite{DS}, quantum projective spaces \cite{DRS}. 
\end{example}

\begin{example} 
Let $(\B, X_{C}, S)$ be an (even) unbounded Kasparov module. Then, in exactly the same way as for spectral triples, we obtain
the one-forms
\[
\Omega^{1}_{S}(\B):=\Big\{ \sum_{finite} a_i^{0}[S,a^{1}_{i}]: a^{j}_{i}\in\B\Big\}\subset\End_C^*(X)
\]
and the derivation $\dee:\B\to\Omega^1_S(\B)$, $b\mapsto \dee(b):=[S,b]$. The operator adjoint provides the dagger structure again.

Compact group actions on algebras provide ``principal bundle'' type examples \cite{W}, especially circle \cite{CNNR} and torus actions \cite[Section 6]{CGRS2}, including graph and $k$-graph algebras, and more generally Cuntz-Pimsner, Cuntz-Pimsner-Nica algebras. 
\end{example}

\begin{example}
\label{eg:Minky}
Suppose that $\H$ is a ($\mathbb{Z}/2$-graded) Hilbert space, $\B\subset \mathbb{B}(\H)$ is a not-necessarily unital $*$-algebra and $\D:\Dom\D\to\H$ is an odd operator such that $\Dom(\D)\cap\Dom(\D^*)$ is dense in $\H$. Further suppose that for all $b\in\B$, $b\cdot\Dom\D\cap\Dom(\D^*)\subset \Dom(\D)\cap\Dom(\D^*)$ and $[\D,b]$ extends to a bounded operator. 
Then we get a first order differential structure just as in Example \ref{eg:spec-trip}, but  $\pi:\Omega^1_u(\B)\to \Omega_\D(\B)$ need not be a morphism of $\dag$-modules.

To see what the issue is, consider the Dirac operator of two dimensional Minkowski space,
\begin{equation}
\D=\begin{pmatrix} 0 & \partial_t+\partial_x\\ \partial_t-\partial_x & 0\end{pmatrix}.
\label{eq:Mink}
\end{equation}
Then for the two real-valued functions $t,x$ the operator adjoint gives 
$[\D,t]^\dag=[\D,t]$ and $[\D,x]^\dag= -[\D,x]$, and so we do not obtain a $\dag$-bimodule morphism.
We observe that Minkowski space is ``time-and-space oriented''. So there exists an adjointable map $\chi:\Omega^1_\D(\B)\to\Omega^1_\D(\B)$ such that 
$\chi^2=1$, $\chi(\omega^\dag)=(\chi\omega)^\dag$ for all one-forms $\omega$, and for all $b\in\B$ we have $\chi([\D,b]^\dag)=-[\D,b^*]$. Thus for the new ``adjoint'' $\dag\circ\chi=\chi\circ\dag$ we do obtain a $\dag$-bimodule morphism $\pi:\Omega^1_u(\B)\to (\Omega^1_\D(\B),\dag\circ\chi)$.
For examples arising from spectral triples we can take $\chi={\rm Id}$, and
the more general situation is motivated by \cite{DR}.
\end{example}

\begin{defn}
The tensor algebra of the first order differential structure $(\Omega^1_\dee(\B),\dag)$ is
\begin{equation}
T^{k}_{\dee}(\B):=\Omega^{1}_{\dee}(\B)^{\otimes_\B k}, \qquad T^{*}_{\dee}(\B):=\bigoplus_{n\geq 0}T^{k}_{\dee}(\B).
\label{eq:tensor-alg}
\end{equation}
\end{defn}
The involution $\dag$ on $\Omega^1_\dee(\B)$
extends to $T^{k}_{\dee}(\B)$ via
\begin{equation}
\label{eq: tensor_involution}
(\omega_{1}\otimes \cdots \otimes \omega_{k})^{\dag}:=\omega_{k}^\dag\otimes \cdots \otimes \omega_{1}^\dag,
\end{equation}
and this definition is compatible with the balancing of the tensor product over $\B$.
As with universal forms, the concatenation product makes $T^{*}_{\dee}(\B)$ into a $*$-algebra.
The map $\pi:\Omega^1_u(\B)\to\Omega^1_\dee(\B)$ extends to universal $k$-forms, defining a map
\[
\pi^{\otimes k}: \Omega^{k}_{u}(\B)\to T^{k}_{\dee}(\B),\quad \omega_{1}\otimes \cdots\otimes\omega_{k}\mapsto \pi(\omega_{1})\otimes \cdots \otimes \pi(\omega_{k}),
\]
which gives a surjective $*$-algebra homomorphism $\widehat{\pi}=\oplus_k\pi^{\ox k}:\Omega^{*}_{u}(\B)\to T^{*}_{\dee}(\B).$


\subsection{$\dag$-bimodules}
\label{subsec:dag}
We describe a class of bimodules over $*$-algebras, that are projective and carry an involution. Our definition overlaps significantly with \cite[Section 8]{DM}, \cite[p191, p266]{BMBook} and \cite{Sch}. Our motivations and results are more similar to \cite{DM,BMBook}. 

\begin{defn}
\label{defn:stern-bimod}
A \emph{$\dag$-bimodule} over the $*$-algebra $\B$ is a $\B$-bimodule $\mathcal{X}$ that is finitely generated projective as a right module and is equipped with a right $\B$-valued pre-$C^*$-inner product $\pairing{\cdot}{\cdot}_\B$ and an antilinear involution $\dag:\mathcal{X}\to\mathcal{X}$ such that $(axb)^{\dag}=b^{*}x^{\dag}a^{*}$.
\end{defn}
Recall that for a pre-$C^*$-inner product, we have $\pairing{xb}{yc}_\B=b^*\pairing{x}{y}_\B c$ and $\pairing{x}{x}_\B\geq 0$ in the $C^*$-completion of $\B$, along with the usual sesqui-linearity, see \cite[Lemma 2.16]{RW}.
\begin{example}
Later we will work with first order differential structures over a local algebra $\B$ which are finitely generated and projective inner product modules, and these provide examples of  $\dag$-bimodules. Together with the $\dag$-bimodules $T^k_\dee(\B)$ associated to such a first order differential structure $(\Omega^1_\dee(\B),\dag)$, these provide our main examples.
\end{example}

The inner product makes a $\dag$-bimodule $\mathcal{X}$ into a pre-$C^{*}$-module over $\B$, and thus $\mathcal{X}$ admits a closure $X$ which is a Hilbert $C^{*}$-module over the $C^{*}$-closure $B$ of $\B$, \cite[Lemma 2.16]{RW}. 
 
Recall \cite{FL02} that if $\mathcal{X}_\B$ is a right inner product $\B$-module, a \emph{frame}\footnote{In the signal analysis literature our frames would be called tight normalised frames.} is a (countable) collection of elements $v:=(x_j)\subset \mathcal{X}_\B$ such that for all $x\in \mathcal{X}_\B$ we have
\begin{equation}
\label{eq:frame}
x=\sum_{j}x_j\pairing{x_j}{x}_\B.
\end{equation}
Frames always exist on countably generated $C^*$-modules, by the Kasparov stabilisation theorem, \cite[Theorem 5.49]{RW}.
We will make extensive use of frames throughout, and for finitely generated $C^*$-modules we can always choose a finite frame. Our frames are global frames (see Example \ref{eg:mfld-frame-gee}), and so the number of elements in a frame is typically not related to any kind of dimension, though we discuss a related concept later. The following results giving existence of frames for ``smooth'' modules are well-known and have appeared in various forms in the literature, eg \cite{BMBook,LRV}.
\begin{prop}
\label{prop:SHFC-End}
Let $B$ be a unital $C^{*}$-algebra and $\B\subset B$ a unital local subalgebra. Suppose that $\mathcal{X}$ is a finitely generated projective $\B$-inner product module with $C^{*}$-closure $X$. The continuous extension of $T\in \End^*_\B(\X)$ to $\overline{T}\in \End^*_B(X)$ defines an inclusion $\End^*_\B(\X)\hookrightarrow \End^*_B(X)$. Then $\End^{*}_{\B}(\mathcal{X})$ is local in $\End^{*}_{B}(X)$.
\end{prop}
\begin{proof}
Since $M_{n}(\End^{*}_{\B}(\mathcal{X}))\simeq \End^{*}_{\B}(\mathcal{X}^{n})$ and $\mathcal{X}^{n}$ is again a finitely generated projective inner product module, it suffices to show that $\End_{\B}^*(\mathcal{X})$ is spectral invariant in $\End^{*}_{B}(X)$. There exists $n\in\mathbb{N}$ and a projection $p\in M_{n}(\B)$ such that $\mathcal{X}\simeq p\B^{n}$ as inner product modules. Therefore  $\End^{*}_{\B}(\mathcal{X})\simeq pM_{n}(\B)p$ is spectral invariant in $pM_{n}(B)p\simeq \End^{*}_{B}(X)$.
\end{proof}
The next Corollary is where we explictly use locality of $\B$ to find frames for $\B$-modules.
\begin{corl} Let $B$ be a unital $C^{*}$-algebra and $\B\subset B$ a unital local subalgebra and $\mathcal{X}$ a  $\B$-inner product module. Then $\X$ is finitely generated projective if and only if $\X$ admits a finite frame.
\end{corl}
\begin{proof} Let $(w_i)$ be a finite frame for the $C^{*}$-closure $X$. By choosing $y_{i}\in\mathcal{X}$ sufficiently close to $w_{i}$ the operator 
\[
g:x\mapsto \sum y_{i}\pairing{y_{i}}{x},
\]
is close to the identity operator in $\End_{\B}^{*}(\X)$, and thus $g$ is invertible in $\End^{*}_{B}(X)$. Locality then implies that $g^{-1/2}\in\End^{*}_{\B}(\X)$, \cite{LBS}, and so $x_{i}:=g^{-1/2}y_{i}\in \X$ is a frame for $\X$.
\end{proof}
\begin{lemma}
\label{lem: bi-Hilbertian} 
Let $\mathcal{X}$ be a $\dag$-bimodule. Then 
\[
{}_{\B}\pairing{x}{y}:=\pairing{x^{\dag}}{y^{\dag}}_{\B},
\]
defines a left inner product on $\mathcal{X}$. If $x_{j}$ is a right frame for $\mathcal{X}$ then $x_{j}^{\dag}$ is a left frame for $\X$. Consequently $\mathcal{X}$ is a bi-Hilbertian bimodule of finite index in the sense of \cite{KajPinWat}. 
\end{lemma}
\begin{proof}
It is straightforward to check that $_{\B}\pairing{\cdot}{\cdot}$ defines a left inner product. For a right frame $x_{j}$, consider
\begin{align*}
\sum_{j}\,_{\B}\pairing{x}{x_{j}^{\dag}} x_{j}^{\dag}=\sum_{j}\pairing{x^{\dag}}{x_{j}}_{\B} x_{j}^{\dag}=
\big(\sum_j x_{j}\pairing{x_{j}}{x^{\dag}}_{\B}\big)^{\dag}
=x^{\dag\dag}=x,
\end{align*}
proving that $(x_{j}^{\dag})$ is a left frame. 
\end{proof}
As the norms on $\X$ induced by the left and right inner products are equivalent, $\dag$ extends to a map $\dag:X\to X$ on the $C^{*}$-module completions. Therefore the completion $X$ is a $\dag$-bimodule over the $C^{*}$-closure $B$.
\begin{defn}
\label{defn:dag-endos}
For a $\dag$-bimodule $\mathcal{X}$, we denote by $\REnd(\X)$ the $*$-algebra of  operators adjointable with respect to the right inner product, and by $\LEnd(\X)$ the $*$-algebra of operators adjointable with respect to the left inner product. We denote by $T\mapsto T^{*}$ the right adjoint and by $T\mapsto {}^*T$ the left adjoint, and refer to right- and left-adjointable operators respectively.
\end{defn}

\begin{lemma}
Let $\mathcal{X}$ be a $\dag$-bimodule. The map
\[
\REnd(\X)\to \LEnd(\X),\quad  T\mapsto\dag\circ T^* \circ\dag 
\] 
is a $*$-algebra anti-isomorphism. 
\end{lemma}
\begin{proof} 
Let $T\in\REnd(\X)$.
The calculation $\dag\circ T^*\circ\dag\circ\dag\circ S^*\circ\dag=\dag\circ(ST)^*\circ\dag$ gives the anti-homomorphism property, while
\begin{align*} 
_{\B}\pairing{(Tx^{\dag})^{\dag}}{y}
=\pairing{Tx^{\dag}}{y^{\dag}}_{\B}
=\pairing{x^{\dag}}{T^{*}y^{\dag}}_{\B}
=\,_{\B}\pairing{x}{(T^{*}y^{\dag})^{\dag}},
\end{align*}
shows that the map $y\mapsto (T^{*}y^{\dag})^{\dag}$, written $\dag\circ T^{*}\circ \dag$ is the left-adjoint to $\dag\circ T\circ\dag$.
\end{proof}
\begin{corl}
\label{cor:left-right-adjoint}
If $T:\X\rightarrow \X$ is a $\C$-linear operator such that $\dag\circ T= T\circ\dag$, then $T$ is right-adjointable if and only if $T$ is left-adjointable. In this case the left and right adjoints satisfy $^{*}T=\dag\circ T^{*}\circ\dag$. In particular $T^{*}=T$ if and only if $^{*}T=T$.
\end{corl}
In the context of $\dag$-bimodules we can make abstract sense of the process of raising and lowering indices akin to the classical formalism.

\begin{lemma}
\label{lem:useful-iso}
Let 
$\X$ a $\dag$-bimodule over $\B$. The map
\begin{align*}\X\otimes_{\B}\mathcal{X}\to \REnd (\mathcal{X}),\quad x\otimes y\mapsto |x\rangle\langle y^{\dag}|,
\end{align*}
is a $\B$-bimodule isomorphism. 
More generally, if $\mathcal{Y}$ is another $\dag$-bimodule, we let $\overrightarrow{\textnormal{Hom}}^{*}(\X,\Y)$ be the right adjointable maps $\X\to\Y$. Then there is a left $\overrightarrow{\End}^*(\Y)$-module isomorphism
\[
\overrightarrow{\alpha}:\Y\ox_{\B} \X\to \overrightarrow{\textnormal{Hom}}^{*}(\X,\Y)
\qquad \overrightarrow{\alpha}(y\ox x)(w)=y\, \pairing{x^\dag}{w}_\B,\qquad x,w\in\X,\,y\in\Y
\]
and analogously a right $\overleftarrow{\End}^*(\Y)$-module isomorphism
\[
\overleftarrow{\alpha}:\X\ox_{\B} \Y\to \overleftarrow{\textnormal{Hom}}^{*}(\X, \Y)
\qquad \overleftarrow{\alpha}(x\ox y)(w)={}_\B\pairing{w}{x^\dag}\,y,\qquad x,w\in\X,\,y\in\Y.
\]
If $(x_j)$ is a (right) frame for $\X$ then we can express the inverse maps as
\[
\overrightarrow{\alpha}^{-1}: \overrightarrow{\textnormal{Hom}}^{*}(\X,\Y)\to \Y\ox_{\B} \X
\qquad \overrightarrow{\alpha}^{-1}(T)=\sum_jTx_j\ox x_j^\dag
\]
and
\[
\overleftarrow{\alpha}^{-1}:\overleftarrow{\textnormal{Hom}}^{*}(\X, \Y)\to \X\ox_{\B} \Y\qquad \overleftarrow{\alpha}^{-1}(T)=\sum_jx_j\ox T(x_j^\dag).
\]
Furthermore the $\dag$s on $\X$ and $\Y$ induce a well-defined anti-linear map $\dag: \X\ox_\B \Y\to \Y\ox_\B\X$  given  by
\[
(x\ox y)^\dag=y^\dag\ox x^\dag,\ \ \ x\in\X,\ y\in\Y.
\]
\end{lemma}
\begin{proof}
All statements are straightforward verifications. 
 \end{proof}

\begin{rmk}
\label{rmk:endo-val-forms}
The  isomorphisms in the Lemma allow us to identify $\eta\ox\rho^\dag\in T^2_\dee$ with the ket-bra $\ketbra{\eta}{\rho}$ (using the right inner product). Similarly a three-tensor
$\eta\ox\omega\ox\rho^\dag\in T^3_\dee$ can be identified with a one-form-valued endomorphism, $\ket{\eta}\omega\bra{\rho}$.
\end{rmk}

  
\subsection{Hermitian connections on $\dag$-bimodules}

Here we explore connections on $\dag$-modules, and introduce the notion of conjugate connection. References for connections on  modules include
\cite{BMBook,Landi,LRV,MR,MRS}.

\begin{defn} 
Let $(\Omega^1_\dee(\B),\dag)$ be a first order differential structure and $\X$ a $\dag$-bimodule over $\B$. 
A \emph{right connection} on $\X$ is a $\C$-linear map $\overrightarrow{\nabla}:\mathcal{X}\to \X\otimes_{\B}\Omega^{1}_{\dee}(\B)$ satisfying 
$\overrightarrow{\nabla}(xb)=\overrightarrow{\nabla}(x)b+x\otimes\dee(b)$. 
Similarly, a \emph{left connection} is a map $\overleftarrow{\nabla}:\X\to \Omega^{1}_{\dee}(\B)\otimes_{\B} \X$ satisfying $\overleftarrow{\nabla}(bx)=b\overleftarrow{\nabla}(x)+\dee(b)\otimes x$.
\end{defn}
\begin{rmk}
A $C^\infty$-type assumption on both $\X$ and $\B$ is evident here: we 
suppose that the connection takes values in $\X\ox_\B\Omega^1_\dee(\B)$ rather than tensor products of lower regularity modules. 
\end{rmk}

It is well-known that the difference of two right connections $\nablar,\nablar'$ is a right $\B$-module map
\[
\nablar-\nablar': \X\to \X\otimes_{\B}\Omega^{1}_{\dee}.
\]
Similarly the difference of two left connections $\nablal,\nablal'$ is a left $\B$-module map
\[
\nablal-\nablal': \X\to \Omega^{1}_{\dee}\otimes_{\B}\X.
\]
By choosing one connection, all other connections are given by  such module homomorphisms. 

{\bf Notation} We denote by $\B^{N}$ the free right $\B$-module of rank $N$ with its standard right inner product and by $\,^{N}\B:=(\B^{N})^{t}$ the free left $\B$-module of rank $N$ with its standard left inner product.
The following lemma is an adaptation of well-known facts to the $\dag$-bimodule setting.

\begin{lemma} 
\label{lem:vee}
Let $\X$ be a $\dag$-bimodule over the $*$-algebra $\B$ and $v=(x_{j})\subset \X$ a finite right frame for $\X$. The maps
\begin{equation}
\label{eq: frame-isometry}
\overrightarrow{v}:\X_\B\to \B^{N},\quad x\mapsto (\pairing{x_j}{x}_\B)_{j=1}^{N},\quad \overleftarrow{v}:\,_{\B}\X\to \,^{N}\B,\quad x\mapsto ({}_\B(\pairing{x}{x_{j}^{\dag}})_{j=1}^{N})^{t}
\end{equation}
are injective right resp. left module maps that preserve the right resp. left inner products on $\X$. The image of $\overrightarrow{v}$ is  $\overrightarrow{p}\B^{N}$ where $\overrightarrow{p}=(\pairing{x_i}{x_j}_\B)_{ij}\in M_{N}(\B)$ and the image of $\overleftarrow{v}$ is 
$(\B^{N})^{t}\overleftarrow{p}$ where $\overleftarrow{p}=({}_\B\pairing{x_i^{\dag}}{x_j^{\dag}})_{ij}\in M_{N}(\B)$. Thus $\X$ is finitely generated and projective as both a right and a left inner product $\B$-module. 
\end{lemma}

The next lemma recalls well-known facts about connections and records the additional features for $\dag$-bimodules. Landi's book \cite{Landi} 
provides a good guide to previous literature.

\begin{lemma} 
\label{lem:A-dag}
Let $(\Omega^1_\dee(\B),\dag)$ be a first order differential structure and $\X$ a $\dag$-bimodule over $\B$.
There are pairings
\begin{align}
\label{eq: pairings}
\X\times \X\otimes_{\B}\Omega^{1}_{\dee}(\B)&\to\Omega^{1}_{\dee}(\B),&\quad   \X\otimes_{\B}\Omega^{1}_{\dee}(\B)\times \X&\to\Omega^{1}_{\dee}(\B)\\
\pairing{ y}{ x\otimes\omega}_{\Omega_\dee^1(\B)} 
&:= \pairing{ y}{x}_\B \omega, 
&\quad \pairing{ x\otimes\omega}{y}_{\Omega_\dee^1(\B)}&:= \omega^{*}\pairing{ x}{y}_\B.\nonumber
\end{align}
A right $\B$-linear map $\mathbb{A}: \X_\B\to \X\otimes_{\B}\Omega^{1}_{\dee}$, corresponds to a unique element $A\in  \X\otimes_{\B}\Omega^{1}_{\dee}\otimes_{\B}\X$ such that $\mathbb{A}=\alphar(A)$ where (using a frame right frame $(x_j)$ for $\X$) we have
\[
A:=\sum_{i}\mathbb{A}x_{i} \otimes  x_{i}^{\dag}
=\sum_{i,j} x_{j}\otimes \pairing{x_{j}}{\mathbb{A}x_{i}}_{\Omega_\dee^1(\B)} x_{i}^{\dag}=\sum x_{i}\otimes(\mathbb{A}^{\dag}x_{i})^{\dag}.
\]
Moreover, such a map $\mathbb{A}=\alphar(A)$ satisfies 
$\pairing{ x}{\alphar(A)y}_{\Omega_\dee^1(\B)} =\pairing{ \alphar(A^{\dag})x }{y}_{\Omega_\dee^1(\B)}$. 
\end{lemma}
\begin{proof} 
Fixing a right frame $(x_j)$ we can compute for $x\in\X$ that
\[
\mathbb{A}x
=\mathbb{A}\left(\sum_{j} x_{j}\pairing{ x_{j}}{x}\right)
=\sum \mathbb{A}(x_{j})\pairing{ x_{j}}{x}
= \sum_{i,j} |x_{i}\rangle\otimes \pairing{ x_{i}}{ \mathbb{A}(x_{j})}\pairing {x_{j}}{x}.
\]
Thus $\mathbb{A}=\alphar(A)$ with 
$A:= \sum_{i,j} x_{j}\otimes \pairing{x_{j}}{\mathbb{A}x_{i}}\otimes x_{i}^{\dag}\in  \X\otimes_{\B}\Omega^{1}_{\dee}(\B)\otimes_{\B}\X$.
\end{proof}

\begin{defn}
\label{defn:grassy}
Let $\X$ be a $\dag$-bimodule over the $*$-algebra $\B$.
Given a frame $v=(x_j)\subset \X$ we obtain maps $\overrightarrow{v}:\X\to \B^{N}$ and $\overleftarrow{v}:\X\to \,^{N}\B$ and  
left and right {\em Grassmann
connections }
\begin{align}
\nonumber
\nablar^{v}&:\X\to \X\ox_B\Omega^1_\dee,\qquad
\nablar^{v}(x):=\sum_j x_j\ox\dee(\pairing{x_j}{x}_\B),\\
\nablal^{v}&:\X\to \Omega^1_\dee(\B)\ox_{B}\X,\qquad
\nablal^{v}(x):=\sum_j \dee({}_{\B}\pairing{x}{x_{j}^{\dag}})\ox x_j^{\dag}.
\label{eq:grassy-knoll}
\end{align}
Grassmann connections evidently depend on the choice of frame.
\end{defn}
 
\begin{defn} 
Let $(\Omega^1_\dee(\B),\dag)$ be a first order differential structure and $\X$ a $\dag$-bimodule over $\B$.
We say that 
a right connection $\nablar$ on a module $\X$ is Hermitian with respect to the right inner product $\pairing{\cdot}{\cdot}_{\B}$  if
for all $x,\,y\in \X$ we have
\begin{equation}
\dee(\pairing{x}{y}_\B)=\pairing{x}{\nablar y}_{\Omega_\dee^1(\B)}-\pairing{\nablar x}{y}_{\Omega_\dee^1(\B)}.
\label{eq:compatible}
\end{equation}
If we work with a left connection $\nablal$, the Hermitian requirement becomes
\begin{equation}
\dee({}_\B\pairing{x}{y})={}_{\Omega_\dee^1(\B)}\pairing{\nablal x}{y}-{}_{\Omega_\dee^1(\B)}\pairing{x}{\nablal y}.
\label{eq:compatible-left}
\end{equation}
\end{defn}

\begin{prop}
\label{lem:compatible}
Let $(\Omega^1_\dee(\B),\dag)$ be a first order differential structure and $\X$ a $\dag$-bimodule.
All right Grassmann connections $\nablar^{v}$ on a module $\X$ are Hermitian.
A right connection $\nablar$ is 
Hermitian with respect to an inner product $\pairing{\cdot}{\cdot}_\B$
if and only if for all right frames $v=(x_j)$ the decomposition $\nablar=\nablar^{v}+\alphar(A_{v})$ satisfies $A_{v}=A^{\dag}_{v}$ if and only if 
\begin{equation}
\sum_j\nablar(x_j) \ox x_j^\dag=\sum_j x_j\ox\nablar(x_j)^\dag\in \X\ox\Omega^1_\dee(\B)\ox\X.
\label{eq:nabla-herm}
\end{equation}
\end{prop}
\begin{proof} 
Let $v=(x_{j})$ be a finite frame for the right module $\X$. Then
\begin{align*}
\dee(\pairing{x}{y}_\B)
&=\sum_{j}\dee(\pairing{x}{x_{j}}_\B\pairing{x_{j}}{y}_\B)
=\sum_{j}\pairing{x}{x_{j}}_\B\dee(\pairing{x_{j}}{y}_\B)
+\dee(\pairing{x}{x_{j}}_\B)\pairing{x_{j}}{y}_\B\\
&=\pairing{ x}{\nablar^{v}(y)}_{\Omega_\dee^1(\B)}-\pairing{\nablar^{v}(x)}{y}_{\Omega_\dee^1(\B)},
\end{align*}
so $\nablar^{v}$ is Hermitian. Hence for $\nablar=\nablar^{v}+\alphar(A_{v})$ the characterisation of Hermitian connections follows from Lemma \ref{lem:A-dag}. Finally \eqref{eq:nabla-herm} follows by  applying the definitions of connection and Hermitian connection \eqref{eq:compatible} and the injectivity of $\alphar$ to see that for all $x\in\X$ 
\begin{align*}
\sum_j&\nablar(x_j) \pairing{ x_j}{x}_\B=\nablar(x)-\sum_jx_j\ox\dee(\pairing{x_j}{x}_\B)\\
&=\nablar(x)+\sum_jx_j\ox\pairing{\nablar(x_j)}{x}_{\Omega_\dee^1(\B)}-x_j\ox\pairing{x_j}{\nablar(x)}_{\Omega_\dee^1(\B)}\\
&=\sum_j x_j\ox\pairing{\nablar(x_j)}{x}_{\Omega_\dee^1(\B)}.\qedhere
\end{align*}
\end{proof}

Using the $\dag$-module structure on $\X$, we can associate to any right connection a left connection and vice versa. Later we will use this method to make contact with earlier work \cite{BMBook,DHLS,DMM} on bimodule connections. The proof of the following statement is a simple check using Proposition \ref{lem:compatible}.
\begin{corl}
\label{lem:left-right}
For any right connection $\nablar:\X\to \X\ox_\B \Omega^1_\dee(\B)$, we obtain a left connection $\nablal:\X\to  \Omega^1_\dee(\B)\ox_\B\X$ by defining
$$
\nablal(x):=-(\nablar(x^\dag))^\dag, \ \ x\in\X.
$$
Then $\nablal$ is (left) Hermitian if and only if $\nablar$ is (right) Hermitian. If $v=(x_{j})$ is a right frame, then the left- and right Grassmann connections $\nablar^{v}$ and $\nablal^{v}$ satisfy $\nablal^{v}=-\dag\circ\nablar^{v}\circ\dag$ and if $\nablar=\nablar^{v}+\alphar(A)$ then $\nablal=\nablal^{v}-\alphal(A^{\dag})$. 
\end{corl}
\begin{defn}
\label{defn:conj-conn}
Let $(\Omega^1_\dee(\B),\dag)$ be a first order differential structure, and $\nablar$ a right connection  on a $\dag$-bimodule $\X$ over $\B$. The {\em conjugate connection} is the left connection 
$\nablal$ defined by 
$\nablal:=-\dag\circ\nablar\circ\dag$.
\end{defn}

We observe that ``conjugate connection'' is a distinct notion from ``bimodule connection'', which will appear later, but see \cite[Section 8.1]{BMBook}.
For later use in subsection \ref{subsec: existence} we record the following expression for the difference of Grassmann connections.
\begin{corl}
\label{cor: change-of-frame} 
Let $(\Omega^1_\dee(\B),\dag)$ be a first order differential structure, $\mathcal{X}$ a $\dag$-bimodule. Let  $v=(x_{i})_{i=1}^{N}$ and $w=(y_j)_{j=1}^{M}$ be right frames for $\mathcal{X}$, and $\nablar^{v},\nablar^{w}$
the associated Grassmann connections. The difference 
\[
B_{vw}:=\alphar^{-1}(\nabla^{v}-\nabla^{w})
=\sum_{i,j}x_{i}\otimes \dee(\pairing{ x_{i}}{y_{j}}_\B)\otimes y_{j}^{\dag}\in \mathcal{X}\otimes_{\B}\Omega^{1}_{\dee}\otimes_{\B}\mathcal{X},
\]
is $\dag$-invariant, so $B_{vw}=B_{vw}^{\dag}$.
\end{corl}
\begin{proof} 
We have
\begin{align*}
(\nabla^{v}-\nabla^{w})(x)
&=\sum_{i} x_{i}\otimes \dee(\pairing{x_{i}}{x}_\B)
-\sum_{j}y_{j}\otimes \dee(\pairing{y_{j}}{x}_\B)\\
&=\sum_{i}x_{i}\otimes \dee(\pairing{x_{i}}{x}_\B)
-\sum_{i,j} x_{i}\otimes \pairing{x_{i}}{y_{j}}_\B \dee(\pairing{ y_{j}}{x}_\B)\\
&=\sum_{i,j} x_{i}\otimes \dee(\pairing{x_{i}}{y_{j}}_\B)\pairing{ y_{j}}{x}_\B,
\end{align*}
so that $\alphar^{-1}(\nabla^{v}-\nabla^{w})
=\sum_{i,j}x_{i}\ox \dee(\pairing{x_{i}}{y_{j}}_\B)\ox y_{j}^{\dag}$. Since both connections $\nabla^{v},\nabla^{w}$ are Hermitian, the element $\alpha^{-1}(\nabla^{v}-\nabla^{w})$ is $\dag$-invariant by Lemma \ref{lem:A-dag}.
\end{proof}

\subsection{Quantum metrics for $\dag$-bimodules}
\label{subsec:the-q-word}
Despite starting from, and relying on, Hilbert module inner products, the bimodule inner products and quantum metrics of 
\cite{BGJ2,BGJ1,BMBook} arise naturally in our setting.

\begin{defn}
\label{defn:Geeeeeeee-ex}
Let $\X$ be a $\dag$-$\B$-bimodule.
Under the isomorphism 
\[
\alphar: \X\ox_\B\X\xrightarrow{\sim}\overrightarrow{\End}^{*}_{\B}(\X),\quad  x\otimes y^{\dag}\mapsto \ket{x} \bra{y},
\]
we let $G_\X\in \X\ox_\B\X$ be the unique element with 
$\alphar(G_\X)= \textnormal{Id}_\X,$ the identity operator on $\X$. 
\end{defn}
Note that for any right frame $(x_j)$ for $\X$
\[
\textnormal{Id}=\sum_{j} |x_{j}\rangle\langle x_{j}|,
\quad
\mbox{so that }
\quad
G_\X=\sum_{j} x_{j}\otimes x_{j}^{\dag},
\]
for any right frame as well, and this is a frame independent expression for $G_\X$. Note that $\alphal(G_\X)=\textnormal{Id}_\X$ as well, since for any right frame $x_{j}$, $x_{j}^{\dag}$ is a left frame. Thus the definition of $G_\X$ is also independent of using either the left or the right module structure. For $x,y\in\X$ we have the useful formula
\begin{equation}
\pairing{G_\X}{x\ox y}_\B=\pairing{x^\dag}{y}_\B,
\label{eq:real-ip}
\end{equation}
and one can see the appearance of the Riemannian inner product of \cite{BGJ2,BGJ1}, which in turn satisfies the definition of ``bimodule inner product'' in \cite[Section 1.3]{BMBook}.
\begin{rmk}[Centrality of $G_{\X}$] We point out that for all $b\in\B$ we have $G_\X b=bG_\X$. Thus, $G_{\X}$ is always a central element of the bimodule $\X\otimes_{\B}\X$. In fact $G_\X$ is a \emph{quantum metric} on $\X$ in the sense of Beggs-Majid \cite[Definition 1.15]{BMBook}, where the right hand side of Equation \eqref{eq:real-ip} provides the ``inverse'' as the frame relation shows. Such quantum metrics are always central by \cite[Lemma 1.16]{BMBook}.  This centrality has also been observed in the context of bi-Hilbertian bimodules \cite[Theorem 2.22]{KajPinWat}.
\end{rmk}
\begin{example}
\label{eg:mfld-frame-gee}
Let $(M,g)$ be a compact oriented Riemannian manifold. 
Then given a finite covering by coordinate charts $(U_\alpha,x_\alpha)$ and a subordinate partition of unity $(\varphi_\alpha)$, we have a natural module frame for the one-forms $\Omega^1_{\slashed{D}}(C^\infty(M))$ with inner product coming from the Riemannian metric. The frame is given by
$\omega_{j\alpha}=\sum_\mu\sqrt{\varphi_\alpha}B^j_\mu dx_\alpha^\mu$ where $e^j_\alpha:=\sum_\mu B^j_\mu dx^\mu_\alpha$ is a local orthonormal frame for the inner product.
One can check directly that $G=\sum_{\alpha,\mu,\nu}\varphi_\alpha g_{\mu\nu}dx^\mu_\alpha\ox dx^\nu_\alpha$ is the line element: see Lemma \ref{lem:grassy-bit}. 
\end{example}
\begin{prop}
\label{prop:left-right-conn}
Let $\X$ be a $\dag$-bimodule and $(\Omega^1_\dee,\dag)$ a first order differential structure. Let
 $\nablar:\X\to \X\ox \Omega^{1}_{\dee}(\B)$ and $\nablal:\X\to \Omega^{1}_{\dee}(\B)\ox\X$ be right and left connections respectively. Then the operator
\[
\nablar\otimes 1+1\otimes\nablal:\X\ox_\B\X\to \X\ox \Omega^{1}_{\dee}(\B)\ox\X, 
\]
is well-defined. Moreover, if $\nablal(x):=-\nablar(x^{\dag})^{\dagger}$ is the conjugate left connection to $\nablar$, then $\nablar$ and $\nablal$ are Hermitian if and only if
\[
(\nablar\otimes 1+1\otimes\nablal)(G_\X)=0.
\]
\end{prop}
\begin{proof}
Well-definedness follows easily: with $x,y\in \X$ and $b\in\B$ we have
\begin{align*}
(\nablar\otimes 1+1\otimes\nablal)(x b\otimes y)
&=\nablar(x b)\otimes y+x b\otimes \nablal(y)\\
&=\nablar(x )\otimes by+x\otimes \dee(b)\otimes y+x b \otimes \nablal(y)\\
&=(\nablar\otimes 1+1\otimes\nablal)(x \otimes b y).
\end{align*}
Now assume that $\nablar$ is Hermitian, so that, by Proposition \ref{lem:compatible}, for any frame $(x_{j})$ we have
\[
\sum_{j}\nablar(x_{j})\otimes x_{j}^{\dag}
=\sum_{j}x_{j}\otimes  \nablar(x_{j})^{\dagger}, 
\]
so that by the definition of the conjugate connection we have
\begin{equation*}
(\nablar\otimes 1+1\otimes\nablal)(G)
=\sum_{j}\nablar(x_{j})\otimes x_{j}^{\dag}-x_{j}\otimes\nablar(x_{j})^{\dagger}=0.      \qedhere
\end{equation*}
\end{proof}
\begin{rmk}
A similar result for bimodule connections is stated in \cite[Lemma 8.4]{BMBook}. 
\end{rmk}

Our main interest in the discussion of this section is for a first order differential structure $\X=\Omega^1_\dee(\B)$, but we need to add the 
requisite inner product data.


\section{Junk tensors and the exterior derivative}
\label{sec:junk}

The aim of this section is to obtain a well-defined differential $\d:\Omega^1_\dee(\B)\to T^2_\dee(\B)$ of a first order differential structure. This requires an additional structural assumption.


\subsection{The bimodule of junk tensors}
The maps $\widehat{\pi}:\Omega^{*}_{u}(\B)\to T^*_\dee$ and $\delta:\Omega^{k}_{u}(\B)\to \Omega^{k+1}_{u}(\B)$ are typically not compatible  in the sense that $\delta$ need not map $\ker \widehat{\pi}$ to itself. Thus in general,  $T^{*}_{\dee}(\B)$ can not be made into a differential algebra.
The issue to address is that
there are universal
forms $\omega\in \Omega^n_u(\B)$ for which $\widehat{\pi}(\omega)=0$
but $\widehat{\pi}(\delta(\omega))\neq0$. The latter are known as {\em junk tensors},  \cite[Chapter VI]{BRB}. 

\begin{defn}
\label{defn:junk}
The bimodule of degree $n$  
junk tensors $JT^n_\dee\subset T^n_\dee$ is the sub-$\B$-bimodule  defined by
\begin{align*}
JT^n_\dee&:=\big\{T\in T^n_\dee(\B):\,T=\widehat{\pi}(\delta(\omega)),\ \omega\in \Omega_u^{n-1}(\B)\cap\ker \widehat{\pi}\big\}\subset T^*_\dee(\B).
\end{align*}
\end{defn}
That $JT^n_\dee$ is a bimodule follows from $\widehat{\pi}$ being a bimodule map,
the (graded) Liebniz rule for $\delta$ and a short computation as follows.
For $\omega\in \ker(\widehat{\pi})$ with $T=\widehat{\pi}(\delta(\omega))$ and $a\in\B$
$$
aT=a\widehat{\pi}(\delta(\omega))=\widehat{\pi}(a\delta(\omega))
=\widehat{\pi}(\delta(a\omega))-\widehat{\pi}(\delta(a)\omega)=\widehat{\pi}(\delta(a\omega)).
$$
The quotient of $T^{*}_{\dee}$ by the bimodule of junk tensors yields a differential graded algebra, and junk tensors form the smallest bimodule with this property. In the algebraic literature, this quotient is known as the \emph{maximal prolongation} of the first order calculus $\dee:\B\to\Omega^{1}_{\dee}$. The following Lemma states this well-known result.
\begin{lemma}[{\rm \cite[Chapter VI, Section 1]{BRB} and \cite[Lemma 1.32]{BMBook}}] There is a well-defined differential
\[
\dee_C:T^{n}_{\dee}/JT^{n}_{\dee}\to T^{n+1}_{\dee}/JT^{n+1}_{\dee}
\]
determined by $\dee_C(\widehat{\pi}(\omega)+JT^{n}_{\dee})=\pi(\delta(\omega))+JT^{n+1}_{\dee}$
\end{lemma}
In the context of spectral triples, the differential $\widehat{\dee}_C$ is a ``lift'' of Connes' differential, as we describe in  Section \ref{subsec:has-to-be-done}.
\begin{example}
\label{eg:classical-junk}
Let $(M,g)$ be a compact oriented Riemannian manifold and $\Omega^1(M)\ox\C$ the complexified one-forms with $\dee:C^\infty(M)\to \Omega^1(M)\ox\C$ the exterior derivative.
Any universal form $\sum_i f_i\delta(h_i)-\delta(h_i)f_i=\sum_if_i\ox h_i-h_i\ox f_i$
is in the kernel of $\pi$ because $\Omega^1(M)\ox\C$ is a central bimodule and
\[
\pi(\sum_i f_i\delta(h_i)-\delta(h_i)f_i)=\sum_if_i\dee(h_i)-\dee(h_i)f_i=0.
\]
Applying the universal differential to $\sum_i f_i\delta(h_i)-\delta(h_i)f_i$ yields 
$
\sum_i \delta(f_i)\ox\delta(h_i)+\delta(h_i)\ox\delta(f_i)
$
which is represented by $\widehat{\pi}$ as 
\[
\sum_i \dee(f_i)\ox\dee(h_i)+\dee(h_i)\ox\dee(f_i).
\]
This two-tensor is not zero in general, and indeed every symmetric two-tensor arises as junk in $JT^{2}_\dee$ (see Lemma \ref{lem: theta-junk-exterior}).
\end{example}


\subsection{Representation of the differential on one-forms}
We make the following assumption, as in \cite{BGJ2,BGJ1}, which allows us to define a differential from one-forms to  two-tensors.
The differential will allow us to define torsion and curvature.
\begin{defn}
\label{ass:standing1}
A \emph{second order differential structure} $(\Omega^1_\dee(\B),\dag,\Psi)$ is a first order differential structure $(\Omega^1_\dee(\B),\dag)$
such that there is a $\B$-bilinear idempotent
$\Psi:T^2_\dee(\B)\to T^2_\dee(\B)$ 
such that $\Psi\circ\dag=\dag\circ\Psi$ and $JT^{2}_{\dee}\subset\mathrm{Im}(\Psi)$. 
Given a second order differential structure $(\Omega^1_\dee(\B),\dag,\Psi)$ we set 
\[
\Lambda^2_{\d}(\B):=(1-\Psi)T^2_\dee(\B).
\]
\end{defn}

The point of Definition \ref{ass:standing1} is to view differential two-forms as a submodule of $T^2_\dee$ rather than as a quotient of $T^2_\dee$. For differential forms on manifolds these approaches are equivalent.
As well as \cite{BGJ2,BGJ1}, the notion of a ``lifting'' of two-forms to two-tensors appears in \cite[p574ff]{BMBook} in order to define Ricci tensors, and plays a similar role as $1-\Psi$ in that setting. Such liftings are used in the following Lemma to show that any abstract second order $\dag$-calculus consisting of modules that are projective in a way compatible with $\dag$, in fact fits Definition \ref{ass:standing1}.

\begin{lemma}
\label{lem:projectivepsi}
Suppose we have a second order $\dag$-differential graded algebra 
$
\B\xrightarrow{\dee} \Omega^{1}_{\dee}\xrightarrow{\dee}\Omega^{2}_{\dee}
$ 
such that 
\begin{enumerate}
\itemsep-3pt
\item $\Omega^{1}_{\dee}$ is generated by $\B$ and $\dee(\B)$;
\item the multiplication of forms $m:T^{2}_{\dee}\to \Omega^{2}_{\dee}$ is surjective and  admits a $\dag$-invariant right $\B$-linear splitting (which may be called $\dag$-projectivity).
\end{enumerate} 
Then there exists an idempotent $\Psi:T^2_\dee(\B)\to T^2_\dee(\B)$ as in Definition \ref{ass:standing1} and $\Omega^{2}_{\dee}\simeq \Lambda^{2}_{\d}$.
\end{lemma}
\begin{proof} 
The bimodule $JT^{2}_{\dee}$ is generated by elements $\sum\dee(a^{0}_{j})\otimes \dee(a^{1}_{j})$ with $\sum a^{0}_{j}\dee(a^{1}_{j})=0$. Since we have a second order $\dag$-differential graded algebra, 
\[
m\left(\sum\dee(a^{0}_{j})\otimes \dee(a^{1}_{j})\right)=\sum\dee(a^{0}_{j})\cdot \dee(a^{1}_{j})=\dee\left(\sum a_{j}^{0}\dee(a^{1}_{j})\right)=0,
\]
so there is an inclusion $JT^{2}_{\dee}\subset \ker m$. Let $s:\Omega^{2}_\dee\to T^{2}_{\dee}$ be a $\dag$-invariant right $\B$-linear splitting for the surjective $\dag$-bimodule map $m:T^{2}_{\dee}\to \Omega^{2}_{\dee}$.
Then 
\[
s(a\omega)=s((\omega a)^{\dag})^{\dag}=s((\omega^{\dag})a^{*})^{\dag}=as(\omega^{\dag})^{\dag}=as(\omega),
\] 
so $s$ is a bimodule map. We conclude that $\Psi:=1-s\circ m$ is an idempotent satisfying $\dag\circ\Psi=\Psi\circ \dag$ and $JT^{2}_{\dee}\subset \mathrm{Im}(\Psi)$.
\end{proof}
The multiplication of one-forms in Lemma \ref{lem:projectivepsi} typically differs from the product of one-forms that we have when $\Omega^1_\dee$ is contained in an algebra. In the latter case $m(T^2_\dee)$ will usually not provide a module of two-forms due to the presence of junk.

The next result shows that conversely, the existence of $\Psi$ as in Definition \ref{ass:standing1} yields a second order calculus.
\begin{prop}
\label{prop:diffrep}
Let $(\Omega^1_\dee(\B),\dag,\Psi)$ be a second order differential structure. The map
\begin{align}
\d&:\,\Omega^1_\dee(\B)\to T^2_\dee(\B),
& &\d(\omega):=(1-\Psi)(\widehat{\pi}\circ \delta\circ \pi^{-1})(\omega)
\end{align}
is well-defined and satisfies $\d\circ\mathrm{d}=0$.
For $a,b\in\B$ and $\omega\in\Omega^1_\dee(\B)$ we also have
\begin{align}
\label{eq: almost-derivation}
\d(\omega^{\dag})=\d(\omega)^{\dag},\quad \d(a\omega b)=(1-\Psi)(\dee(a)\otimes \omega b)+a(\d\omega)b-(1-\Psi)(a\omega\otimes\dee(b)).
\end{align}
\end{prop}
\begin{proof}
Since $JT^2_{\dee}\subset \mathrm{Ran }\Psi$, $(1-\Psi)JT_{\dee}^{2}=0$, so $\d$ is well-defined and satisfies $\d\circ \dee=0$. Equation \eqref{eq: almost-derivation} follows from the fact that $\pi$ is a $\dag$-bimodule map and 
that $\delta:\Omega^{1}_{u}(\B)\to \Omega^{2}_{u}(\B)$ is a graded derivation. 
\end{proof}
\begin{rmk}
The map $\d$ is not a $\B$-bimodule derivation in the usual sense, instead $\d$ satisfies the more intricate Equation \eqref{eq: almost-derivation} arising from the graded Leibniz rule, and the presence of $(1-\Psi)$.
\end{rmk}
\begin{rmk}
One could take $\Psi={\rm Id}_{T^2_\dee}$, but the result would be $\d=0$
yielding a vacuous theory. In practice, there is usually a natural constraint on the range of $\Psi$ in examples, as described in subsection \ref{subsec:has-to-be-done}. 
\end{rmk}
\begin{rmk} Note that if $\Omega^{1}_{\dee}$ is projective, then so is $T^{2}_{\dee}$ and by Proposition \ref{prop:diffrep}, the existence of $\Psi$ yields a second order $\dag$-differential graded algebra with $\Lambda^{2}_{\dee}:=(1-\Psi)(T^{2}_{\dee})$ a projective $\dag$-sub-bimodule of $T^{2}_{\dee}$. Together with Lemma \ref{lem:projectivepsi}, this shows that every second order $\dag$-differential graded algebra with both $\Omega^{1}_{\dee}$ and $\Omega^{2}_{\dee}$ projective, such the quotient map $T^{2}_{\dee}\to\Omega^{2}_{\dee}$ is $\dag$-split, arises in this way. The assumption that $\Omega^{1}_{\dee}$ and $\Omega^{2}_{\dee}$ are projective can be considered a ``geometric'' assumption, as they are then noncommutative vector bundles. The existence of a $\dag$-invariant splitting allows one to define Ricci curvature (see \cite[p574ff]{BMBook} and \cite{MRCurve}), and can therefore be considered a geometric assumption as well.
\end{rmk}

\begin{example}
\label{eg:classical-junk-proj}
In Example \ref{eg:classical-junk} we saw that for a Dirac triple on a manifold, any symmetric two-tensor arises as junk. So to define $\d$ in line with the usual exterior derivative we need to choose $\Psi$ to be the projection onto symmetric forms. For one forms $\omega,\rho$ we have
\[
\Psi(\omega\ox\rho)=\frac{1}{2}(\omega\ox\rho+\rho\ox\omega).
\]
The map $\d$ is obtained by sending $a\dee(b)\mapsto \dee(a)\ox\dee(b)$ and then anti-symmetrising.
\end{example}
To define curvature and torsion, we require the second order differential $\d:\Omega^1_\dee\to \Lambda^2_\dee$, but not any higher degree forms. We introduce torsion in the next section, while for curvature  the classical definition is now available.

\begin{defn}
If $(\Omega^1_\dee(\B),\dag,\Psi)$ is a second order differential structure,  define the curvature of any right connection $\nablar$ on a finite projective right module $\X_\B$ to be
\[
R^{\nablar}(x)=(1\ox(1-\Psi))\circ (\nablar\ox 1+1\ox\d)\circ\nablar(x)\in \X\ox_\B\Lambda^2_\dee(\B),\qquad x\in\X.
\]
For a connection $\nablal$ on a left module ${}_\B\X$ we define
the curvature to be
\[
R^{\nablal}(x)=((1-\Psi)\ox1)\circ (1\ox\nablal-\d\ox1)\circ\nablal(x)\in \Lambda^2_\dee(\B)\ox_\B\X,\qquad x\in\X.
\]
\end{defn}

The difference between the signs for left and right here is due to the fact that $\d$ is a graded derivation. We will pursue computations of curvature in \cite{MRCurve}.

\subsection{Comparison with Connes' calculus}
\label{subsec:has-to-be-done}

In the situation that the ambient bimodule $\Omega^1_\dee(\B)\subset M$ is a $*$-algebra, we can build modules of higher degree forms $\Omega^k_\dee(\B)$ using the multiplication. Such modules of forms arose first in the context of spectral triples, \cite{BRB}, where $M=\mathbb{B}(H)$ for some Hilbert space $H$.

\begin{defn} 
Let $\Omega^1_\dee(\B)\subset M$ be a first order differential structure with $\dag$-structure given by the adjoint on $M$.
The represented $k$-tensors for the first order differential structure $(\Omega^1_\dee(\B),\dag)$ are defined by
$$
\Omega^k_\dee(\B):=m(T^k_\dee(\B))=\Big\{\sum_{j=1}^Na_j\dee(b_j^{1})\cdots \dee(b^{k}_j):\,a_j,\,b_j^{i}\in\B, \ 1\leq i\leq k, \ N\in\N\Big\}.
$$
\end{defn}

We let $\Omega^*_\dee(\B)\subset M$ be the $*$-algebra 
generated by $\B$ and $\Omega^1_\dee(\B)$ and denote by 
\[
m:M^{\otimes k}\to M,\quad T_{1}\otimes \cdots \otimes T_{k}\mapsto T_{1}\cdots T_{k} 
\] 
the $k$-fold multiplication map.
Composition of $\widehat{\pi}$ with $m$ in each degree gives a surjective $*$-algebra homomorphism
\[
m\circ\widehat{\pi}:\Omega_{u}^{*}(\B)\xrightarrow{\widehat{\pi}}T^{*}_{\dee}(\B)\xrightarrow{m} \Omega^*_\dee(\B).
\]
\begin{rmk}
The distinctions between $\widehat{\pi}$ and $m\circ\widehat{\pi}$ were highlighted in \cite{BGJ1,BGJ2}. Here we point out an important difference between the modules $T^*_\dee$ and $\Omega^*_\dee$.  An inner product on $\Omega^1_\dee$ induces inner products on all $T^k_\dee$, but additional assumptions are necessary to obtain inner products on $\Omega^k_\dee$. We will not consider inner products on $\Omega^k_\dee$, $k\geq2$, and work
predominately with $T^*_\dee$ after this section.
\end{rmk}

Analogously to $JT^n_\dee$, in this particular setting we can define junk tensors
\[
J^n_\dee:=\{T\in\Omega^n_\dee(\B):\,T=m\circ\widehat{\pi}(\delta(\omega)),\ \omega\in \Omega_u^{n-1}(\B)\cap\ker m\circ\widehat{\pi}\}\subset \Omega^n_\dee(\B).
\]

In this context we can ask that the idempotent
$\Psi\in \End_\B^{*}(T^2_\dee(\B))$ 
satisfies the stronger criteria
\[
JT^{2}_{\dee}\subset\mathrm{Im}(\Psi)\subset m^{-1}(J^{2}_{\dee}).
\] 
Since $m(JT^{2}_{\dee})=J^{2}_{\D}$, this requirement is equivalent to requiring that $JT^{2}_{\dee}\subset\mathrm{Im}(\Psi)$ and $\mathrm{Im}(m\circ\Psi)=J^{2}_{\dee}$. In \cite{BGJ2} the assumption is that the module of one-forms is centred, ${\rm Im}(\Psi)=m^{-1}(J^{2}_{\dee})$ and that $\Psi=(\sigma+1)/2$ where $\sigma$ is the canonical flip \cite{Skeide}, defined only on centred bimodules. The ability to define $\Omega^k_\dee(\B)$ allows us to require a natural upper bound on the size of ${\rm Im}(\Psi)$.

\begin{prop}
\label{prop:diffrep-AGAIN}
Let $(\Omega^1_\dee(\B),\dag,\Psi)$ be a second order differential structure
with $\Omega^1_\dee(\B)\subset M$ for some $*$-algebra $M$. 
The map
\begin{align}
m\circ\d&:\,\Omega^1_\dee(\B)\to \Omega^2_\dee(\B),
& &m\circ\d(\omega):=m\circ(1-\Psi)(\widehat{\pi}_\dee\circ \delta\circ \pi_\dee^{-1})(\omega)
\end{align}
is well-defined and satisfies $m\circ\d\circ\mathrm{d}=0$. There is a commutative diagram
\[
\xymatrix{
  \Omega^{1}_{\dee}(\B)\ar[d]^{\d}\ar[r]^{\mathrm{d}_C\ \ \\ } &\Omega^{2}_{\dee}(\B)/J^{2}_{\dee}\\ 
 T^2_{\dee}(\B) \ar[r]^{m}  & \Omega^{2}_{\dee}(\B) \ar[u]^{q}}
 \]
in which $\mathrm{d}_C:\Omega^{1}_{\dee}(\B)\to \Omega^{2}_{\dee}(\B)/J^{2}_{\dee}$ is the Connes differential and $q:\Omega^{2}_{\dee}(\B)\to  \Omega^{2}_{\dee}(\B)/J^{2}_{\dee}$ the quotient map.
\end{prop}
\begin{proof}
We saw in Proposition \ref{prop:diffrep} that $\d$ is well-defined, and so too then is $m\circ\d$. Then, since in degree 1, $\widehat{\pi}=\pi$ and $q\circ m\circ\Psi=0$ we have for any $\omega\in\Omega^{1}_{\dee}(\B)$
\begin{align*}
q\circ m\circ\d(\omega)&= q\circ m \circ (1-\Psi)\circ \widehat{\pi}\circ  \delta \circ \pi^{-1}(\omega)
=q\circ m \circ  \widehat{\pi}\circ  \delta \circ \pi^{-1}(\omega)\\ &=q\circ \widehat{\pi}\circ  \delta \circ\pi^{-1}(\omega)=\mathrm{d}_C(\omega),
\end{align*}
as desired.
\end{proof}

The maps $\d$ and $m\circ\d$ are lifts of Connes' differentials $\widehat{\dee}_C$ and $\dee_C$ respectively. Thus when $\Omega^1_\dee(\B)\subset M$ for a $*$-algebra $M$,  Connes' calculus can be defined and we recover it.

\section{Existence of Hermitian torsion-free connections}
\label{subsec:RNDS}
To discuss Hermitian torsion-free connections, we equip the module of one-forms  with an inner product and differential.

\begin{defn}
\label{ass:standing}
An \emph{Hermitian differential structure} $(\Omega^1_\dee(\B),\dag,\Psi,\pairing{\cdot}{\cdot}_\B)$ is a second order differential structure $(\Omega^1_\dee(\B),\dag,\Psi)$
such that
$\Omega^1_\dee(\B)$ is finitely generated and projective as a right $\B$-module with a
right $\B$-valued inner product $\pairing{\cdot}{\cdot}_\B$ such that
$\Psi=\Psi^{*}\in \overrightarrow{\End}_\B^{*}(T^2_\dee(\B))$. 
\end{defn}

If $(\Omega^1_\dee(\B),\dag,\Psi,\pairing{\cdot}{\cdot}_\B)$ is an Hermitian differential structure, then $\Omega^{1}_{\dee}(\B)$ is a $\dag$-bimodule,
and Corollary \ref{cor:left-right-adjoint} says that $\Psi$ is left adjointable.
For the remainder of this section, $G$ will be the quantum metric of an Hermitian differential structure $\Omega^1_\dee(\B)$, as in Definition \ref{defn:Geeeeeeee-ex}.  
We also make the definition
\begin{defn}
\label{defn:pee-kew}
Let  $(\Omega^1_\dee(\B),\dag,\Psi,\pairing{\cdot}{\cdot}_\B)$ be 
an Hermitian differential structure.
Define the $\B$-bimodule projections
\[
P,Q:T^3_\dee(\B)\to T^3_\dee(\B)
\]
by $P=\Psi\ox 1$ and $Q=1\ox\Psi$. 
\end{defn}

%

\subsection{The torsion tensor}

The vanishing of torsion in non-commutative algebraic settings has been investigated by many authors
\cite{BMBook,BGJ1,BGJ2,DM,DMM,DHLS,DSZ,Landi}. This algebraic approach characterises torsion-free connections (in our set-up) as follows.

\begin{defn}
\label{defn:torsion-free}
Let $(\Omega^1_\dee(\B),\dag,\Psi,\pairing{\cdot}{\cdot}_\B)$ be an Hermitian differential structure. The \emph{torsion} of a right connection $\nablar$ is the map
\[\overrightarrow{T}_{\!\!\Psi}(\nablar):=(1-\Psi)\nablar+\d:\Omega^{1}_{\dee}\to T^{2}_{\dee}.\]
Similarly, for a left connection $\nablal$ the \emph{torsion} is the map
\[\overleftarrow{T}_{\!\!\Psi}(\nablal):=(1-\Psi)\nablal-\d:\Omega^{1}_{\dee}\to T^{2}_{\dee}.\]
A connection is \emph{torsion-free} if its torsion vanishes.
\end{defn}
Using the Leibniz rules \eqref{eq: almost-derivation}, one verifies that \[\overrightarrow{T}_{\!\!\Psi}(\nablar)\in\overrightarrow{\mathrm{Hom}}(\Omega^{1}_{\dee},T^{2}_{\dee}),\quad\textnormal{and}\quad \overleftarrow{T}_{\!\!\Psi}(\nablal)\in\overleftarrow{\mathrm{Hom}}(\Omega^{1}_{\dee},T^{2}_{\dee}).\]
A right connection $\nablar:\Omega^1_\dee\to T^2_\dee$ is torsion-free if $(1-\Psi)\circ\nablar=-\d$. A left connection is torsion-free if $(1-\Psi)\circ\nablal=\d$.

The torsion of a connection so defined is independent of the inner product and $\dag$-structure on the $\dag$-bimodule $\Omega^{1}_{\dee}$. Using these extra structures, Lemma \ref{lem:useful-iso} gives us isomorphisms
\[\overrightarrow{\mathrm{Hom}}(\Omega^{1}_{\dee},T^{2}_{\dee})\xleftarrow{\alphar}T^{3}_{\dee}\xrightarrow{\alphal}\overleftarrow{\mathrm{Hom}}(\Omega^{1}_{\dee},T^{2}_{\dee}).\]
Therefore, under our assumption that $\Omega^{1}_{\dee}$ is a $\dag$-bimodule, the torsion of a left- or right connection corresponds to three tensors in $T^{3}_{\dee}$ given by
\[\alphar^{-1}(\overrightarrow{T}_{\!\!\Psi}(\nablar))=\sum_{j}(1-P)(\nablar(\omega_{j})\otimes \omega_{j}^{\dag}+\d(\omega_{j})\otimes\omega_{j}^{\dag})\]
\[
\alphal^{-1}(\overleftarrow{T}_{\!\!\Psi}(\nablal))=\sum_{j}( 1-Q)(\omega_{j}\otimes \nablal(\omega_{j}^{\dag})- \omega_{j}\otimes\d\omega_{j}^{\dag}).
\]
The next result shows that for Hermitian connections, the vanishing of the torsion can be interpreted as the requirement that the covariant derivative of the quantum metric $G$ vanishes. 

\begin{prop}
\label{prop:herm+tf=herm+tf}
Suppose that $\nablar$ is a Hermitian right connection on $\Omega^{1}_{\dee}(\B)$ and $\nablal$ the conjugate left connection as in Definition \ref{defn:conj-conn}. Then $(1-\Psi)\nablar=-\d$ if and only if $(1-\Psi)\nablal=\d$ if and only if
\[(1-Q)(\nablar\otimes 1+1\otimes \d)(G)=(1-P)(\d\otimes 1 - 1\otimes\nablal)(G)=0.\] 
\end{prop}
\begin{proof}
By Proposition \ref{prop:diffrep} we have  $\d(\omega^{\dag})^\dag=\d(\omega)$ and since $\Psi\circ\dag=\dag\circ\Psi$ 
\[
(1-\Psi)\nablal(\omega)=-(1-\Psi)\nablar(\omega^{\dag})^\dag=-((1-\Psi)\nablar(\omega^{\dag}))^{\dag},
\]
it follows that $(1-\Psi)\nablar=-\d$ if and only if $(1-\Psi)\nablal=\d$. By Proposition \ref{prop:left-right-conn}
\[\sum_{j}\nablar(\omega_{j})\otimes \omega_{j}^{\dag}=-\sum_{j}\omega_{j}\otimes\nablal(\omega_{j}^{\dag}),\]
for Hermitian connections. Thus, if $\overleftarrow{T}_{\!\!\Psi}(\nablal)=0$ then 
\begin{align*}
0&=\alphal^{-1}\left(\overleftarrow{T}_{\!\!\Psi}(\nablal)\right)\\
&=
\sum_{j}(1-Q)\omega_{j}\otimes \nablal(\omega_{j}^{\dag})-\omega_{j}\otimes \d\omega_{j}^{\dag}\\
&=-\sum_{j}(1-Q)\left(\nablar(\omega_{j})\otimes \omega_{j}^{\dag}+\omega_{j}\otimes \d\omega_{j}^{\dag}\right)\\
&=-(1-Q)(\nablar\otimes 1+1\otimes \d)(G),
\end{align*}
The argument for right connections is similar.
\end{proof}

\begin{rmk}
We emphasise that the equivalences of the previous result hold only for Hermitian connections.
\end{rmk}

\subsection{The connection form of a Hermitian torsion-free connection}
In this section we examine existence of Hermitian torsion-free connections on the module $\Omega^{1}_{\dee}(\B)$ associated to an  Hermitian differential structure $(\Omega^{1}_{\dee}(\B),\dag,\Psi, \pairing{\cdot}{\cdot}_\B)$. We will drop the $\B$ from $\pairing{\cdot}{\cdot}_\B$ for simplicity: any unlabelled inner product is a right inner product.

Recall the projections $P=\Psi\otimes 1,Q=1\otimes\Psi$ are both left and right adjointable maps on $T^3_\dee(\B)$.
For a right frame $(\omega_{j})$ for $\Omega^1_\dee(\B)$ we write
\begin{equation}
\label{eq: W}
W:=\sum_{j}\d(\omega_{j})\otimes \omega_{j}^{\dag},\quad W^{\dag}
:=\sum_{j}\omega_{j}\otimes \d(\omega_{j}^{\dag}) \in T^3_\dee(\B).
\end{equation}
The tensors $W,W^\dag$ are frame dependent. Via  
Lemmas \ref{lem:vee} and \ref{defn:grassy}, the frame $(\omega_j)$ also gives us
the associated stabilisation isometry and Grassmann connection
\[
v:\Omega^{1}_{\dee}(\B)\to \B^{N},\quad\overrightarrow{\nabla}^{v}: \Omega^{1}_{\dee}(\B)\to T^2_{\dee}(\B).
\]
We write a general right connection as
\[
\nablar=\nablar^{v}+\alphar(A_{v})
\]
where using Proposition \ref{lem:compatible}, $A_{v}=\nablar(\omega_j)\ox\omega_j^\dag$ is an element of $T^3_\dee(\B)$. 
\begin{lemma}
\label{lem: torsion-free iff} 
Let $\nablar:\Omega^{1}_{\dee}(\B)\to T^{2}_{\dee}(\B)$ be an Hermitian connection. Then $\nablar$ is torsion-free if and only if for any frame $v=(\omega_{j})$ the connection form $A_{v}$ satisfies $(1-P)A_{v}=-W$.
\end{lemma}
\begin{proof}
Writing $\nablar=\nablar^{v}+\alphar(A_{v})$, by Proposition \ref{prop:herm+tf=herm+tf}, $\nablar$ is 
torsion free if and only if $(1-\Psi)\nablar=-\d$. This holds if and only if
\begin{align*}
-\d=(1-\Psi)\nablar=(1-\Psi)\nablar^{v}+(1-\Psi)\alphar(A_{v})=(1-\Psi)\nablar^{v}+\alphar((1-P)A_{v}),
\end{align*}
which is true if and only if
\begin{align*}
\alphar((1-P)A_{v})&(\omega)=-\d(\omega)-(1-\Psi)\nablar^{v}(\omega)=-\d(\omega)-(1-\Psi)\big(\sum_{j}\omega_{j}\otimes \dee(\pairing{\omega_{j}}{\omega})\big)\\
&=\sum_{j}-\d(\omega_{j})\pairing{\omega_{j}}{\omega}+(1-\Psi)(\omega_{j}\otimes \dee(\pairing{\omega_{j}}{\omega}))-(1-\Psi)(\omega_{j}\otimes \dee(\pairing{\omega_{j}}{\omega})\\
&=\sum_{j} -\d(\omega_{j})\pairing{\omega_{j}}{\omega}
=-\alphar(W)(\omega).
\end{align*}
Since $\alphar$ is an isomorphism, $\nablar$ is torsion-free if and only if $(1-P)A_{v}=-W$.
\end{proof}
Our next observation is central to the development of the subsequent theory and the construction of Hermitian torsion-free connections on noncommutative differential forms.
\begin{lemma} 
\label{prop: connection-form}
Let $\nablar:\Omega^{1}_{\dee}(\B)\to T^2_{\dee}(\B)$ be an Hermitian torsion-free right connection. For any frame $(\omega_{j})$ the decomposition $\nablar=\nablar^{v}+\alphar(A_{v})$ satisfies
\[
(1-PQ)A_{v}=-W-PW^{\dag},\quad (1-QP)A_{v}=-W^{\dag}-QW.
\]
Moreover for every $n>0$ we have
\begin{equation}
\label{eq:series}
(1-(PQ)^{n})A_{v}=-\sum_{k=0}^{n-1}(PQ)^{k}(W+PW^{\dag}),\quad (1-(QP)^{n})A_{v}=-\sum_{k=0}^{n-1}(QP)^{k}(W^{\dag}+QW).
\end{equation}
\end{lemma}
\begin{proof}
Writing $\nablar=\nablar^{v}+\alphar(A_{v})$, the fact that $\nablar$ is 
Hermitian gives $A^{\dag}_{v}=A_{v}$. By Lemma \ref{lem: torsion-free iff}, $\nablar$ is torsion free if and only if $(1-P)A_{v}=-W$ if and only if $(1-Q)A_{v}=-W^{\dag}$.
These equations are equivalent to the equations
\begin{equation}
\label{eq:iterate}
A_{v}=PA_{v} -W,\quad A_{v}=QA_{v}-W^{\dag}.
\end{equation}
Substituting Equations \eqref{eq:iterate} into one another yields
\[
A_{v}=PQ A_{v} -W-PW^{\dag},\quad A_{v}=QPA_{v}-W^{\dag}-QW.
\]
This gives the desired equations for $n=1$. The equations for $n>1$ follow by induction  using Equation \eqref{eq:iterate} and the fact that $PW=QW^{\dag}=0$.
\end{proof}

The formulae \eqref{eq:series} suggest that taking a limit as $n\to\infty$ will allow us to construct the connection form $A_{v}$ relative to the frame $(\omega_{j})$. Therefore we need to address the convergence properties of the relevant sequences associated to the pair of projections $(P,Q)$. This is the subject of the next subsection, which is used only as a reference in the sequel. The reader interested mainly in the existence and construction of a Hermitian torsion-free connection can proceed directly to subsection \ref{subsec: existence}.



\subsection{The two projection problem}
\label{subsec:freaky-angle}

In view of Lemma \ref{prop: connection-form}, we need to characterise the limiting properties of the sequence $(PQ)^{n}$. This leads us to consider the \emph{two projection problem} in Hilbert $C^{*}$-modules, see \cite{MRtwoprojns,DMR23} and references therein. The module $T^{3}_{\dee}$ is only a pre-$C^{*}$-module over a local algebra $\B$, and we therefore now present a refinement of the $C^{*}$-module two-projection problem as discussed in \cite{MRtwoprojns}. In particular, this yields a sufficient condition that will assure us that the limit of the series $\sum_{k=0}^{n}(PQ)^{k}$ in \eqref{eq:series} is a projection on $T^{3}_{\dee}$. It is remarkable that this sufficient condition, stated in Theorem \ref{thm:smo-intersect} below, is purely algebraic.

\begin{defn} Let $\mathcal{B}$ be $*$-algebra and $\mathcal{X}$ a right inner product $\mathcal{B}$-module.
Given a submodule $\mathcal{M}\subset \mathcal{X}$ its \emph{orthogonal complement} is the submodule
\[
\mathcal{M}^{\perp}:=\left\{x\in \mathcal{X}:\pairing{m}{x}_\B=0\,\, \forall x\in \mathcal{M}\right\}.
\]
By orthogonality, the addition map $\mathcal{M}\oplus \mathcal{M}^{\perp}\to \mathcal{M}+\mathcal{M}^{\perp}$ is an inner product preserving isomorphism and 
the module $\mathcal{M}$ is \emph{complemented in} $\mathcal{X}$ if 
$\mathcal{X}=\mathcal{M}+\mathcal{M}^{\perp}$. 
\end{defn}

If $B$ is a $C^{*}$-algebra and $X$ a Hilbert $C^{*}$-module it is well-known that not every closed submodule is complemented. In general there is a bijection between complemented submodules of $\X$ and projections $P=P^{2}=P^{*}\in\End^{*}_{\B}(\X)$. The decomposition $\X=\mathcal{M}+\mathcal{M}^{\perp}$ guarantees that every $x\in \X$ can be written uniquely as a sum $x=x_{\mathcal{M}}+x_{\mathcal{M}^{\perp}}$ and orthogonality guarantees that $P_{\mathcal{M}}(x):=x_{\mathcal{M}}$ is a self-adjoint idempotent. Conversely a projection $P\in\End^{*}_{\B}(\mathcal{X})$ gives rise to a complemented submodule $\mathcal{M}:=P(\mathcal{X})$ with $\mathcal{M}^{\perp}:=(1-P)(\mathcal{X})$. 

For the present paper, our setting is that of complemented submodules $\mathcal{M},\mathcal{N}\subset \mathcal{X}$ of a finitely generated projective inner product module $\mathcal{X}$ over a unital ${*}$-algebra $\B$, with associated projections $P_{\mathcal{M}}, P_{\mathcal{N}}$. Observe that there is an exact sequence
\begin{align}
\label{sum-intersection-sequence} 0\to \mathcal{M}\cap \mathcal{N} \to \mathcal{M}\oplus \mathcal{N}\to \mathcal{M}+\mathcal{N}\to 0,
\end{align}
where the first non-zero map is the embedding 
\[
\mathcal{M}\cap \mathcal{N} \to \mathcal{M}\oplus \mathcal{N},\quad x\mapsto (x,-x),
\]
and the second map is given by addition in $\mathcal{X}$. The module $\mathcal{M}\oplus \mathcal{N}$ is finitely generated and projective. However, the submodule $\mathcal{M}\cap \mathcal{N}\subset \mathcal{X}$ need not be finitely generated, whereas the finitely generated submodule $\mathcal{M}+\mathcal{N}\subset \mathcal{X}$ need not be projective. Note that there is an inclusion $\mathcal{M}^{\perp}+\mathcal{N}^{\perp}\subset (\mathcal{M}\cap \mathcal{N})^{\perp}$.

In our setting the $*$-algebra $\B$ is dense in a $C^{*}$-algebra $B$, and by taking $C^{*}$-module closures $X:=\overline{\mathcal{X}}$, $M:=\overline{\mathcal{M}}$ and $N:=\overline{\mathcal{N}}$ we obtain a pair of complemented submodules of a Hilbert $C^{*}$-module $X$ fitting in an exact sequence like \eqref{sum-intersection-sequence}. Recall that a Hilbert $C^{*}$-module over a unital $C^{*}$-algebra is finitely generated if and only if it is projective \cite[Theorem 5.9]{FL02}. For finitely generated modules we can  use analysis to address the algebraic properties of the exact sequence \eqref{sum-intersection-sequence}, via the limiting properties of the sequence $(P_{M}P_{N})^{n}$.

Recall that the $*$-strong topology on $\End^{*}_{B}(X)$ is given by the seminorms
\[
\|T\|_{x}:=\max\left\{\|Tx\|,\|T^{*}x\|\right\},\quad x\in X.
\]
It is straightforward to show that if $X$ is algebraically finitely generated, then the norm and $*$-strong topologies coincide.

The following Proposition links the convergence properties of the sequence $(P_{M}P_{N})^{n}$ to the exact sequence \eqref{sum-intersection-sequence}.

\begin{prop}[{cf. \rm\cite[Proposition 3.12]{MRtwoprojns}}]\label{prop:twoprojnsequivs} Let $X$ be a finitely generated Hilbert $C^{*}$-module over a unital $C^{*}$-algebra $B$ and $M,N$ complemented submodules. 
Then the following are equivalent:
\vspace{-8pt}
\begin{enumerate}
\itemsep-3pt
\item $(M\cap N)\oplus (M^{\perp}+N^{\perp})=X$;
\item $(M\cap N)\oplus (M^{\perp}+N^{\perp})$ is dense in $X$;
\item $M+N$ \mbox{ is closed in} $X$;
\item $M\cap N $ \mbox{is complemented } $\|P_{M}P_{N}-P_{M\cap N}\|<1$; 
\item $M\cap N $ \mbox{is complemented and the operator } $1+P_{M\cap N}-P_{M}P_{N}$\mbox{ is invertible} in $\End^{*}_{B}(X)$;
\item \mbox{The sequence }$(P_{M}P_{N})^{n}$\mbox{ is Cauchy for the operator norm};
\item \mbox{The sequence }$(P_{M}P_{N})^{n}$\mbox{ is Cauchy for $*$-strong topology}.
\end{enumerate}
\vspace{-8pt}
In case the above equivalent conditions hold, both $M\cap N$ and $M+N$ are finitely generated projective Hilbert $C^{*}$-modules, the sequence \eqref{sum-intersection-sequence} splits, and $\lim_{n\to\infty}(P_MP_N)^n=P_{M\cap N}$.
\end{prop}
\begin{proof}
The equivalence of (1) and (3)-(6) is proved \cite[Proposition 3.12]{MRtwoprojns}. The equivalence of $(2)$ and $(6)$ is proved in \cite[Theorem 2.9]{MRtwoprojns}. Since we are considering only finitely generated modules, $(6)$ is equivalent to $(7)$ and we are done.
\end{proof}

\begin{rmk} 
In case $B=C(Z)$ is commutative and unital, the projective modules $X,M,N$ correspond to sections of vector bundles $\mathcal{V}\to Z$, with $\mathcal{E},\mathcal{F}\subset\mathcal{V}$ two subbundles. The conditions in Proposition \ref{prop:twoprojnsequivs} are equivalent to the requirement that $\mathcal{E}\cap \mathcal{F}$ and $\mathcal{E}+\mathcal{F}$ are again subbundles of $\mathcal{V}$. This need not be the case (see \cite[Remark 3.17]{MRtwoprojns} and \cite{DMR23}). Note that in the countably generated case, conditions 2.) and 7.) are not equivalent to the others.
\end{rmk}
\begin{rmk} 
\label{rmk:F}
The \emph{Friedrichs angle} $c(M,N)$ is a numerical invariant for  a pair of complemented submodules $(M,N)$ on a (not necessarily finitely generated) Hilbert $C^{*}$-module $X$ (see \cite[Section 3.1]{MRtwoprojns}).  We have $c(M,N)<1$ if and only if the equivalent conditions 1.) and 3.)-6.) of Proposition \ref{prop:twoprojnsequivs} hold, and then $c(M,N)=\|P_{M}P_{N}-P_{M\cap N}\|$. 
\end{rmk}

\begin{corl}
In case the equivalent conditions of Proposition \ref{prop:twoprojnsequivs} hold, we have 
\[
(1+P_{M\cap N}-P_MP_N)^{-1}=\sum_{n=0}^{\infty}(P_MP_N-P_{M\cap N})^{n}=P_{M\cap N}+\sum_{n=0}^{\infty}(P_MP_N)^{n}(1-P_{M\cap N}),
\]
with convergence in norm.
\end{corl}

So far we have shown that the convergence of $(P_MP_N)^{n}$ in the $C^{*}$-norm is equivalent to the module $M+N$ being projective. To obtain a similar statement for the dense submodules $\mathcal{M},\mathcal{N}\subset\mathcal{X}$, 
we need the $*$-algebra $\B$ to be local in its $C^{*}$-closure $B$ (see Definition \ref{def: local}). 

\begin{lemma}
\label{lem: idempotent-subspace} Let $X$ be a vector space, $P:X\to X$ an idempotent and $\mathcal{X}\subset X$ a vector subspace of $X$. If $P\mathcal{X}\subset \mathcal{X}$ then $P\mathcal{X}=\mathcal{X}\cap PX$.
\end{lemma}
\begin{proof}
Since $P^{2}=P$ we have that
$\mathcal{X}\cap PX=P(\mathcal{X}\cap PX)\subset P\mathcal{X}.$
By assumption $P\mathcal{X}\subset \mathcal{X}$, so we also have $P\mathcal{X}\subset \mathcal{X}\cap PX$, and the asserted equality holds.
\end{proof}

\begin{thm} 
\label{thm:smo-intersect}
Let $B$ be a unital $C^{*}$-algebra, $\B\subset B$ a unital local $*$-subalgebra and $\mathcal{X}$ a finitely generated projective inner product module over $\mathcal{B}$ with $C^*$-closure $X$. Suppose that $\mathcal{M}$ and $\mathcal{N}$ are complemented submodules of $\mathcal{X}$.
Then the following are equivalent:
\vspace{-6pt}
\begin{enumerate}
\itemsep-2pt
\item $\mathcal{X}=(\mathcal{M}\cap\mathcal{N})\oplus(\mathcal{M}^{\perp}+\mathcal{N}^{\perp});$
\item the sequence $(P_{\mathcal{M}}P_{\mathcal{N}})^{n}$ is Cauchy for the $C^{*}$-norm on $\End^{*}_{\B}(\mathcal{X})\subset \End^*_B(X)$ and the limit projection $\Pi:=\lim_{n\to\infty}(P_{\mathcal{M}}P_{\mathcal{N}})^{n}$ satisfies $\Pi\mathcal{X}\subset \mathcal{X}$. 
\end{enumerate}
\vspace{-6pt}
If these conditions hold $(1+\Pi-P_{\mathcal{M}}P_{\mathcal{N}})^{-1}\in\End^{*}_{\mathcal{B}}(\mathcal{X})$ and $\Pi(\mathcal{X})=\mathcal{M}\cap\mathcal{N}$ so $\Pi=P_{\mathcal{M}\cap\mathcal{N}}$.
\end{thm}
\begin{proof} 
Suppose that the decomposition 1.) holds. Then taking $C^{*}$-closures $M:=\overline{\mathcal{M}}$, $N:=\overline{\mathcal{N}}$ in $\overline{\mathcal{X}}:=X$ we find that 
\[
X=\overline{\mathcal{M}\cap \mathcal{N}}\oplus\overline{\mathcal{M}^{\perp}+\mathcal{N}^{\perp}}.
\]
Since $\overline{\mathcal{M}\cap \mathcal{N}}\subset M\cap N$ and $M^{\perp}+N^{\perp}\subset \overline{\mathcal{M}^{\perp}+\mathcal{N}^{\perp}}$, orthogonality implies that the direct sum $(M\cap N)\oplus M^{\perp}+N^{\perp}$
is dense in $X$. By Proposition \ref{prop:twoprojnsequivs},
 $(P_{\mathcal{M}}P_{\mathcal{N}})^{n}$ is Cauchy for the $C^{*}$-norm on $\End^{*}_{\mathcal{B}}(\mathcal{X})$ and for $x\in\mathcal{X}$ we can write $x=x_{0}+x_{1}$ with $x_0\in \mathcal{M}\cap \mathcal{N}$ and $x_{1}\in \mathcal{M}^{\perp}+\mathcal{N}^{\perp}$. Therefore $\Pi x=x_0\in\mathcal{X}$.

For the converse, observe that since $\Pi=\lim_{n\to\infty}(P_{\mathcal{N}}P_{\mathcal{M}})^{n}$ and $\Pi\in\End^{*}_{\B}(\X)$ is a projection we have $\mathcal{X}=\Pi(\mathcal{X})\oplus (1-\Pi)(\mathcal{X})$. Since also $P_{\mathcal{M}}\mathcal{X}\subset \mathcal{X}$ and $P_{\mathcal{N}}\mathcal{X}\subset \mathcal{X}$, by Lemma \ref{lem: idempotent-subspace} we find
\begin{align*}
\Pi(\mathcal{X})&= \mathcal{X}\cap \Pi(X)=\mathcal{X}\cap P_{\mathcal{M}}(X)\cap P_{\mathcal{N}}
(X)=\mathcal{M}\cap\mathcal{N}.
\end{align*}

By Proposition \ref{prop:SHFC-End},
the subalgebra $\End^{*}_{\B}(\mathcal{X})\subset \End^{*}_{B}(X)$ is local. Since $\Pi,P_{\mathcal{M}},P_{\mathcal{N}}$ map $\mathcal{X}$ to itself, the invertible operator $1+\Pi-P_{\mathcal{M}}P_{\mathcal{N}}$ maps $\mathcal{X}$ to itself so it is in $\End^{*}_{\B}(\mathcal{X})$. Therefore $(1+\Pi-P_{\mathcal{M}}P_{\mathcal{N}})^{-1}\in \End^{*}_{\B}(\mathcal{X})$ as well, and thus 
$$
\mathcal{X}=(1+\Pi-P_{\mathcal{M}}P_{\mathcal{N}})(\mathcal{X})=\Pi(\mathcal{X})\oplus (1-P_{\mathcal{M}}P_{\mathcal{N}})(\mathcal{X}),
$$
so that $\Pi(\mathcal{X})=(1-P_{\mathcal{M}}P_{\mathcal{N}})(\mathcal{X})^\perp$. 
Now $1-P_{\mathcal{M}}P_{\mathcal{N}}=(1-P_{\mathcal{M}})P_{\mathcal{N}}+1-P_{\mathcal{N}}$ from which we derive that
\[
(1-P_{\mathcal{M}}P_{\mathcal{N}})(\mathcal{X})\subset \mathcal{M}^{\perp}+\mathcal{N}^{\perp}\subset\ker \Pi=(1-\Pi)(\mathcal{X})=(1-P_{\mathcal{M}}P_{\mathcal{N}})(\mathcal{X}),
\]
so all of these sets are equal. This proves that 
$\mathcal{X}=(\mathcal{M}\cap\mathcal{N})\oplus(\mathcal{M}^{\perp}+\mathcal{N}^{\perp}).$
\end{proof}
\begin{defn} Let $\mathcal{B}$ be a local subalgebra of a unital $C^{*}$-algebra $B$ and $\mathcal{X}$ a finitely generated and projective inner product $\mathcal{B}$ module. A pair $(\mathcal{M},\mathcal{N})$ of $\mathcal{B}$-submodules of $\mathcal{X}$ is \emph{concordant} if $\mathcal{X}=(\mathcal{M}\cap\mathcal{N})\oplus(\mathcal{M}^{\perp}+\mathcal{N}^{\perp}).$
\end{defn}
\begin{rmk}
Given two projections $P,Q$ on a Hilbert $C^{*}$-module $X$, the self-adjoint unitaries $2P-1$ and $2Q-1$ generate a representation of the infinite dihedral group $\Z\rtimes\Z_2$, whose $C^*$-algebra is the universal algebra generated by two projections, \cite{RS}. Asking for concordance is a constraint on this representation. 
\end{rmk}

\begin{prop}
\label{prop:pi-compact}
Let $X$ be a right Hilbert $C^*$-module over $B$ and $P,Q\in \End^{*}_{B}(X)$ two projections. If the associated unitary representation $\pi:\Z\rtimes\Z_2\to\End^*_B(X)$ has norm precompact image then the submodules $PX$ and $QX$ are concordant.
\end{prop}

\begin{proof}
\label{def:concordant-pair}
The precompactness of $\pi(\Z\rtimes\Z_2)\subset \End_B^*(X)$ ensures that $(PQ)^n$ has a convergent subsequence. By \cite[Lemma 2.3]{MRtwoprojns}, the sequence $(PQ)^n$ itself converges, necessarily to $\Pi$ the projection onto ${\rm Im}(P)\cap{\rm Im}(Q)$. By Proposition \cite[3.12]{MRtwoprojns}, the submodules 
$PX$ and $QX$ are concordant.
\end{proof}

\begin{rmk}
The precompactness of the image of $\pi$ means that, viewed as an element of $C_b(\Z\rtimes\Z_2,\End^*_B(X))$, we in fact have $\pi\in C_b(\Z\rtimes\Z_2)\ox\End_B^*(X)$, \cite{Williams}.
\end{rmk}
\begin{lemma}
\label{lem: finite-action} 
Let $\mathcal{B}$ be a local subalgebra of a unital $C^{*}$-algebra $B$ and $\mathcal{X}$ a finitely generated and projective inner product $\mathcal{B}$-module. Suppose that $P,Q\in\End^{*}_{\B}(\X)$ are projections such that the self-adjoint unitaries $2P-1$ and $2Q-1$ generate a finite group inside $\End^{*}_{\B}(\X)$.
Then the submodules $P\X$ and $Q\X$ are concordant.
\end{lemma}
\begin{proof}
If the representation takes values in a finite group $\Gamma$, the image is compact, and so we obtain concordance from Proposition \ref{prop:pi-compact}. If furthermore the generators $(2P-1),(2Q-1)$ preserve a smooth submodule $\X$, then the
limit projection $\Pi\in \End^*_B(X)$ preserves $\X$. To see this, write
$\Pi$ as the finite sum
\[
\Pi=\frac{1}{|\Gamma|}\sum_{\gamma\in\Gamma}\gamma.
\]
Since each element preserves $\X$, so Theorem 4.18
gives concordance for $\X$.
\end{proof}
For  the canonical flip $\sigma$ of a centred bimodule \cite{Skeide}, the unitaries $\sigma\ox1$ and $1\ox\sigma$ generate a representation which factors through the permutation group $S_3$, \cite{BGJ2,BGJ1}, as in Corollary \ref{cor: S3action}.
We conclude our discussion with a corollary for future reference, as many geometric examples exhibit the structure described by it. 

\begin{corl}
\label{cor: S3action} 
Let $\mathcal{B}$ be a local subalgebra of a unital $C^{*}$-algebra $B$ and $\mathcal{X}$ a finitely generated and projective inner product $\mathcal{B}$-module. Suppose that $P,Q\in\End^{*}_{\B}(\X)$ are projections such that
\begin{equation}
\label{eq: abstract-Yang-Baxter}
(2P-1)(2Q-1)(2P-1)=(2Q-1)(2P-1)(2Q-1).
\end{equation}
Then the submodules $P\X$ and $Q\X$ are concordant.
\end{corl}
\begin{proof}
Equation \eqref{eq: abstract-Yang-Baxter} implies that the self-adjoint unitaries $U:=2P-1$ and $V=2Q-1$ generate a unitary representation of the symmetric group $S_{3}$ on $\X$. 
\end{proof}


\subsection{Existence of a Hermitian torsion-free connection}
\label{subsec: existence}
In this section we provide a necessary and sufficient condition for the existence of a Hermitian torsion-free connection. Given a frame $(\omega_{j})$, we provide a  construction of such a connection. We will address the issue of uniqueness subsequently.
In Definition \ref{defn:Bram's-condition} we introduce a condition on the differential structure which ensures that the differential forms support an Hermitian torsion-free connection. 
\begin{defn} An Hermitian differential structure $(\Omega^1_\dee(\B),\dag,\Psi,\pairing{\cdot}{\cdot}_\B)$ is  \emph{concordant} if the submodules 
${\rm Im}(P)={\rm Im}(\Psi\otimes 1)$ and ${\rm Im}(Q)={\rm Im }(1\otimes\Psi)$
are concordant in $T^{3}_{\dee}$ in the sense of Definition \ref{def:concordant-pair}. 
\end{defn}
By Theorem \ref{thm:smo-intersect}, the above definition is equivalent to the requirement that for the pair of projections $(P,Q)=(\Psi\otimes 1,1\otimes\Psi)$ the $C^{*}$-norm limit $\Pi:=\lim_{n\to\infty}(PQ)^{n}$ exists and $\Pi(T^{3}_{\dee})\subset T^{3}_{\dee}$. Also, if $(\Omega^1_\dee(\B),\dag,\Psi,\pairing{\cdot}{\cdot}_\B)$ is concordant then $(1+\Pi-PQ)^{-1}\in \End^{*}_{\B}(T^{3}_{\dee})$.

\begin{example}[Compact Riemannian manifold]
\label{eg:classical-angle}
Let $(M,g)$ be a compact Riemannian manifold.
Then denote the flip map by
\begin{align*}
\sigma:\Omega^{1}(M)\otimes_{C^{\infty}(M)}\Omega^{1}(M)\to \Omega^{1}(M)\otimes_{C^{\infty}(M)}\Omega^{1}(M),\quad \omega\otimes \eta\mapsto  \eta\otimes\omega, 
\end{align*}
so that the symmetrisation map is $\Psi=\frac{1}{2}(1+\sigma)$. On the space of three-tensors $T^{3}(M)=(\Omega^{1}(M))^{\otimes 3}$, the permutation action of the group $S_{3}$ is generated by the operators $U:=\sigma\otimes 1=2P-1$ and $V:=1\otimes \sigma=2Q-1$. By Corollary \ref{cor: S3action} the Hermitian differential structure $(\Omega^1(M),\dag, \Psi, \pairing{\cdot}{\cdot}_{g})$ is concordant and the projection $\Pi$ exists. Note that we have 
\[
P(T^{3}_{\D})\cap Q(T^{3}_{\D}) = T^{3}(M)^{S_{3}}=\textnormal{Sym}^{3}(M),
\]
so that the range of $\Pi$ consists of the completely symmetric three-tensors. 
\end{example}

We introduce some notation suggested by the classical example above.
\begin{defn} 
\label{defn:symm}
Let $(\Omega^1_\dee(\B),\dag,\Psi,\pairing{\cdot}{\cdot}_\B)$ be a concordant Hermitian differential structure. The \emph{module of (smooth) symmetric 3-tensors}  is defined to be
\[
\textnormal{Sym}^{3}_{\dee}:={\rm Im}(\Psi\otimes 1)\cap {\rm Im }(1\otimes\Psi)=\Pi(T^{3}_{\dee})\subset T^{3}_{\dee}.
\]
\end{defn}

\begin{prop} 
\label{prop: omega-form}
Let $(\Omega^1_\dee(\B),\dag,\Psi,\pairing{\cdot}{\cdot}_\B)$ be a concordant Hermitian differential structure.
Write $P=\Psi\otimes 1$ and $Q=1\otimes\Psi$. Then for any Hermitian torsion-free right connection $\nablar:\Omega^{1}_{\dee}\to T^2_{\dee}$ and any frame $v=(\omega_{j})$ for $\Omega^{1}_{\dee}$ such that $\nablar=\nablar^{v}_{\dee}+\alphar(A_{v})$ we have
\[
(1-\Pi)A_{v}=-(1+\Pi-PQ)^{-1}(W+PW^{\dag})
=-(1+\Pi-QP)^{-1}(W^{\dag}+QW).
\]
Consequently, given two torsion-free Hermitian right connections $\nablar$ and $\nablar'$, we have that $(1-\Pi)\alphar^{-1}(\nablar-\nablar')=0$.
\end{prop}
\begin{proof}
The concordance assumption guarantees us that $\lim_{k\to\infty}(PQ)^{k}=\lim_{k\to\infty}(QP)^{k}=\Pi$ in norm, $\Pi(T^{3}_{\dee})\subset T^{3}_{\dee}$ and $(1+\Pi-PQ)^{-1}\in\End_{\B}^{*}(T^{3}_{\dee})$. 
Now we observe that
\[
W+PW^{\dag}=(1-P)W+P(1-Q)W^{\dag}=(1-P)(W-(1-Q)W^{\dag})+(1-Q)W^{\dag}\in\ker\Pi,
\]
From Equation \eqref{eq:series} we then deduce that for all $n$,
\[
(1-(PQ)^{n})A_{v}=-\sum_{k=0}^{n-1}(PQ-\Pi)^{k}(W+PW^{\dag}).
\] 
Analogous equations hold with $P,Q$ interchanged and the result follows when $n\to\infty$.
\end{proof}

\begin{corl}
\label{cor:nice-conn} 
Suppose there exists an Hermitian torsion-free connection on $\Omega^{1}_{\dee}$. Then for all frames $(\omega_{j})$ we have the equality
\begin{equation}
\label{eq: dagaligned}
(1+\Pi-PQ)^{-1}(W+PW^{\dag})
=(1+\Pi-QP)^{-1}(W^{\dag}+QW).
\end{equation}
\end{corl}

\begin{defn}
\label{defn:Bram's-condition}
An  Hermitian differential structure $(\Omega^1_\dee(\B),\dag,\Psi,\pairing{\cdot}{\cdot}_\B)$ is  $\dag$- concordant if it is concordant and
for some frame $(\omega_j)$ and $W=\sum_j\d(\omega_j)\ox\omega_j^\dag$ the Equation \eqref{eq: dagaligned} holds.
\end{defn}
\begin{rmk}
If we have a concordant Hermitian differential structure $(\Omega^1_\dee(\B),\dag,\Psi,\pairing{\cdot}{\cdot}_\B)$ and a frame $(\omega_j)$ for $\Omega^1_\dee(\B)$ such that $W=0$, the differential structure is $\dag$-concordant. In particular, if the frame is {\em exact}, so $\omega_j=\dee(b_j)$, $b_j\in\B$, then $\dag$-concordance holds.
\end{rmk}

We will now prove that the equality \eqref{eq: dagaligned} is not only necessary, but also sufficient for the existence of a Hermitian torsion-free connection. It is important to note that Equation \eqref{eq: dagaligned} involves the element $W$, which is frame dependent. We will show that if \eqref{eq: dagaligned} holds for some frame, then it holds for all frames. First we make a technical observation.

\begin{lemma}
\label{lem:affine} 
Let $(\Omega^1_\dee(\B),\dag,\Psi,\pairing{\cdot}{\cdot}_\B)$ be a concordant Hermitian differential structure and $E,F\in T_{\dee}^{3}(\B)$. Suppose that there exists $C=C^{\dag}$ such that 
$$
E-F=(1-PQ)C.
$$
Then $(1+\Pi-PQ)^{-1}E-(1+\Pi-QP)^{-1}E^{\dag}=(1+\Pi-PQ)^{-1}F-(1+\Pi-QP)^{-1}F^{\dag}.$
\end{lemma}
\begin{proof}
Since $(1-PQ)=(1-PQ)(1-\Pi)$ and $(1-QP)=(1-QP)(1-\Pi)$ we may assume $\Pi C=0$ and thus that $E-F=(1+\Pi-PQ)C$ and $E^{\dag}-F^{\dag}=(1+\Pi-QP)C$. Now
\begin{align*}
(1+\Pi-PQ)^{-1}E-(1+\Pi-QP)^{-1}E^{\dag}&\!=\!(1+\Pi-PQ)^{-1}F+C-(1+\Pi-QP)^{-1}F^{\dag}-C\\
&\!=\!(1+\Pi-PQ)^{-1}F-(1+\Pi-QP)^{-1}F^{\dag},
\end{align*}
as claimed.
\end{proof}

\begin{prop}
\label{prop:BC-FI}
The three tensor
\[
(1+\Pi-PQ)^{-1}(W+PW^{\dag})
-(1+\Pi-QP)^{-1}(W^{\dag}+QW)\in T^3_\dee
\]
is frame independent.
Hence the property of being $\dag$-concordant is frame independent, and so depends only on the inner product and the projection $\Psi$.
\end{prop}
\begin{proof}
Let $(\omega_j)$ and $(\eta_k)$ be two frames for $\Omega^1_\dee(\B)$. Denote by $W$ and $V$ the forms defined in Equation \eqref{eq: W} associated to $\omega_{j}$ and $\eta_{k}$ respectively. Then, using Corollary \ref{cor: change-of-frame},
\begin{align*}
W-V &=\sum_{j,k}(1-P)(\omega_j\ox\dee(\pairing{\omega_j}{\eta_k})\ox\eta_k^\dag).
\end{align*}
Now
\begin{align*}
W+PW^{\dag}-V-PV^{\dag}&=\sum_{j,k}((1-P)+P(1-Q))(\omega_j\ox\dee(\pairing{\omega_j}{\eta_k})\ox\eta_k^\dag)\\
&=(1-PQ)\Big(\sum_{j,k}\omega_j\ox\dee(\pairing{\omega_j}{\eta_k})\ox\eta_k^\dag\Big)
\end{align*}
and the element $\sum_{j,k}\omega_j\ox\dee(\pairing{\omega_j}{\eta_k})\ox\eta_k^\dag$ is $\dag$-invariant by Corollary \ref{cor: change-of-frame}. Then 
the statement follows by applying Lemma \ref{lem:affine}
and noting that $\Pi(W+PW^\dag)=0$.
\end{proof}
We can now characterise existence of torsion-free Hermitian connections for an Hermitian differential structure in terms of $\dag$-concordance. 
\begin{thm} 
\label{thm:existence}
Let $(\Omega^1_\dee(\B),\dag,\Psi,\pairing{\cdot}{\cdot}_\B)$ be a concordant Hermitian differential structure.
Then there exists a Hermitian torsion-free  right connection $\nablar:\Omega^{1}_{\dee}(\B)\to T^2_{\dee}(\B)$  if and only if $(\Omega^1_\dee(\B),\dag,\Psi,\pairing{\cdot}{\cdot}_\B)$ is $\dag$-concordant. 
\end{thm}
\begin{proof} 
Choose a right frame $v=(\omega_{j})$ and suppose there exists an Hermitian torsion-free connection $\nabla=\nabla^{v}+\alphar(A_{v})$. By Corollary \ref{cor:nice-conn}, the differential structure is $\dag$-concordant.

Conversely, if the Hermitian differential structure is $\dag$-concordant then for any frame the connection
\[
\nablar:=\nablar^{v}-\alphar((1+\Pi-PQ)^{-1}(W+PW^{\dag})),
\]
is Hermitian since $\nablar^{v}$ is Hermitian and then $\dag$-concordance gives
\[
\big((1+\Pi-PQ)^{-1}(W+PW^{\dag})\big)^\dag=(1+\Pi-QP)^{-1}(W^\dag+QW)=(1+\Pi-PQ)^{-1}(W+PW^{\dag}).
\]
Next observe that
\[
(1-P)(1+\Pi-PQ)^{-1}=1-P, 
\]
so that we have $(1-P)A_{v}=-W$, which by Lemma \ref{lem: torsion-free iff} proves that $\nablar$ is torsion-free.
\end{proof}
If we have a $\dag$-concordant Hermitian differential structure then for any frame $v=(\omega_j)$ we can define a connection form 
\begin{equation}
A_{v}=-\sum_{k=0}^{\infty}(PQ)^{k}(W+PW^{\dag})=-\sum_{k=0}^{\infty}(QP)^{k}(W^{\dag}+QW),
\label{eq:a-for-LC}
\end{equation}
and then $\nablar^v+\alphar(A_v)$ will be Hermitian and torsion-free.
In fact Equation \eqref{eq:a-for-LC} provides a practical computational tool as we will show in Section \ref{sec:examples} and subsequent work. A particularly tractable setting which arises for manifolds and their $\theta$-deformations is described  next.

\begin{prop}
\label{prop:suff-for-exist}
Let $(\Omega^1_\dee(\B),\dag,\Psi,\pairing{\cdot}{\cdot}_\B)$ be an Hermitian differential structure with $\mathcal{B}$ local. Suppose that $P=\Psi\ox1,Q=1\ox\Psi\in\End^{*}_{\B}(T^3_\dee)$ satisfy
\begin{equation*}
(2P-1)(2Q-1)(2P-1)=(2Q-1)(2P-1)(2Q-1).
\end{equation*}
Suppose further that $\Omega^1_\dee(\B)$ has a frame such that $W^{\dag}=(2P-1)(2Q-1)W.$
Then there exists an Hermitian torsion-free connection.
\end{prop}
\begin{proof}
The submodules $PT^3_\dee$ and $QT^3_\dee$ are concordant by Corollary \ref{cor: S3action}.
Thus we need to prove $\dag$-concordance.
Let $(\omega_j)$ be a frame such that  
$
W^{\dag}=(2P-1)(2Q-1)W.
$

Since $PW=0$ we then find $PW^{\dag}=2PQW$ and $QW=2QPW^{\dag}$, so that $(PQ)^{2}W=\frac{PQW}{4}$. By induction we then obtain, 
\begin{equation}
\label{eq: W-series}
(PQ)^{n}W=4^{-n+1}PQW,\quad (QP)^{n}W^\dag=4^{-n+1}QPW^\dag, \quad n\geq 1.
\end{equation} 
Therefore we end up with a geometric series whose sum is
\begin{align*}
(1+\Pi-PQ)^{-1}(W+&PW^{\dag})=(1+\Pi-PQ)^{-1}(W+2PQW)=\sum_{n\geq 0}(PQ)^{n}(W+2PQW)\\
&=W+3\sum_{n\geq 1}(PQ)^{n}W=W+3\sum_{n\geq 1}4^{-n+1}PQW=W+4PQW,
\end{align*}
and by the same token $(1+\Pi-QP)^{-1}(W^{\dag}+QW)=W^{\dag}+4QPW^{\dag}$. Thus to show that the  Hermitian differential structure is $\dag$-concordant, we need to show that 
\[
W+4PQW=W^{\dag}+4QPW^{\dag}.
\]
Since $QW=2QPW^{\dag}=4QPQW$ we find
\begin{align*}
W^{\dag}+4QPW^{\dag}&=(2P-1)(2Q-1)W+8QPQW\\
&=W+4PQW-2QW+8QPQW=W+4PQW,
\end{align*}
and thus the Hermitian differential structure is $\dag$-concordant.
Thus, for the frame $v=(\omega_{j})$ the Hermitian torsion-free right connection constructed in Theorem \ref{thm:existence} equals
\[
\nablar:=\nablar^{v}-\alphar((1+4PQ)W).\qedhere
\]
\end{proof}


\section{Uniqueness of Hermitian torsion-free connections}
\label{sec:unique}

Section \ref{subsec:RNDS} shows that for a $\dag$-concordant Hermitian differential structure the equations
\[
(1+\Pi-PQ)A_{v}=-W-PW^{\dag},\quad (1+\Pi-QP)A_{v}=-W^{\dag}-QW^{\dag}
\]
\emph{do not determine $A_{v}$ uniquely, but only up to a three-tensor in} $\textnormal{Sym}_{\dee}^{3}$, Definition \ref{defn:symm}. In the present section we address the question of uniqueness of Hermitian torsion-free connections.


\subsection{Uniqueness up to equivalence}

Given two Hermitian right connections
$
\nablar,\nablar':\Omega^{1}_{\dee}(\B)\to T^2_{\dee}(\B),
$
their difference
\[
\nablar-\nablar':\Omega^{1}_{\dee}(\B)\to T^2_{\dee}(\B),
\]
is a right $\B$-module map and determines a unique element $\alphar^{-1}(\nablar-\nablar')\in T^{3}_{\dee}(\B)$. Using this, we define an equivalence relation on the affine space of all connections on $\Omega^{1}_{\dee}(\B)$.
\begin{defn} 
Two Hermitian right connections
\[
\nablar,\nablar':\Omega^{1}_{\dee}(\B)\to T^2_{\dee}(\B),
\]
are said to be \emph{equivalent} $\mathrm{mod}\, \textnormal{Sym}^{3}_{\dee}$, written $\nablar\sim_{\Pi}\nablar'$ if $\alphar^{-1}(\nablar-\nablar')\in\textnormal{Sym}^{3}_{\dee}$.
\end{defn}
It is straightforward to check that $\sim_{\Pi}$ defines an equivalence relation on the space of Hermitian right connections. Also, $\nablar\sim_{\Pi}\nablar'$ if and only if $(1-\Pi)\alphar^{-1}(\nablar-\nablar')=0$. 
 Choosing a frame $v=(\omega_{j})$ and writing $\nablar=\nablar^{v}+\alphar(A_{v})$ we see that $\nablar\sim_{\Pi} \nablar'$ if and only if $(1-\Pi)(A_{v}-A'_{v})=0$ so that $(1-\Pi)A_{v}=(1-\Pi)A_{v}'$. Thus, the set of equivalence classes of $\sim_{\Pi}$ is in (non-canonical) bijection with $(\textnormal{Sym}^{3}_{\dee})^{\perp}$ via the map $[\nablar]_{\Pi}\mapsto (1- \Pi)A_{v}$.

In particular two connections are equivalent if and only if relative to any frame, their connection forms have \emph{the same} component in $(\textnormal{Sym}^{3}_{\dee})^{\perp}$. We arrive at our first relative uniqueness statement.

\begin{prop}
\label{lem: unique} 
Let $\nablar,\nablar':\Omega^{1}_{\dee}(\B)\to T^2_{\dee}(\B)$ be torsion-free Hermitian right connections. Then $\nablar\sim_\Pi\nablar'$.
\end{prop}
\begin{proof} 
Choose a frame and write
$\nablar=\nablar^{v}+\alphar(A_{v})$, $\nablar'=\nablar^{v}+\alphar(A_{v}')$,
so that 
\[
\nablar-\nablar'=\alphar(A_{v}-A_{v}').
\]
By Proposition \ref{prop: omega-form} it follows that $(1-\Pi)A_{v}=(1-\Pi)A_{v}',$ so $\nablar\sim_{\Pi}\nablar'$.
\end{proof}


\subsection{Reality and uniqueness of bimodule connections}

\label{sec:new-uniqueness}
The fundamental theorem of Riemannian geometry states that, on the space of sections $\Omega^{1}(M)=\Gamma(M,T^{*}M)$ of the cotangent bundle $T^{*}M$, there is a unique torsion-free metric compatible connection $\nabla^{LC}$, the \emph{Levi-Civita connection}. The cotangent bundle $T^{*}M\to M$ is a \emph{real} vector bundle, whereas our constructions, up to the present point, only pertain its complexification. In particular we have 
\[
\Omega^{1}(M)\otimes_{\mathbb{R}} \mathbb{C}=\Omega^{1}(M)\oplus i\Omega^{1}(M),
\]
which motivates the following definition.

\begin{defn} 
Let $(\Omega^1_\dee(\B),\dag)$ be a first order differential structure.
The real zero and one-forms are those $b\in\B$ and $\omega\in\Omega^1_\dee(\B)$  satisfying
\[
b=b^*\qquad \omega=-\omega^\dag.
\]
We denote the space of real one-forms by $\Omega^1_\dee(\B)^{\dag}$.
\end{defn}
\begin{rmk} The real zero and one-forms are (not inner-product) modules over the algebra $\B^J:=\left\{\sum_ia_i\ox b_i^\op\in\B\ox\B^\op:\sum_ia_i\ox b_i^\op=\sum_ib^*_i\ox a_i^{*\op}\right\}\subset\B\ox\B^\op$, where $\B^\op$ is the opposite algebra. See \cite{ConnesGrav}.
\end{rmk}
For manifolds,
recovering $T^{2}(M)$ from its complexification $T^{2}_{\mathbb{C}}(M):=T^{2}(M)\otimes_{\mathbb{R}}\mathbb{C}$ algebraically cannot be done by using the $\dag$-operation alone. The reason for this is that 
$(\omega\ox\rho)^\dag=\rho^*\ox\omega^*$ and the tensor factors are flipped. If
\[\sigma:T^{2}(M)\otimes\mathbb{C}\to T^{2}(M)\otimes\mathbb{C},\quad \omega\otimes \eta\mapsto \eta\otimes \omega,\]
denotes the standard flip map, then the subspace $T^{2}(M)\subset T^{2}_{\mathbb{C}}(M)$ is characterised by the property that $\dag\circ\sigma(\xi)=\xi$. Thus to obtain a sensible notion of real two tensors more generally, an abstract analogue of the flip map is necessary, and we now adopt the following notion from \cite[Section 8.1]{BMBook}. 
\begin{defn}
\label{def:mumble-real}
Let $(\Omega^1_\dee(\B),\dag)$ be a first order differential structure, and $\X$  a $\dag$-bimodule over the $*$-algebra $\B$.
A \emph{braiding} is an invertible bimodule map $\sigma:\X\ox_\B\Omega^1_\dee\to\Omega^1_\dee\ox_\B\X$ satisfying
\[
\sigma^{-1}\circ\dag=\dag\circ \sigma. 
\]
For a given braiding $\sigma$ on $T^2_\dee=\Omega^1_\dee\ox_\B\Omega^1_\dee$, we say that a two-tensor $\sum_i\omega_i\ox\rho_i$ is $\sigma$-real if
\begin{equation}
\dag\circ\sigma\Big(\sum_i\omega_i\ox\rho_i\Big)=\sum_i\omega_i\ox\rho_i.
\label{eq:real-two-form}
\end{equation}
We denote the space of $\sigma$-real two-tensors by $T^{2}_{\dee}(\B)^{\dag\circ \sigma}$.
\end{defn}
\begin{rmks} 
1. The $\sigma$-real two-tensors are again a module over $\B^J$. 
In the classical case of a Riemannian manifold, we can take $\sigma$ to be the flip
map and Definition \ref{def:mumble-real} captures the real two-forms. 

2. For a second order differential structure $(\Omega^1_\dee(\B),\dag,\Psi)$, the map $2\Psi-1$ will always satisfy the definition of braiding. That said, $2\Psi-1$ will typically not be compatible with connections in the way we require below. This occurs for instance in the example of the Podle\'{s} sphere, as we show in \cite{MRPods}.
\end{rmks}

\begin{lemma}
\label{lem: real-decomp}
Let $(\Omega^1_\dee(\B),\dag)$ be a first order differential structure and $\sigma:T^{2}_{\dee}(\B)\to T^{2}_{\dee}(\B)$ a braiding. Then we have decompositions
\[
\Omega^{1}_{\dee}(\B)=\Omega^{1}_{\dee}(\B)^{\dag}\oplus i\Omega^{1}_{\dee}(\B)^{\dag},\quad T^{2}_{\dee}(\B)
=T^{2}_{\dee}(\B)^{\dag\circ\sigma}\oplus iT^{2}_{\dee}(\B)^{\dag\circ\sigma},
\]
as $\mathbb{R}$-vector spaces (and as $\B^J$-modules).
\end{lemma}
\begin{proof} 
This follows directly from the fact that both $\dag$ 
and $\dag\circ\sigma$ are $\mathbb{R}$-linear idempotents.
\end{proof}

Bimodule connections have arisen multiple times in  the  literature, for instance \cite{BMBook,DHLS,DMM,DM} and more recently in \cite{CO,GKOSS}. Here we  link bimodule connections to conjugate connections, a notion called $*$-connection in \cite[p572ff]{BMBook}.
\begin{defn}
\label{def: dagbimodconn} 
Let $(\Omega^1_\dee(\B),\dag)$ be a first order differential structure and $\X$ be a $\dag$-bimodule over the $*$-algebra $\B$. A $\dag$-\emph{bimodule connection} on $\X$ is a pair $(\nablar,\sigma)$ where $\nablar:\X\to \X\otimes_{\B}\Omega^{1}_{\dee}$ is a right connection on $\X$ and $\sigma:\X\ox_\B\Omega^1_\dee\to\Omega^1_\dee\ox_\B\X$ is a braiding
such that $\sigma\circ\nablar=\nablal$, 
where $\nablal=-\dag\circ \nablar\circ\dag$ is the conjugate left connection.
\end{defn}
\begin{rmk}
\label{rmk: flippintwist}
As discussed in \cite{Maj-bimod}, \cite[p568]{BMBook}, if there exists such a $\sigma$, then it is unique. That is, if $(\nablar,\sigma)$ and $(\nablar,\sigma')$ are bimodule connections then $\sigma=\sigma'$. In fact, $\sigma$ is determined by the formula
\begin{equation}
\label{eq: connectiontwist}
\sigma^{-1}(\dee(b)\otimes x)=[\nablar,b]x,\qquad b\in\B,\,x\in\X.
\end{equation}
In case both $\X$ and $\Omega^{1}_\dee$ 
are central bimodules, this amounts to
\begin{align*}
\sigma^{-1}(\dee(b)\otimes x)&=[\nablar,b]x=\nablar(bx)-b\nablar(x)
=\nablar(xb)-\nablar(x)b=x\otimes \dee(b),
\end{align*}
that is, 
$\sigma:\Omega^{1}_{\dee}\otimes_{\B}\X\to 
\X\otimes_{\B}\Omega^{1}_{\dee}$ is forced 
to be the flip map $\omega\otimes x\mapsto x\otimes\omega$.
\end{rmk}

\begin{example}
Even for a manifold $M$ with Riemannian metric $g$, there are two Levi-Civita connections, left and right. We can see this already with the flat connection(s) on $\R^n$. Given a one-form 
$\omega=\sum_ja_j\mathrm{d}x^j=\sum_j\mathrm{d}x^ja_j\in T^*\R^n\ox\C$ with $a_j\in C_c^\infty(\R^n)$, 
the two connections are determined by
\[
\nablar(\omega)=\sum_{j,k}\mathrm{d}x^j\ox \mathrm{d}x^ka_{j,k},\qquad \nablal(\omega)=\sum_{j,k}a_{j,k}\mathrm{d}x^k\ox \mathrm{d}x^j,
\]
where $a_{j,k}=\partial_k(a_j)$.
By Remark \ref{rmk: flippintwist}, the map $\sigma$ can only be the flip map, and indeed the flip map satisfies $\sigma\circ \nablar=\nablal$.

With some elaboration, the same result holds for the Levi-Civita connection on  general pseudo-Riemannian manifolds. In local coordinates this follows from the flat calculation above, and the symmetry of the Christoffel symbols. 
\end{example}

The next result says that a $\dagger$-bimodule connection maps real one-forms to $\sigma$-real two-forms, and that this property characterises the class of such connections.

\begin{prop}
\label{lem:basic-real}
Let $(\Omega^1_\dee(\B),\dag)$ be a first order differential structure and $\sigma$ a braiding on $T^2_\dee$. A pair $(\nablar,\sigma)$ is a $\dag$-bimodule connection if and only if $\nablar:\Omega^{1}_{\dee}(\B)^{\dag}\to T^{2}_{\dee}(\B)^{\dag\circ\sigma}$.
Moreover, if $(\nablar,\sigma)$ is a $\dagger$-bimodule connection on $\Omega^1_\dee$ then
\begin{align}
\frac{1}{2}(1+\dag\circ  \sigma)\circ\nablar(\omega)=\begin{cases} 0
&\omega=\omega^\dag\\
\nablar(\omega) &\omega=-\omega^\dag
\end{cases}.
\label{eq:real-proj}
\end{align}
\end{prop}
\begin{proof} 
Suppose that $(\nablar,\sigma)$ a $\dagger$-bimodule 
connection on $\Omega^1_\dee$. 
Applying $\dag$ to both sides of $\sigma\circ\nablar=\nablal$ and
using the relation
$\nablar(\omega)^\dag=-\nablal(\omega^\dag)$ 
proves \eqref{eq:real-proj}, and thus implies that $\nablar$ 
restricts to $\nablar:\Omega^{1}_{\dee}(\B)^{\dag}\to T^{2}_{\dee}(\B)^{\dag\circ\sigma}$. That is, $\nablar$ maps real 
one-forms to $\sigma$-real two-tensors.

Conversely, suppose that $\nablar$ restricts to $\nablar:\Omega^{1}_{\dee}(\B)^{\dag}\to T^{2}_{\dee}(\B)^{\dag\circ\sigma}$ and 
$\nablal:=-\dag\circ \nablar\circ\dag$. Then for 
$\eta\in \Omega^1_\dee(\B)^{\dag}$ we have
\[
\nablar(\eta)=\sigma^{-1}\circ\dag\circ\nablar(\eta)=-\sigma^{-1}\circ\nablal(\eta^\dag)=\sigma^{-1}\circ\nablal(\eta).
\]
Hence $\sigma\circ\nablar=\nablal$ on $\Omega^1_\dee(\B)^{\dag}$. As both sides are $\C$-linear, by Lemma \ref{lem: real-decomp} we find that $(\nablar,\sigma)$ is a $\dag$-bimodule connection.
\end{proof}

By Remark \ref{rmk: flippintwist}, if both $\X$ 
and $\Omega^{1}_{\dee}$ are central bimodules, 
then $\sigma$ is forced to be the flip map. 
In general, for a $\B$-bimodule $\X$, we denote its \emph{centre} by
\begin{equation}
\mathcal{Z}(\X)=\left\{x\in \X: bx=xb\ \mbox{for all }b\in\B\right\}.
\label{eq:bimod-centre}
\end{equation}
Note that $\mathcal{Z}(\X)$ is a linear subspace of $\X$, 
but not a $\B$-sub-bimodule, only a $\mathcal{Z}(\B)$-bimodule.
\begin{lemma}
\label{lem:bimod-centre}
Let $\X,\Y$ be $\dag$-$\B$-bimodules. Then there are isomorphisms
of $\mathcal{Z}(\B)$-bimodules
\[
\overrightarrow{\alpha}:\mathcal{Z}(\Y\ox \X)\to\overleftrightarrow{\textnormal{Hom}}^{*}(\X,\Y),\,\,\overleftarrow{\alpha}:\mathcal{Z}(\Y\ox\X)\to \overleftrightarrow{\textnormal{Hom}}^{*}(\Y,\X)
\]
where $\overleftrightarrow{\textnormal{Hom}}^{*}(\X,\Y)$ denotes linear maps from $\X$ to $\Y$ which are both left and right adjointable.
\end{lemma}
\begin{proof}
Given a bimodule map $T:\X\to \Y$ and $a,b\in\B$, $x\in \X$ we have
$T(axb)=a(Tx)b$. So write $T=\sum_j\alphar(y_j\ox x_j^\dag)$ so that we are  encoding that $T$ is a right module map. Then since $T$ is also a left bimodule map we have
\begin{align*}
aT(x)&=a\sum_j\alphar(y_j\ox x_j^\dag)(x)
=\sum_j\alphar(ay_j\ox x_j^\dag)(x)
=T(ax)\\
&=\sum_j\alphar(y_j\ox x_j^\dag)(ax)=\sum_jy_j\pairing{x_j}{ax}_\B=\sum_jy_j\pairing{a^*x_j}{x}_\B
=\sum_jy_j\pairing{(x_j^\dag a)^\dag}{x}_\B\\
&=\sum_j\alphar(y_j\ox x_j^\dag a)(x)
\end{align*}
and hence $\sum_jy_j\ox x_j\in\mathcal{Z}(\Y\ox \X^*)$. Conversely, if 
$y\ox x^\dag\in \mathcal{Z}(\Y\ox \X^*)$ then $\overrightarrow{\alpha}(y\ox x^\dag)$ is right $\B$-linear automatically. For $c\in\B$ and $\tilde{x}\in
X$ we have 
\[
\overrightarrow{\alpha}(y\ox x^\dag)(c\tilde{x})=\overrightarrow{\alpha}(y\ox x^\dag c)(\tilde{x})=\overrightarrow{\alpha}(cy\ox x^\dag)(\tilde{x})=c\overrightarrow{\alpha}(y\ox x^\dag)(\tilde{x}),
\] 
and so $\overrightarrow{\alpha}(y\ox x^\dag)$ is a bimodule map. By \cite[Corollary 4.6]{KajPinWat}, it is also adjointable.
\end{proof}

\begin{lemma}
\label{lem: centre-bimodule-connection} 
Let $(\Omega^1_\dee(\B),\dag)$ be a second order differential structure, $\mathcal{X}$ a $\dag$-bimodule over the $*$-algebra $\B$ and finally let $\sigma:\X\otimes_{B}\Omega^{1}_{\dee}(\B)\to \Omega^{1}_{\dee}(\B)\otimes_{B}\X$ be
a  braiding. If $(\nablar,\sigma)$ and $(\nablar',\sigma)$ are bimodule connections then 
\[
\nablar-\nablar':\X\to \X\otimes_{\B}\Omega^{1}_{\dee}(\B),\qquad\nablal-\nablal':\X\to \Omega^{1}_{\dee}(\B)\otimes_{\B}\X,
\]
are bimodule maps. Equivalently by Lemma \ref{lem:bimod-centre}
\[
\alphar^{-1}(\nablar-\nablar'),\ \ \alphal^{-1}(\nablal-\nablal')\in \mathcal{Z}(\X\otimes_{\B}\Omega^{1}_{\dee}(\B)\otimes_{B}\X).
\]
\end{lemma}
\begin{proof}
Equation \eqref{eq: connectiontwist} says that for $x\in \X$ the identity
\[
\sigma^{-1}(\dee(a)\otimes x)=[\nablar,a]x=[\nablar',a]x,
\]
holds true, and so $[\nablar-\nablar',a]=0$, as claimed.
\end{proof}
The following condition guaranteeing uniqueness is analogous to \cite[Theorem 4.13]{BGJ2}.
\begin{thm}
\label{thm: uniqueness-twist} 
Let $(\Omega^1_\dee(\B),\dag,\Psi,\pairing{\cdot}{\cdot}_\B)$ be a $\dag$-concordant Hermitian differential structure with a braiding on two-tensors $\sigma:T^2_\dee\to T^2_\dee$  such that 
\begin{align*}
\alphar+\sigma^{-1}\circ\alphal:\mathcal{Z}(\textnormal{Sym}^{3}_{\dee})\to \overleftrightarrow{\textnormal{Hom}}^{*}(\Omega^{1}_{\dee}, \Omega^{1}_{\dee}\otimes_{\B}\Omega^{1}_{\dee}),
\end{align*}
is injective. Suppose that $(\nablar,\sigma)$ and $(\nablar',\sigma)$ are $\dag$-bimodule connections on $\Omega^{1}_{\dee}$. If both $\nablar$ and $\nablar'$ are Hermitian and torsion-free then $\nablar=\nablar'$.
\end{thm}
\begin{proof} 
Let $(\omega_{j})$ be a frame for $\Omega^{1}_{\dee}$ and write
\[
\nablar=\nablar^{v}+\alphar(A_{v}),\quad \nablar'=\nablar^{v}+\alphar(A_{v}').
\]
Since $\nablar,\nablar'$ are torsion-free, by Proposition \ref{prop: omega-form} we have $(1-\Pi)(A_{v}-A_{v}')=0$, so that $A_{v}-A_{v}'=\Pi(A_{v}-A_{v}').$
Thus, by Lemma \ref{lem: centre-bimodule-connection} we have that \begin{equation}
\label{eq: difference} 
A_{v}-A_{v}'=\Pi(A_{v}-A_{v}')\in\mathcal{Z}(\textnormal{Sym}^{3}_{\dee}).
\end{equation} 
Now since $\nablar,\nablar'$ are Hermitian $\dag$-bimodule connections relative to $\sigma$ we have
\begin{align*}
\nablar^{v}+\alphar(A_{v})&=\sigma^{-1}\circ\nablal^{v}-\sigma^{-1}\circ\alphal(A_{v})\\
\nablar^{v}+\alphar(A_{v}')&=\sigma^{-1}\circ\nablal^{v}-\sigma^{-1}\circ\alphal(A_{v}'),
\end{align*}
by Corollary \ref{lem:left-right}. Subtracting these two gives
\[
(\alphar+\sigma^{-1}\circ\alphal)(A_{v}-A_{v}')=0.
\]
Combining this with Equation \eqref{eq: difference} and the injectivity of $\alphar+\sigma^{-1}\circ\alphal$ on $\mathcal{Z}(\textnormal{Sym}^{3}_{\dee})$ we find that $A_{v}-A_{v}'=0$ as desired.
\end{proof}

\begin{corl}
\label{eq:affine-space}
Let $(\Omega^1_\dee(\B),\dag,\Psi,\pairing{\cdot}{\cdot}_\B)$ be a $\dag$-concordant Hermitian differential structure and $\sigma:T^2_\dee\to T^2_\dee$ a braiding.
Let $(\nablar,\sigma)$ be a $\dag$-bimodule connection and $A=A^\dag\in T^3_\dee$ be such that $\alphar(A):\Omega^{1}_{\dee}(\B)^{\dag}\to T^2_\dee(\B)^{\dag\circ\sigma}$. Then $(\nablar+\alphar(A),\sigma)$ is a $\dag$-bimodule connection.
Hence if $\alphar+\sigma^{-1}\circ\alphal$ is injective on $\mathcal{Z}({\rm Sym}^3_\dee)$ then no $A=A^\dag\in \mathcal{Z}({\rm Sym}^3_\dee)$ maps real one-forms to $\sigma$-real two-tensors.
\end{corl}
\begin{proof} Let $A\in T^3_\dee$ and $\eta\in \Omega^1_\dee$.
It follows from Lemma \ref{lem:useful-iso} that $\dag\circ \alphar(A)(\eta)=\alphal(A^\dag)(\eta^\dag)$. So if $A=A^\dag$, $\alphar(A)$ maps real one-forms to $\sigma$-real two-tensors  and $\eta=-\eta^\dag$, applying $\sigma^{-1}$ to both sides yields 
\[
\alphar(A)(\eta)=-\sigma^{-1}\alphal(A)(\eta).
\]
As both $\alphar(A)$ and $\sigma^{-1}\circ\alphal(A)$ are complex linear we find that
$\alphar(A)+\sigma^{-1}\circ\alphal(A)=0$. Then for $\eta=-\eta^\dag\in\Omega^1_\dee$ we have
\begin{align*}
\sigma^{-1}\circ\dag\big(\nablar(\eta)+\alphar(A)(\eta)\big)
&=-\nablar(\eta^\dag)+\sigma^{-1}\circ\alphal(A)(\eta^\dag)
=\nablar(\eta)-\sigma^{-1}\circ\alphal(A)(\eta)\\
&=\nablar(\eta)+\alphar(A)(\eta)
\end{align*}
and so $\nablar+\alphar(A)$ maps real one-forms to $\sigma$-real two-tensors.
The final statement follows since any three-tensor $A=A^\dag$ in $\mathcal{Z}({\rm Sym}^3_\dee)$ which maps real one-forms to $\sigma$-real two-tensors is zero by injectivity of $\alphar+\sigma^{-1}\circ\alphal$.
\end{proof}

In the case that we have a unique Hermitian torsion-free connection on $\Omega^1_\dee$ relative to a given braiding $\sigma$, we do not in general have a closed formula for it. Nonetheless, the proof of Theorem \ref{thm: uniqueness-twist} does give us an approach to obtaining the connection form in examples.
\begin{corl}
\label{corl:formula-LC}
Let $(\Omega^1_\dee(\B),\dag)$ be a $\dag$-concordant Hermitian differential structure, $\nabla^{0}=\nablar^{v}+\alphar(A_{v})$ a Hermitian torsion-free right connection relative to a frame $v$ and satisfying $\Pi A_{v}=0$ and $\sigma$ a braiding. If a Hermitian torsion-free $\dag$-bimodule connection $(\nablar^{LC},\sigma)$ exists, then it is of the form $\nablar^{LC}=\nabla^{0}+\alphar(B)$
for some $B=B^\dag\in T^3_\dee$ and $B$ satisfies
\begin{equation}
(\alphar+\sigma^{-1}\circ\alphal)(B)=\sigma^{-1}\nablal^0-\nablar^0.
\label{eq:LC-vs-C}
\end{equation}
\end{corl}
\begin{proof}
With
$
\nablar^{LC}=\nablar^0+\alphar(B)
$
the bimodule condition gives
\begin{equation*}
\nablar^{LC}=\sigma^{-1}\circ\nablal^{LC}=\sigma^{-1}\circ\nablal^0-\sigma^{-1}\circ\alphal(B).
\end{equation*}
Subtracting these two expressions for $\nablar^{LC}$ gives the result.
\end{proof}
The right hand side of Equation \eqref{eq:LC-vs-C} is determined, so the ``correction term'' $B$ can likely be computed. More examples will refine this approach, as currently we have no example for which $B\neq0$.

\section{The Levi-Civita connection for $\theta$-deformations of Riemannian manifolds}
\label{sec:examples}

Let $(M,g)$ be a compact Riemannian manifold such that $\mathbb{T}^{2}\subset\textnormal{Isom}(M,g)$. Let $\slashed{S}$ be a Dirac bundle on $M$ \cite{LM}, and $(C^{\infty}(M), L^{2}(M,\slashed{S}), \slashed{D})$ the associated Dirac spectral triple. We follow \cite{CL} in their construction of the isospectrally deformed spectral triple $(C^{\infty}(M), L^{2}(M,\slashed{S}), \slashed{D})$. In order to keep computations tractable, we only deal with $\mathbb{T}^2$ actions, but we expect that our results extend to $\mathbb{T}^n$-deformations using \cite{BGJ2,Cac}.

The existence and uniqueness of the Levi-Civita connection for $\theta$-deformations of free torus actions was proved in \cite{BGJ2}. The freeness requirement restricts the examples to principal bundles. For a free action, the bimodule $\Omega^{1}_{\slashed{D}}(M_{\theta})$ of one-forms on the deformed manifold $M_{\theta}$ is generated by its center. This property guarantees the existence of a flip map
\[\sigma:T^{2}_{\slashed{D}}(M_{\theta})\to T^{2}_{\slashed{D}}(M_{\theta}), \]
given by $\sigma(\omega\otimes \eta):=\eta\otimes\omega$, whenever $\omega,\eta$ are central one forms (\cite[Theorem 6.1]{Skeide}). The fact that $\Omega^{1}_{\slashed{D}}(M_{\theta})$ is generated by its center then guarantees that $\sigma$ extends to a bimodule map on all of $T^{2}_{\slashed{D}}(M_{\theta})$.

Here we extend their results to any isometric torus action, and we generalise the map $\sigma$ above to this setting in Lemma \ref{lem:braiding-theta}. Examples of non-free $\mathbb{T}^{2}$-actions on compact manifolds are ubiquitous: for $n\geq 2$, the unit sphere $\mathbb{S}^{2n-1}\subset\mathbb{C}^{n}$ admits the $\mathbb{T}^{2}$-action
\[(s_{1},s_{2})\cdot (z_{1},\cdots,z_{n}):=(s_{1}w_{1},s_{2}w_{2},z_{3},\cdots, z_{n}),\quad s=(s_{1},s_{2})\in\mathbb{T}^{2},\quad z=(z_{1},\cdots,z_{n})\in\mathbb{S}^{2n-1}.\]
which is not free whenever $n\geq 3$. 
\subsection{Isospectral deformations}
We first develop some general notation. Let $H$ be a Hilbert space carrying a unitary representation $U:\mathbb{T}^{2}\to \mathbb{B}(H)$ and denote by 
$$
\alpha:\mathbb{T}^{2}\to \textnormal{Aut}(\mathbb{B}(H)),\quad \alpha_{s}(T):=U(s)TU(s)^{*},\quad s\in \mathbb{T}^2,
$$
the associated group of $*$-automorphisms of $\mathbb{B}(H)$. Here $T\mapsto T^*$ is the usual operator adjoint. The representation $U$ is necessarily of the form $U(s)=e^{is_{1}p_{1}+is_{2}p_{2}}$ where the $p_{i}$ are the self-adjoint generators of the one-parameter groups associated to the coordinates $s_1,s_2$ of $\mathbb{T}^{2}$. We write
\[
C^{\infty}_{\alpha}(H):= \left\{T\in \mathbb{B}(H): s\mapsto \alpha_{s}(T)\,\,\textnormal{is smooth for the norm topology}\right\},
\]
for the algebra of $\alpha$-smooth vectors. The torus action induces a $\mathbb{Z}^{2}$-grading on the algebra $C^{\infty}_{\alpha}(H)$ given by
\begin{equation}
\label{eq: homseries}
C^{\infty}_{\alpha}(H):=\bigoplus_{n=(n_{1},n_{2})\in\mathbb{Z}^{2}}C^{\infty}_{\alpha}(H)_{(n_1,n_2)},\quad T=\sum_{n=(n_{1},n_{2})}T_{n_1,n_2},
\end{equation}
and we refer to the operators $T_{n_1,n_2}$ as the \emph{homogeneous components} of $T\in C^\infty_\alpha(H)$. The family $T_{n_1,n_2}$ is of rapid decay in $(n_1,n_2)$ so that the series \eqref{eq: homseries} is norm convergent. We note that to an arbitrary $T\in\mathbb{B}(H)$ we can associate homogeneous components $T_{n_1,n_2}$ but the associated series may not converge. For a homogeneous operator $T$ we denote its bidegree by $n(T)=(n_1(T),n_2(T))$. In general the bidegree of homogeneous operators $S,T$ satisfies
\[
n(ST)=n(S)+n(T),\quad n(T^{*})=-n(T).
\]

For the torus equivariant spectral triple $(C^{\infty}(M), L^{2}(M,\slashed{S}), \slashed{D})$ the operator $\slashed{D}$ is of degree $(0,0)$ and for all $s\in \mathbb{T}^2$ and $T\in C^\infty_\alpha(H)$ we have
\[[\slashed{D},U(s)]=0,\quad [\slashed{D},\alpha_{s}(T)]=\alpha_{s}([\slashed{D},T]).\]
In particular if $T$ is homogeneous then $n([\slashed{D},T])=n(T)$.
We now fix $\lambda=e^{2\pi i\theta}\in \mathbb{T}\subset\mathbb{C}$ and define, for a pair $(S,T)$ of homogeneous operators
\begin{equation}
\label{eq: theta-cocycle}
\Theta(S,T):=\lambda^{n_{2}(S)n_{1}(T)-n_{2}(T)n_{1}(S)}=\Theta(n(S),n(T)).
\end{equation}
The last equality indicates that $\Theta(S,T)$ depends only on the degree of $S,T$, and from this point of view defines a 2-cocycle on $\Z^2=\widehat{\mathbb{T}^2}$, \cite{Cac}.
In particular we have the identities
\begin{equation}
\Theta(T,S)=\overline{\Theta(S,T)}=\Theta(S,T)^{-1},\quad \Theta(S_{1},T)\Theta(S_{2},T)=\Theta(S_{1}S_{2},T).
\label{eq:theeta}
\end{equation}
Following \cite{CL}, we define the \emph{left-twist} of an operator $T\in C^{\infty}_{\alpha}(H)$ by $L(T)=T\lambda^{n_2(T) p_1}$.
For homogeneous $S,T\in\mathbb{B}(H)$ define a new multiplication $*$ and adjoint $\dag$ by
\begin{equation}
S*T:=\lambda^{n_{2}(S)n_{1}(T)}ST,\quad S^{\dag}:=\lambda^{n_1(S)n_2(S)}S^{*}
\label{eq:prod-adj}
\end{equation}
so that $L(S*T)=L(S)L(T)$ and $L(S^{\dag})=L(S)^{*}$. These equations equip $C^{\infty}_{\alpha}(H)$ with a new $*$-algebra structure for which $L$ is a $*$-representation. 
\begin{lemma}
\label{lem: commutation-in-deformation}
Suppose that $S,T\in\mathbb{B}(H)$ are both homogeneous and $ST=TS$. Then $S*T=\Theta(S,T)T*S$ and hence $L(S)L(T)=\Theta(S,T)L(T)L(S)$.
\end{lemma}
\begin{proof} Straightforward verification using the definitions.
\end{proof}

Specialising to $H=L^{2}(M,\slashed{S})$, the $C^{*}$-algebra $C(M_{\theta})$ is the closure of $L(C^{\infty}(M))$ inside $\mathbb{B}(H)$ and the main result of \cite[Section 5]{CL} is that $(C^{\infty}(M_{\theta}),L^{2}(M,\slashed{S}),\slashed{D})$ is a spectral triple. 
Using the deformed spectral triple $(C^{\infty}(M_{\theta}),L^{2}(M,\slashed{S}),\slashed{D})$, we can define the one-forms
\[
\Omega^{1}_{\slashed{D}}(M_{\theta})=\Big\{\sum_i L(f_{i})[\slashed{D},L(g_i)]: f_{i},g_{i}\in C^{\infty}(M_{\theta})\Big\}.
\]
Note that we can always write a 1-form as a sum of homogeneous components. The $\mathbb{T}^{2}$-action and bidegree extend to the tensor modules $T^{k}_{\slashed{D}}(M_{\theta})$ via
\[
\alpha_{t}(\omega_{1}\otimes\cdots\otimes\omega_{k}):= \alpha_{t}(\omega_{1})\otimes\cdots\otimes\alpha_{t}(\omega_{k}),\quad n(\omega_{1}\otimes\cdots\otimes\omega_{k}):=\sum_{i=1}^{k}n(\omega_{i}).
\]
These definitions, together with Equation \eqref{eq: theta-cocycle}, allow us to define $\Theta(x,y)$ for any pair of homogeneous tensors $x\in T^{n}_{\slashed{D}}(M)$ and $y\in T^{m}_{\slashed{D}}(M)$. 
\subsection{Braiding, junk projection and exterior derivative}
We now construct a $\dag$-concordant Hermitian differential structure for the spectral triple $(C^{\infty}(M_{\theta}),L^{2}(M,\slashed{S}),\slashed{D})$. We have the one-forms and $\dag$, and next consider the braiding.
\begin{lemma}
\label{lem:braiding-theta}
The map $\sigma_{\theta}:T^{2}_{\slashed{D}}(M_{\theta})\to T^{2}_{\slashed{D}}(M_{\theta})$ defined by
\[
\sigma_{\theta}(\omega\otimes \eta):=\Theta(\omega,\eta)(\eta\otimes\omega),
\]
extends to a well-defined bimodule map satisfying $\sigma^{2}_{\theta}=1$ and $\sigma_{\theta}\circ\dag=\dag\circ\sigma_{\theta}$. On $T^{3}_{\slashed{D}}(M_{\theta})$ the identity
\begin{equation}
\label{YangBaxter}
(\sigma_{\theta}\otimes 1)(1\otimes \sigma_{\theta})(\sigma_{\theta}\otimes 1)=(1\otimes \sigma_{\theta})(\sigma_{\theta}\otimes 1)(1\otimes \sigma_{\theta}),
\end{equation}
holds true. In addition $\sigma_\theta:T^{2}_{\slashed{D}}(M_{\theta})\to T^{2}_{\slashed{D}}(M_{\theta})$ is self-adjoint for any inner product on $T^{2}_{\slashed{D}}(M_{\theta})$ induced by an inner product on $\Omega^1_{\slashed{D}}(M_\theta)$.
\end{lemma}
\begin{proof}
We need to prove that for all one-forms $\omega,\eta\in \Omega^{1}_{\slashed{D}}(M_{\theta})$ and $a\in C^\infty(M_\theta)$ we have
\begin{equation}
\label{eq: well-defined-balanced}
\sigma_{\theta}(\omega * a\otimes \eta)=\sigma_{\theta}(\omega\otimes a*\eta).
\end{equation} First observe that $\omega * a\otimes \eta$ can be written as a sum
\[
\omega * a\otimes \eta=\sum_{i,j,k}\omega_i * a_j\otimes \eta_{k},
\]
where all $\omega_i,a_j,\eta_k$ are homogeneous, and so it suffices to prove \eqref{eq: well-defined-balanced} in the case all elements are homogeneous. So assuming that each of $\omega,\eta$ and $a$ are homogeneous, we have
\begin{align*}
\sigma_{\theta}(\omega * a\otimes \eta)
&=\Theta(\omega * a,\eta)\eta\otimes \omega * a
=\Theta(\omega,\eta)\Theta (a,\eta)\eta\otimes \omega * a\\
&=\Theta(\omega,\eta)\Theta (a,\eta)\Theta (a,\omega)\eta\otimes a*\omega=\Theta(\omega,\eta)\Theta (\omega,a)\Theta (a,\eta)\eta*a\otimes \omega\\
&=\Theta(\omega,\eta)\Theta (\omega,a)a*\eta\otimes \omega
=\Theta(\omega,a*\eta)a*\eta\otimes \omega\\
&=\sigma_{\theta}(\omega\otimes a*\eta),
\end{align*}
which proves well-definedness. Commutation of $\sigma_{\theta}$ with $\dag$ follows from
\begin{equation}
(\sigma_{\theta}\big((\omega\otimes\eta)^{\dag}\big))^{\dag}
=\big(\sigma_{\theta}(\eta^{\dag}\otimes\omega^{\dag})\big)^{\dag}
=(\Theta(\eta^{\dag},\omega^{\dag})\omega^{\dag}\otimes\eta^{\dag})^{\dag}
=\Theta(\omega,\eta)\eta\otimes\omega
=\sigma_{\theta}(\omega\otimes\eta).
\label{eq:bob}
\end{equation}
Next we show $\sigma_{\theta}$ is a right module map, where it again suffices to consider homogeneous elements, and then
\begin{align}
\sigma_{\theta}(\omega \otimes \eta* a)
&=\Theta(\omega,\eta*a)\eta*a\otimes\omega
=\Theta(\omega,\eta)\theta(\omega,a)\eta\otimes a*\omega\nonumber\\
&=\Theta(\omega,\eta)\eta\otimes \omega *a
=\sigma_{\theta}(\omega\otimes\eta)*a.
\label{eq:alice}
\end{align}
Combining Equations \eqref{eq:alice} with  \eqref{eq:bob} establishes that $\sigma_{\theta}$ is a left module map as well.
The identity \eqref{YangBaxter} follows by comparing the two identities
\begin{align*}
(\sigma_{\theta}\otimes 1)(1\otimes \sigma_{\theta})(\sigma_{\theta}\otimes 1)(\omega\otimes\eta\otimes\rho)
&=\Theta(\omega,\eta)\Theta(\omega,\rho)\Theta(\eta,\rho)(\rho\otimes\eta\otimes\omega)\\
(1\otimes \sigma_{\theta})(\sigma_{\theta}\otimes 1)(1\otimes \sigma_{\theta})(\omega\otimes\eta\otimes\rho)
&=\Theta(\eta,\rho)\Theta(\omega,\rho)\Theta(\omega,\eta)(\rho\otimes\eta\otimes\omega).
\end{align*}

Now we address the self-adjointness of $\sigma_\theta$. Given $\omega_j,\eta_j\in \Omega^{1}_{\slashed{D}}(M_{\theta})$, $j=1,2$, and any $(\cdot|\cdot)$ $C^\infty(M_\theta)$-valued right inner product on $\Omega^1_{\slashed{D}}(M_\theta)$, and $(\cdot|\cdot)_{T^2_{\slashed{D}}}$ be the induced inner product on $T^2_{\slashed{D}}(M_\theta)$. Then we have
\begin{align*}
(\sigma_\theta(\omega_1\ox\omega_2)|\eta_1\ox\eta_2)_{T^2_{\slashed{D}}}
&=\overline{\Theta(\omega_1,\omega_2)}(\omega_1|(\omega_2|\eta_1)\eta_2)\\
&=\Theta(\omega_2,\omega_1)\Theta((\omega_2|\eta_1),\eta_2)(\omega_1|\eta_2)(\omega_2|\eta_1)\\
&=\Theta(\omega_2,\omega_1)\Theta(\eta_1,\eta_2)\overline{\Theta(\omega_2,\eta_2)}(\omega_2(\eta_2|\omega_1)|\eta_1)\\
&=\Theta(\omega_2,\omega_1)\Theta(\eta_1,\eta_2)\Theta(\eta_2,\omega_2)\overline{\Theta(\omega_2,(\eta_2|\omega_1))}((\eta_2|\omega_1)\omega_2|\eta_1)\\
&=\Theta(\eta_1,\eta_2)(\omega_1\ox\omega_2|\eta_2\ox\eta_1)_{T^2_{\slashed{D}}}\\
&=(\omega_1\ox\omega_2|\sigma_\theta(\eta_1\ox\eta_2))_{T^2_{\slashed{D}}}.\qedhere
\end{align*}
\end{proof}
Similar to the classical symmetrisation, the braiding for $\theta$-deformed manifolds is closely related to the junk projection.
\begin{lemma}
\label{lem: theta-junk-exterior} 
Let $\Psi_{\theta}:=\frac{1+\sigma_{\theta}}{2}$. Then $\Psi_{\theta}$ is idempotent and there is an equality of bimodules
\[JT^{2}_{\slashed{D}}(M_{\theta})=\Psi_{\theta}(T^{2}_{\slashed{D}}(M_{\theta})).\]
The associated differential $\dt:\Omega^{1}_{\slashed{D}}(M_{\theta})\to T^{2}_{\slashed{D}}(M_{\theta})$ is of degree $0$. That is, $\dt$ preserves the bidegree of homogeneous forms.
\end{lemma}
\begin{proof} Define 
\[
\rho_{u}:C^{\infty}(M_{\theta})\otimes C^{\infty}(M_{\theta})\to C^{\infty}(M_{\theta})\otimes C^{\infty}(M_{\theta}),
\]
on homogeneous elements by $\rho_{u}(f\otimes g):=\Theta(f,g) g\otimes f$. Then $\rho_{u}$ preserves the universal one-forms $\Omega^{1}_{u}(M_{\theta})=\ker m\subset C^{\infty}(M_{\theta})^{\otimes 2}$ since
\begin{align*}
\sum f_{i}\otimes g_{i}\in \ker m\Leftrightarrow \sum f_{i}*g_{i}=0\Leftrightarrow \sum \Theta(f_{i}, g_i)g_{i}*f_{i}=0\Leftrightarrow \sum \rho_{u}(f_{i}\otimes g_{i})\in\ker m.
\end{align*}
Next, we have $\pi_{\slashed{D}}\circ\rho_{u}=-\pi_{\slashed{D}}$, since for $\omega:=\sum f_i\otimes g_i\in\Omega^{1}_u(M_{\theta})$ we have
\begin{align*}
\pi_{\slashed{D}}\circ\rho_{u}(\omega)&=\sum\Theta(f_{i},g_{i})\sum L(g_{i})[\slashed{D},L(f_i)]
=-\sum\Theta(f_{i},g_{i})[\slashed{D},L(g_i)]L(f_{i})\\
&=-\sum L(f_{i})[\slashed{D},L(g_i)]=-\pi_{\slashed{D}}(\omega).
\end{align*}
Since $\rho_{u}^{2}=1$ the endomorphisms $P_{\pm}=\frac{1\pm \rho_u}{2}$ are idempotent, with sum $P_{+}+P_{-}=1$. We have shown that ${\rm Im}(P_+)={\rm Im}((1+\rho_u)/2)\subset \ker(\pi_{\slashed{D}})$. Now we claim that $\pi_{\slashed{D}}(\delta\ker\pi_{\slashed{D}})=\pi_{\slashed{D}}(\delta P_{+}\ker\pi_{\slashed{D}})$. By \cite[Lemma 5.4]{MRCurve}, the linear map
defined on the tensor product of homogenous $\omega,\eta$ by the formula
\[
T_{\theta}^{\Omega^{1}_{\slashed{D}},\Omega^{1}_{\slashed{D}}}:T^{2}_{\slashed{D}}(M)\rightarrow T^{2}_{\slashed{D}}(M_{\theta}),\quad \omega\otimes\eta\mapsto \lambda^{-n_{2}(\omega)n_{1}(\eta)}\omega\otimes \eta,
\]
is a linear isomorphism of vector spaces. Now let $\sum_{i} f_{i}\otimes g_{i}\in \ker\pi_{\slashed{D}}$ and, since $\slashed{D}$ is of degree $0$, we may without loss of generality assume that all $f_{i}$ are homogeneous of the same bidegree, and similarly for the $g_{i}$. Then we compute
\begin{align*}
\pi_{\slashed{D}}\left(\delta P_{-}\left(\sum_{i} f_{i}\otimes g_{i}\right)\right)&=\sum_{i}[\slashed{D},L(f_{i})]\otimes [\slashed{D},L(g_{i})]-\Theta(f_{i},g_{i})[\slashed{D},L(g_{i})]\otimes [\slashed{D},L(f_{i})]\\
&=\lambda^{n_{2}(f_{i})n_{1}(g_{i})}T_{\theta}^{\Omega^{1}_{\slashed{D}},\Omega^{1}_{\slashed{D}}}\left(\sum_{i}[\slashed{D},f_{i}]\otimes[\slashed{D},g_{i}]-[\slashed{D},g_{i}]\otimes [\slashed{D},f_{i}]\right)=0.
\end{align*} 
The last equality follows since $\sum_{i}[\slashed{D},f_{i}]\otimes[\slashed{D},g_{i}]-[\slashed{D},g_{i}]\otimes [\slashed{D},f_{i}]=0$ as it is the Clifford representation of the two form $\sum_{i}\mathrm{d}(f_{i}\mathrm{d}g_{i})$ which is $0$ since $\sum_{i}f_{i}\otimes g_{i}\in\ker\pi_{\slashed{D}}$.
Therefore $\pi_{\slashed{D}}(\delta P_-\ker(\pi_{\slashed{D}}))=0$ and so
\begin{align*}
JT^{2}_{\slashed{D}}(M_{\theta})=\pi_{\slashed{D}}(\delta P_{+}\ker\pi_{\slashed{D}}),
\end{align*}
and 
\begin{align*}
\pi_{\slashed{D}}(\delta(f\otimes g+\rho_{u}(f\otimes g)))&=[\slashed{D},L(f)]\otimes [\slashed{D},L(g)]+\Theta(f,g)[\slashed{D},L(g)]\otimes [\slashed{D},L(f)]\\
&=2\Psi_{\theta}([\slashed{D},L(f)]\otimes [\slashed{D},L(g)]),
\end{align*}
which proves the claim. The fact that $\dt$ preserves the bidegree of homogeneous forms follows from the fact that both $\slashed{D}$ and $\Psi_{\theta}$ are of degree $0$ for the $\mathbb{T}^{2}$ action.
\end{proof}
\subsection{Inner product and invariant frames}
We wish to equip $\Omega^{1}_{\slashed{D}}(M_{\theta})$ with an inner product. To this end we deform the complexification of the classical inner product on $\Omega^{1}(M)$ used to construct the Dirac operator. This inner product will coincide with the metric of \cite[Proposition 3.12]{BGJ2}. Previously, $\theta$-deformed inner products have been considered in \cite{BLS,Cac}. For homogeneous forms we define the $\theta$-deformed right inner product
\begin{equation}
\label{eq: theta-inner-product}
\pairing{\omega}{\eta}_{\theta}:=\lambda^{(n_{1}(\omega)-n_1(\eta))n_{2}(\omega)}\pairing{\omega}{\eta},\qquad \omega,\eta\in \Omega^1_{\slashed{D}}(M_\theta).
\end{equation}
\begin{lemma}
\label{lem:inn-prod}
For $\omega,\eta\in\Omega^1_{\slashed{D}}(M_\theta)$ and $a\in C^\infty(M_\theta)$, the inner product \eqref{eq: theta-inner-product} satisfies
\begin{align*}
\pairing{\omega}{\eta}_{\theta}=(\pairing{\eta}{\omega}_{\theta})^{\dag},\  \pairing{\omega}{\eta *a}_{\theta}&=\pairing{\omega}{\eta}_{\theta} * a, \  \pairing{\omega}{\eta}_{\theta}=
\Theta(\omega,\eta)\pairing{\eta^{\dag}}{\omega^{\dag}}_{\theta}.
\end{align*}
Also $n(\pairing{\omega}{\eta}_{\theta})=n(\eta)-n(\omega)$ and $\Theta(S,\pairing{\omega}{\eta}_\theta)=\Theta(S,\omega)^{-1}\Theta(S,\eta)$ for any $S\in  C^\infty_\alpha(H)$. Both $\sigma_\theta$ and $\Psi_\theta$ are self-adjoint with respect to the induced inner product on $T^2_{\slashed{D}}(M_\theta)$.
\end{lemma}
\begin{proof} These are all straightforward algebraic verifications, and for simplicity we suppose that $\omega,\eta$ and $a$ are homogeneous. The first follows from the definition of the adjoint \eqref{eq:prod-adj} via
\begin{align*}
\pairing{\omega}{\eta}_{\theta}=\lambda^{(n_{1}(\omega)-n_1(\eta))n_{2}(\omega)}\pairing{\omega}{\eta}
=\lambda^{(n_{1}(\omega)-n_1(\eta))n_{2}(\omega)}\lambda^{-n_{1}(\pairing{\omega}{\eta})n_{2}(\pairing{\omega}{\eta})}\pairing{\eta}{\omega}^{\dag}=(\pairing{\eta}{\omega}_{\theta})^{\dag}.
\end{align*}
The second and third are similar.
For $f,g\in C^{\infty}(M)$ the inner product on $\Omega^{1}_{\slashed{D}}(M)$ satisfies
\[
2\pairing{[\slashed{D},f]}{[\slashed{D},g]}=[\slashed{D},f]^{*}[\slashed{D},g]+[\slashed{D},g][\slashed{D},f]^{*},
\]
and thus for any two homogeneous 1-forms $\omega$ and $\eta$ and all $\theta$ we have
$
n(\pairing{\omega}{\eta}_{\theta})=n(\eta)-n(\omega).
$
Therefore $\Theta(S,\pairing{\omega}{\eta}_\theta)=\Theta(S,\omega)^{-1}\Theta(S,\eta)$ for any $S\in  C^\infty_\alpha(H)$. 
The self-adjointness of $\sigma_\theta$ and $\Psi_\theta$ follows directly from Lemma \ref{lem:braiding-theta} and the definition of $\Psi_\theta$.
\end{proof}
The inner product $\pairing{\cdot}{\cdot}_\theta$ also satisfies the positivity property $\pairing{\omega}{\omega}_{\theta}\geq 0$ in $C^{\infty}(M_{\theta})$. Although it is possible to prove this directly from the definition, we will derive it from the existence of a nice class of frames for the undeformed module $\Omega^{1}_{\slashed{D}}(M)$. 

\begin{lemma} 
\label{lem: homogeneous-frame}
The module $\Omega^{1}_{\slashed{D}}(M)$ admits a finite frame $(\omega_{j})$ that is homogeneous for the action of $\mathbb{T}^{2}$. Any such homogeneous frame $(\omega_{j})$ is also a frame for $\Omega^{1}_{\slashed{D}}(M_{\theta})$.
\end{lemma}
\begin{proof}
By \cite[Lemma 2.33]{Cac} there exists a finite dimensional representation $(\pi,V_{\pi})$ of $\mathbb{T}^{2}$ and an equivariant isometry
\[
v:\Omega^{1}_{\slashed{D}}(M)\to V_{\pi}\otimes C^{\infty}(M).
\]
As all irreducible representations of $\mathbb{T}^{2}$ are 1-dimensional, the vector space $V_{\pi}$ decomposes as a direct sum of one-dimensional representations. We can thus choose a basis $e_{j}$ of $V_{\pi}$ consisting of homogeneous elements and set $\omega_{j}:=v^{*}(e_{j}\otimes 1)$, and these form a frame for $\Omega^{1}_{\slashed{D}}(M)$. We claim they also form a frame for $\Omega^{1}_{\slashed{D}}(M_{\theta})$. For homogeneous $\eta$ the computation
\begin{align*}
\sum_{j}\omega_{j}*\pairing{\omega_{j}}{\eta}_{\theta}&=\sum_{j}\lambda^{n_{2}(\omega_{j})(n_{1}(\eta)-n_{1}(\omega_{j}))}\omega_{j}\pairing{\omega_{j}}{\eta}_{\theta}\\
&=\sum_{j}\lambda^{n_{2}(\omega_{j})(n_{1}(\eta)-n_{1}(\omega_{j}))}\lambda^{(n_{1}(\omega_j)-n_1(\eta))n_{2}(\omega_j)}\omega_{j}\pairing{ \omega_{j}}{\eta}
=\sum_{j}\omega_{j}\pairing{\omega_{j}}{\eta}=\eta,
\end{align*}
proves the frame relation, and extending to linear combinations completes the proof.
\end{proof}
\begin{corl} The inner product $\pairing{\cdot}{\cdot}_{\theta}$ satisfies $\pairing{\omega}{\omega}_{\theta}\geq 0$ for all $\omega\in\Omega^{1}_{\slashed{D}}(M_{\theta})$.
\end{corl}
\begin{proof} 
For  a homogeneous frame $(\omega_{j})\subset\Omega^{1}_{\slashed{D}}(M_\theta)$, Lemmas \ref{lem:inn-prod} and \ref{lem: homogeneous-frame} show that in $C^{\infty}(M_{\theta})$ 
\begin{align*}
\pairing{\omega}{\omega}_{\theta}&=\sum_{j}\pairing{\omega}{\omega_{j}*\pairing{\omega_j}{\omega}_{\theta}}_{\theta}
=\sum_{j}\pairing{\omega_{j}}{\omega}_{\theta}^{\dag}*\pairing{\omega_{j}}{\omega}_{\theta}\geq 0,\qquad \omega\in \Omega^{1}_{\slashed{D}}(M_\theta).\qedhere
\end{align*}
\end{proof}
\begin{defn} Let $\mathcal{X}$ be a $\dag$-bimodule and $v=(x_{j})$ a frame for $\X$. The frame $v$ is $\dag$-\emph{invariant} if there is an equality of sets $\{x_{j}\}=\{x_{j}^{\dag}\}$.
\end{defn}
\begin{lemma} 
\label{lem:left-right-frame}
Let $(\omega_{j})$ be a homogeneous right frame for $\Omega^{1}_{\slashed{D}}(M_{\theta})$. Then $(\omega_j^\dag)$ is a right frame as well. In particular $\Omega^{1}_{\slashed{D}}(M_{\theta})$ admits frames $(\omega_j)$ that are homogeneous and $\dag$-invariant.
\end{lemma}
\begin{proof}
We compute using the  properties \eqref{eq:theeta} of the map $\Theta$. Since $\omega_{j}$ is homogeneous we have $\Theta(\omega_{j},\omega_{j}^{\dag})=1$.
Then compute, for homogeneous $\eta$, using the commutation relation
and the behaviour of $\Theta$ on inner products 
\begin{align*}
\sum_{j}\omega_{j}^{\dag}*\pairing{\omega_{j}^{\dag}}{\eta}_{\theta}
&=\sum_{j}\Theta(\omega_{j}^{\dag}, \pairing{\omega_{j}^{\dag}}{\eta}_{\theta})\pairing{ \omega_{j}^{\dag}}{\eta}_{\theta}*\omega_{j}^{\dag}
=\sum_{j}\Theta(\omega_{j}^{\dag},\omega_{j})\Theta(\omega_{j}^{\dag}, \eta)\pairing{\omega_{j}^{\dag}}{\eta}_{\theta}*\omega_{j}^{\dag}\\
&=\sum_{j}\pairing{\eta^{\dag}}{\omega_{j}}_{\theta} *\omega_{j}^{\dag}
=\Big(\sum_{j}\omega_{j}*\pairing{\omega_{j}}{\eta^{\dag}}_{\theta}\Big)^{\dag}=\eta
\end{align*}
and we are done. Given any right frame $(\omega_{j})$ the set $\Big\{\frac{\omega_{j}}{\sqrt{2}}, \frac{\omega_{j}^{\dag}}{\sqrt{2}}\Big\}$ yields a $\dag$-invariant frame.
\end{proof}
\begin{lemma}
\label{lem: Grassmann-bimodule}
If $v=(\omega_{j})$ is a homogeneous $\dag$-invariant 
frame for $\Omega^{1}_{\slashed{D}}(M_{\theta})$ then $\sigma\circ\nablar^{v}=\nablal^{v}$.
\end{lemma}
\begin{proof}
This is checked by computation using 
Lemma \ref{lem:inn-prod} to see that for $\omega$ homogeneous
\begin{align*}
\sigma\circ\nablar^{v}(\omega)&=\sum_{j}\sigma\left(\omega_{j}\otimes \left[\slashed{D},\pairing{\omega_{j}}{\omega}_{\theta}\right]\right)
=\sum_{j}\Theta(\omega_{j},\omega)\left[\slashed{D},\pairing{\omega_{j}}{\omega}_{\theta}\right]\otimes \omega_{j}\\
&=\sum_{j}\left[\slashed{D},\pairing{\omega^{\dag}}{ \omega_{j}^{\dag}}_{\theta}\right]\otimes \omega_{j}=\sum_{j}\left[\slashed{D},\pairing{\omega^{\dag}}{\omega_{j}}_{\theta}\right]\otimes \omega_{j}^{\dag}
=\nablal^{v}(\omega),
\end{align*}
and extending by linearity we are done.
\end{proof}
\subsection{Existence and uniqueness of Hermitian torsion-free connections}
In order to address uniqueness of the $\theta$-deformed Levi-Civita connection, we now identify the centre of the modules $T^n_{\slashed{D}}(M_\theta)$. To this end define
\[\mathbb{Z}^{2}_{\alpha}:=\{(n_1,n_2)\in\mathbb{Z}^{2}: \exists a \in C^{\infty}(M),\,\, n_{\alpha}(a)=(n_1,n_2)\},\]
which is a subgroup of $\mathbb{Z}^{2}$. We then set
\begin{align*}
T^n_{\slashed{D}}(M_\theta)^\Theta:=\textnormal{span}^{\infty}\left\{x\in T^n_{\slashed{D}}(M_\theta)_{k}:k\in\mathbb{Z}^{2},\,\,\forall m\in\mathbb{Z}^{2}_{\alpha} ,\,\,\Theta(m,k)=1\right\}.
\end{align*}
Here the $\textnormal{span}^{\infty}$ is taken to mean all series with rapid decay coefficients.
 \begin{lemma}
 \label{lem:fixed-point}
 Let $\alpha:\mathbb{T}^{2}\to \textnormal{Isom}(M,g)$ be a smooth  homomorphism and $\lambda=e^{2\pi i\theta}\in\mathbb{T}$. 
 Then the central $n$-tensors are given by $\mathcal{Z}(T^{n}_{\slashed{D}})=T^{n}_{\slashed{D}}(M_{\theta})^{\Theta}$. Consequently, if $x\in \mathcal{Z}(T^{n}_{\slashed{D}}(M_\theta))$ is homogeneous then $\Theta(\xi,x)=1$ for all homogeneous $\xi\in T^n_{\slashed{D}}(M_\theta)$.
  \end{lemma}
 \begin{proof}
First observe that $x\in \mathcal{Z}(T^{n}_{\slashed{D}}(M_{\theta}))$ if and only if $xa=ax$ for all homogeneous $a\in C^{\infty}(M_{\theta})$. Then for $x\in \mathcal{Z}(T^{n}_{\slashed{D}}(M_{\theta}))$ we have $x=\sum_{k\in\mathbb{Z}^{2}} x_{k}$ and for homogeneous $a$ the identity $xa=ax$ implies that $ax_{k}=x_{k}a$. This in turn implies that $x\in \mathcal{Z}(T^{n}_{\slashed{D}}(M_{\theta}))$ if and only if $x_k\in \mathcal{Z}(T^{n}_{\slashed{D}}(M_{\theta}))$ for all $k$. Thus to prove the equality $\mathcal{Z}(T^{n}_{\slashed{D}})=T^{n}_{\slashed{D}}(M_{\theta})^{\Theta}$ it suffices to consider homogeneous elements.

Now if $x\in \mathcal{Z}(T^{n}_{\slashed{D}})$ and $a\in C(M_\theta)$ are both homogeneous then for all $\eta\in L^2(M,\slashed{S})$ we have
 \begin{align*}
0= \pairing{ L(x)L(a)\eta}{ (L(a)L(x)-L(x)L(a))\eta}= (\Theta(a,x)-1)\pairing{ L(x)L(a)\eta}{L(x)L(a)\eta}.
 \end{align*}
Since this is true for all $\eta\in L^2(M,\slashed{S})$, 
we have $\Theta(a,x)=1$ for all such $a$ and $x$. Hence $x\in T^{n}_{\slashed{D}}(M_\theta)^\Theta$. 
  Conversely, suppose that $x\in T^{n}_{\slashed{D}}(M_{\theta})^{\Theta}$ and $a\in C(M_\theta)$ are homogeneous. We have
 \[
 x*a=\Theta(x,a)a*x=a*x
 \]
 and so $x\in \mathcal{Z}(T^{n}_{\slashed{D}})$. 
The last statement follows since $\Theta$ depends only on the degrees of its arguments, and all possible degrees of tensors $\xi\in T^{n}_{\slashed{D}}(M_\theta)$ are already exhausted by considering elements of $C(M_\theta)$.
 \end{proof}

 \begin{lemma}
 \label{lem:perm-rep}
 For all $x,\rho,\omega,\eta\in \Omega^1_{\slashed{D}}(M_\theta)$ there is an equality of two-tensors 
\begin{equation}
\sigma_{\theta}\circ\alphal(\rho\otimes\omega\otimes\eta)(x)=\Theta(x,\rho\otimes\omega\otimes\eta)\alphar\circ (1\otimes \sigma_{\theta})(\sigma_{\theta}\otimes 1)(1\otimes \sigma_{\theta})(\rho\otimes\omega\otimes\eta)(x).
\label{eq:left-right-miracle}
\end{equation}
As a consequence there is an equality 
 \begin{equation}
 \sigma_{\theta}\circ\alphal=\alphar\circ (1\otimes \sigma_{\theta})(\sigma_{\theta}\otimes 1)(1\otimes \sigma_{\theta}):\mathcal{Z}(T^{3}_{\slashed{D}}(M_{\theta}))\to \overleftrightarrow{\textnormal{Hom}}^{*}_{\B}(\Omega^{1}_{\slashed{D}}(M_{\theta}),T^{2}_{\slashed{D}}(M_{\theta})),
 \label{eq:left-right-centre}
 \end{equation}
as maps. In particular the map
\begin{equation}
\alphar+\sigma_{\theta}\circ\alphal: \mathcal{Z}(\textnormal{Sym}^{3}_{\slashed{D}}(M_{\theta}))\to \overleftrightarrow{\textnormal{Hom}}^{*}_{\B}(\Omega^{1}_{\slashed{D}}(M_{\theta}),T^{2}_{\slashed{D}}(M_{\theta})),
\label{eq:left-rightcentre-sym}
\end{equation}
is injective.
\end{lemma}
\begin{proof} 
Fix one-forms $\rho,\omega,\eta$ and $x$. Then 
to see \eqref{eq:left-right-miracle}, first apply the definitions and use $\Theta$ for commutations and swapping from left to right inner product, giving
\begin{align*}
\sigma_{\theta}\circ\alphal(\rho\otimes\omega\otimes\eta)(x)
&=\Theta(\pairing{x^\dag}{\rho^\dag}_\theta *\omega,\eta)\,\eta\ox \pairing{x^\dag}{\rho^\dag}_\theta *\omega\\
&=\Theta(\pairing{x^\dag}{\rho^\dag}_\theta *\omega,\eta)\,\Theta(\pairing{x^\dag}{\rho^\dag}_\theta,\omega)\,\eta\ox \omega*\pairing{x^\dag}{\rho^\dag}_\theta\\
&=\Theta(\pairing{x^\dag}{\rho^\dag}_\theta *\omega,\eta)\,\Theta(\pairing{x^\dag}{\rho^\dag}_\theta,\omega)\,\Theta(x^\dag,\rho^\dag)\,\eta\ox \omega*\pairing{\rho}{x}_\theta\\
&=\Theta(x,\eta\ox\omega\ox\rho)\,\Theta(\eta,\rho)\,\Theta(\omega,\eta)\,\Theta(\omega,\rho)\,\eta\ox \omega*\pairing{\rho}{x}_\theta
\end{align*}
where we have used Lemma \ref{lem:inn-prod} and the relations \eqref{eq:theeta} repeatedly. The definition of $\alphar$ and $\sigma_{\theta}$ now yields
\begin{align*}
\sigma_{\theta}\circ\alphal(\rho\otimes\omega\otimes\eta)(x)
&=\Theta(x,\eta\ox\omega\ox\rho)\,\alphar\big(\Theta(\eta,\rho)\,\Theta(\omega,\eta)\,\Theta(\omega,\rho)\,\eta\ox \omega\ox\rho\big)(x)\\
&=\Theta(x,\eta\ox\omega\ox\rho)\,\alphar\big(\sigma_{\theta}\ox1\circ1\ox\sigma_{\theta}\circ\sigma_{\theta}\ox1(\eta\ox \omega\ox\rho)\big)(x).
\end{align*}
Restricting to $x\in \mathcal{Z}(T^3_{\slashed{D}}(M_\theta))$, by Lemma \ref{lem:fixed-point} $\Theta(x,\eta\ox\omega\ox\rho)=1$, and so \eqref{eq:left-right-centre} holds.
Finally, 
\[
\mathcal{Z}(\textnormal{Sym}^{3}_{\slashed{D}}(M_{\theta}))=\mathcal{Z}(T^{3}_{\slashed{D}}(M_{\theta}))\cap \textnormal{Sym}^{3}_{\slashed{D}}(M_{\theta}),
\]
by Lemma \ref{lem:fixed-point}, so we have that $\Theta(\rho\otimes\omega\otimes\eta,x)=1$ for all $\rho\otimes\omega\otimes\eta\in \mathcal{Z}(\textnormal{Sym}^{3}_{\slashed{D}}(M_{\theta}))$. Furthermore $(1\otimes \sigma_{\theta})(\sigma_{\theta}\otimes 1)(1\otimes \sigma_{\theta})=1$ on $\textnormal{Sym}^{3}_{\slashed{D}}(M_{\theta})$ by definition of $\textnormal{Sym}^{3}_{\slashed{D}}(M_{\theta})$. Therefore $\alphar+\sigma_{\theta}\circ\alphal=2\alphar$ on $\mathcal{Z}(\textnormal{Sym}^{3}_{\slashed{D}}(M_{\theta}))$, which is injective.
\end{proof}
We now prove the existence and uniqueness of an Hermitian torsion-free $\dag$-bimodule connection for $\theta$-deformations. We denote the unique such connection by $\nablar^{G_\theta}$, where $G_\theta$ is the quantum metric for the inner product $\pairing{\cdot}{\cdot}_\theta$.
\begin{thm} 
\label{thm:theta-unique}
Let $M$ be a compact Riemannian manifold, $\slashed{S}\to M$ a $\mathbb{T}^2$-equivariant Dirac bundle, 
and $\theta\in\mathbb{T}$. Let $\sigma_{\theta}: T^{2}_{\slashed{D}}(M_{\theta})\to T^{2}_{\slashed{D}}(M_{\theta})$ be
the $\theta$-deformed flip map, $\Psi_{\theta}=(1+\sigma_\theta)/2$  and $\pairing{\cdot}{\cdot}_{\theta}$ the $\theta$-deformed inner product. Then
$(\Omega^{1}_{\slashed{D}}(M_{\theta}),\dag,\Psi_{\theta},\pairing{\cdot}{\cdot}_{\theta})$ is a $\dag$-concordant Hermitian differential structure,
and
there exists a unique Hermitian
  torsion-free $\dag$-bimodule connection $(\nablar^{G_{\theta}},\sigma_{\theta})$ on $\Omega^{1}_{\slashed{D}}(M_{\theta})$.
\end{thm}
\begin{proof} 
Observe that by Lemma \ref{lem:braiding-theta}, $\sigma_\theta\otimes 1$ and $1\otimes\sigma_\theta$ generate a representation of the symmetric group $S_{3}$, so that by Corollary \ref{cor: S3action} the differential structure is concordant. Let $(\omega_{j})$ be a homogeneous $\dag$-invariant frame for $\Omega^{1}_{\slashed{D}}(M_{\theta})$. Then the three-tensor
$W=\sum_{j} \dt\omega_{j}\otimes \omega_{j}^\dag$
has the property that

$$
W^\dag=\sum_{j} \omega_{j}\otimes\dt\omega_{j}^{\dag}=\sum_{j}\omega_{j}^\dag\otimes \dt\omega_{j}.
$$ 

Here we note that $2P-1=\sigma_{\theta}\otimes 1$ and $2Q-1=1\otimes \sigma_{\theta}$ so that $(2P-1)(2Q-1)$ is the twisted cyclic permutation $$\omega\otimes \eta\otimes \epsilon\mapsto \Theta(\omega,\epsilon)\Theta(\eta,\epsilon)\epsilon\otimes\omega\otimes \eta.$$
By Lemma \ref{lem: theta-junk-exterior},  $\dt\omega_{j}$ is homogeneous of the same degree as $\omega_{j}$ so that by Lemma \ref{lem:fixed-point} we find 
\[
W^{\dag}=(2P-1)(2Q-1)W.
\] 
By Proposition \ref{prop:suff-for-exist}, the Hermitian structure is $\dag$-concordant and
the Hermitian torsion-free right connection constructed in Theorem \ref{thm:existence} equals
\[
\nablar^{G_{\theta}}:=\nablar^{v}-\alphar((1+4PQ)W).
\]
It remains to show that $\nablar^{G_{\theta}}$ is a $\dag$-bimodule connection relative to the twisted flip map $\sigma_{\theta}$, as such connections are unique by Theorem \ref{thm: uniqueness-twist}. To this end, using Lemma \ref{lem: Grassmann-bimodule} we compute
\begin{align*}
\sigma_{\theta}\circ\nablar^{G_{\theta}}=\sigma_{\theta}\circ\nablar^{v}-\sigma_{\theta}\circ\alphar((1+4PQ)W)=\nablal^{v}-\sigma_{\theta}\circ\alphar((1+4PQ)W),
\end{align*}
and since $-\dag\circ\nablar^{G_{\theta}}\circ\dag=\nablal^{v}+\alphal((1+4QP)W^{\dag}),$ we need to prove that the relation $\alphal((1+4QP)W^{\dag})=-\sigma_{\theta}\circ\alphar((1+4PQ)W)$ holds. By Lemma \ref{lem:perm-rep}, Equation \eqref{eq:left-right-miracle}, we need only show that $(\sigma_{\theta}\ox1)(1\ox\sigma_{\theta})(\sigma_{\theta}\ox1)(1+4PQW)=-(1+QP)W^\dag$.
The equalities
$\sigma_{\theta}\otimes 1 = 2P-1$ and $1\otimes \sigma_{\theta} = 2Q-1$, $(PQ)^{2}W=\frac{PQW}{4}$ and $2QPW^{\dag}=4QPQW$ yield  
\begin{align*}
(2P-1)(2Q-1)(2P-1)(1+4PQ)W&=-W^{\dag}+4(2P-1)(2Q-1)PQW\\
&=-W^{\dag}+4(4PQ-2Q-2P+1)PQW\\
&=-W^{\dag}+16(PQ)^{2}W-4PQW-8QPQW\\
&=-W^{\dag}-8QPQW=-W^{\dag}-4QPW^{\dag},
\end{align*}
as desired. 
\end{proof}

\subsection{Recovering the Levi-Civita connection for manifolds}
By considering the trivial torus action on a manifold $M$, we obtain an immediate corollary to Theorem \ref{thm:theta-unique}. In this case all functions and one-forms are homogeneous of degree zero and thus all structures remain undeformed.

\begin{thm} 
Let $(M,g)$ be a compact Riemannian manifold with a Dirac bundle $\slashed{S}\to M$ and $\sigma: T^{2}_{\slashed{D}}(M)\to T^{2}_{\slashed{D}}(M)$ the standard flip map. Then there exists a unique Hermitian torsion-free $\dag$-bimodule connection $(\nablar^{G},\sigma)$ on $\Omega^{1}_{\slashed{D}}(C^\infty(M))\cong \Omega^1(M)\ox\C$. The restriction $\nablar^{G}:\Omega^{1}(M)\ox1_\C\to \Omega^{1}(M)\ox1_\C$ coincides with the Levi-Civita connection on $\Omega^{1}(M)$.
\end{thm}
\begin{proof}
The proof follows from Theorem \ref{thm:theta-unique}, noting that it is valid for the trivial torus action, but we elaborate on the relation to the classical statement. 

The definition  
of Hermitian connection restricts to the usual definition of compatibility with the metric $g$ for one forms $\omega\ox1_\C$, where $\omega$ is a real one-form. 
As the definition of connection is also the classical one for modules of sections of vector bundles, the definition of torsion-free connection, Definition \ref{defn:torsion-free}, coincides with one of the many equivalent definitions  of torsion-free for the classical case.

By Proposition \ref{lem:basic-real}, the restriction of a torsion-free Hermitian $\dag$-bimodule connection to the real one-forms defines a map $\Omega^1(M)\ox1_\C\to \Omega^1(M)\ox_{C^\infty(M)}\Omega^1(M)\ox1_\C$, and so restricts to a connection on the real bundles. Thus Theorem \ref{thm:theta-unique} yields a unique metric compatible torsion-free connection on the differential one-forms, which is then the Levi-Civita connection on forms by the fundamental theorem of Riemannian geometry.
\end{proof}

To be more explicit, we may compute the local expression for this connection, and the following lemma will provide all the tools. To
ensure that the usual one-forms are the image of a $\dag$-bimodule 
map from universal forms, we set $(\dee x)^\dag=-\dee x$. This is consistent
with using differential forms in a Clifford representation
such that $\dee x^\mu\cdot \dee x^\nu+\dee x^\nu\cdot \dee x^\mu=-2g^{\mu\nu}$, though we do not use this fact.

We will also use summation convention, but caution the reader that repeated indices are not always ``one up, one down'' due to the conflict between differential geometry and module notation. Also, as we work globally, there are implicit sums over a cover by coordinate charts, labelled by $\alpha$ below. We label (local) partial derivatives by $\partial_j(f)=f_{,j}$.

\begin{lemma}
\label{lem:grassy-bit}
Let $(M,g)$ be a complete oriented Riemannian manifold.
Define a (global) frame for $\Gamma(T^*M)$
\begin{equation}
\omega_{j\alpha}=\sqrt{\varphi_\alpha}B_\mu^j\dee x^\mu=\sqrt{\varphi_\alpha}e^j,
\label{eq:real-classy-frame}
\end{equation}
where the $\varphi_\alpha$ are a partition of unity subordinate to a covering of $M$ by coordinate charts $(U_\alpha,x_\alpha)$,
and the $(e^j)$ are a (local) orthonormal frame on each open set.
The $B^\mu_j$ are the change of basis matrix elements $\dee x^\mu=B^\mu_je^j$ and satisfy
\[
B^\mu_j B_\mu^k=\delta_j^k,\quad B^\mu_j B_\nu^j=\delta^\mu_\nu,\quad
B^j_\mu B^k_\nu g^{\mu\nu}=\delta^{jk},\quad B^k_\nu g^{\mu\nu}=B^\mu_j\delta^{jk}.
\]
The Grassmann connection of this frame evaluated on the one-form $f_\nu \dee x^\nu$ is
\begin{align*}
\nablar^v(f_\nu \dee x^\nu)&=\omega_{j\alpha}\ox\dee(\pairing{\omega_{j\alpha}}{f_\nu \dee x^\nu})\\
&=\varphi_\alpha (f_\nu)_{,\rho} \dee x^\nu\ox \dee x^\rho+\varphi_\alpha g^{\sigma\nu} f_\nu\left(B^j_\mu(B^j_\sigma)_{,\rho} -g_{\mu\sigma,\rho}\right) \dee x^\mu\ox \dee x^\rho.
\end{align*}
We also have
\begin{align*}
\alphar(W^\dag)(f_\nu dx^\nu)=\omega_{j\alpha}\ox\pairing{\dee(\omega_{j\alpha})}{f_\nu \dee x^\nu}
=\frac{1}{2}\varphi_\alpha f_\nu g^{\sigma\nu}(B^j_\mu B^j_{\sigma,\rho}-B^j_\mu B^j_{\rho,\sigma})\dee x^\mu \ox \dee x^\rho
\end{align*}
and
\begin{align*}
\alphar(W)(f_\nu \dee x^\nu)=\dee(\omega_{j\alpha})\pairing{\omega_{j\alpha}}{f_\nu \dee x^\nu}=
\frac{1}{2}\varphi_\alpha f_\nu g^{\sigma\nu}\left(B^j_{\rho,\mu}B^j_\sigma-B^j_{\mu,\rho}B^j_\sigma\right)\dee x^\mu\ox \dee x^\rho.
\end{align*}
\end{lemma}
\begin{proof}
The first two statements about the change of basis matrix are clear, while the third follows from the orthonormality of the (local) frame $(e^j)$. The fourth follows from 
\[
B^k_\nu g^{\mu\nu}=\pairing{\dee x^\mu}{B^k_\nu \dee x^\nu}=\pairing{B^\mu_j e^j}{e^k}=B^\mu_j\delta^{jk}.
\]

Let $f_\nu \dee x^\nu$ be a compactly supported section of the cotangent bundle.
The Grassmann connection is computed as 
\begin{align}
\nablar^v(f_\nu \dee x^\nu)
&=\omega_{j\alpha}\ox\dee(\pairing{\omega_{j\alpha}}{f_\nu \dee x^\nu})\nonumber\\
&=\varphi_\alpha B_\mu^j  (B_\sigma^j f_\nu g^{\sigma\nu})_{,\rho} \dee x^\mu\ox \dee x^\rho
+\sqrt{\varphi_\alpha}\sqrt{\varphi_\alpha}_{,\rho}B_\mu^j  B_\sigma^j f_\nu g^{\sigma\nu}\dee x^\mu\ox \dee x^\rho.
\label{eq:grass-class-1}
\end{align}
Since 
\begin{equation}
0=\big(\sum_\alpha\varphi_\alpha\big)_{,\rho}=2\sum\sqrt{\varphi_\alpha}\sqrt{\varphi_\alpha}_{,\rho},
\label{eq:PO1-vanish}
\end{equation}
the second term  on the right hand side of \eqref{eq:grass-class-1} vanishes and we have
\begin{align*}
\nablar^v(f_\nu \dee x^\nu)&=\varphi_\alpha B_\mu^j  (B_\sigma^jg^{\sigma\nu})_{,\rho} f_\nu  \dee x^\mu\ox \dee x^\rho
+\varphi_\alpha B_\mu^j  B_\sigma^j (f_\nu)_{,\rho} g^{\sigma\nu} \dee x^\mu\ox \dee x^\rho\\
&=  \varphi_\alpha B^j_\mu(B^j_\sigma)_{,\rho}g^{\sigma\nu} f_\nu  \dee x^\mu\ox \dee x^\rho
+\varphi_\alpha B^j_\mu B^j_\sigma g^{\sigma\nu}_{,\rho} f_\nu  \dee x^\mu\ox \dee x^\rho+
\varphi_\alpha (f_\nu)_{,\rho} \dee x^\nu\ox \dee x^\rho\\
&=  \varphi_\alpha B^j_\mu(B^j_\sigma)_{,\rho}g^{\sigma\nu} f_\nu  \dee x^\mu\ox \dee x^\rho
+\varphi_\alpha g_{\mu\sigma} g^{\sigma\nu}_{,\rho} f_\nu  \dee x^\mu\ox \dee x^\rho+
\varphi_\alpha (f_\nu)_{,\rho} \dee x^\nu\ox \dee x^\rho\\
&=\varphi_\alpha (f_\nu)_{,\rho} \dee x^\nu\ox \dee x^\rho+\varphi_\alpha g^{\sigma\nu} f_\nu\left(B^j_\mu(B^j_\sigma)_{,\rho} -g_{\mu\sigma,\rho}\right) \dee x^\mu\ox \dee x^\rho.
\end{align*}

Equation \eqref{eq:PO1-vanish} applies in the following computations as well, and we will omit the sums where $\sqrt{\varphi_\alpha}\sqrt{\varphi_\alpha}_{,\rho}$ appears.
Next we compute $\alphar(W)$, 
\begin{align}
\d(\omega_{j\alpha})\pairing{\omega_{j\alpha}}{f_\nu \dee x^\nu}
&=\left(\sqrt{\varphi_\alpha}B^j_\mu\right)_{,\rho}\dee x^\rho\wedge \dee x^\mu \sqrt{\varphi_\alpha}B^j_\sigma f_\nu g^{\sigma\nu}\nonumber\\
&=\big(\frac{1}{2}\varphi_{\alpha,\rho}g_{\mu\sigma}g^{\sigma\nu}f_\nu+\varphi_\alpha B^j_{\mu,\rho}B^j_\sigma f_\nu g^{\sigma\nu}\big)\dee x^\rho\wedge \dee x^\mu\nonumber\\
&=\frac{1}{2}d(\varphi_\alpha)\wedge f_\nu dx^\nu+\frac{1}{2}\varphi_\alpha f_\nu g^{\sigma\nu}\left(B^j_{\mu,\rho}B^j_\sigma-B^j_{\rho,\mu}B^j_\sigma\right)\dee x^\rho\ox \dee x^\mu\nonumber\\
&=\frac{1}{2}\varphi_\alpha f_\nu g^{\sigma\nu}\left(B^j_{\rho,\mu}B^j_\sigma-B^j_{\mu,\rho}B^j_\sigma\right)\dee x^\mu\ox \dee x^\rho.\nonumber
\end{align}

For $W^{\dag}$, we remember that $\dee x^\dag=-\dee x$, and use the Leibniz rule to see that
\begin{align}
\omega_{j\alpha}\ox\pairing{\dee(\omega_{j\alpha})}{f_\nu \dee x^\nu}
&=\sqrt{\varphi_\alpha}B^j_\mu \dee x^\mu\ox \pairing{(\sqrt{\varphi_\alpha}B^j_\sigma)_{,\rho}\dee x^\rho\wedge \dee x^\sigma}{f_\nu \dee x^\nu}\nonumber\\
&=-\frac{1}{4}f_\nu \varphi_{\alpha,\rho}g_{\mu\sigma}\left(g^{\rho\nu} \dee x^\mu\ox \dee x^\sigma-g^{\sigma\nu}\dee x^\mu\ox \dee x^\rho\right)\nonumber\\
&\qquad\qquad-\frac{1}{2}\varphi_\alpha f_\nu B^j_\mu B^j_{\sigma,\rho}\left(g^{\rho\nu} \dee x^\mu\ox \dee x^\sigma-g^{\sigma\nu}\dee x^\mu\ox \dee x^\rho\right)\nonumber\\
&=-\frac{1}{4}\pairing{\dee(\varphi_\alpha)}{f_\nu \dee x^\nu} g_{\mu\sigma}\dee x^\mu\ox \dee x^\sigma
+\frac{1}{4}f_\nu \varphi_{\alpha,\rho}\dee x^\nu\ox \dee x^\rho\nonumber\\
&\qquad\qquad-\frac{1}{2}\varphi_\alpha f_\nu g^{\sigma\nu}(B^j_\mu B^j_{\rho,\sigma}-B^j_\mu B^j_{\sigma,\rho})\dee x^\mu \ox \dee x^\rho\nonumber\\
&=
\frac{1}{2}\varphi_\alpha f_\nu g^{\sigma\nu}(B^j_\mu B^j_{\sigma,\rho}-B^j_\mu B^j_{\rho,\sigma})\dee x^\mu \ox \dee x^\rho\nonumber.
\end{align}
Here we used the fact that $g$ is a tensor to see that
\[
\frac{1}{4}\pairing{\dee(\varphi_\alpha)}{f_\nu dx^\nu} g_{\mu\sigma}\dee x^\mu\ox \dee x^\sigma=0,
\]
because the term $g_{\mu\sigma}\dee x^\mu\ox \dee x^\sigma$ is independent of $\alpha$ and so we can use $\sum_\alpha \dee(\varphi_\alpha)=0$.
\end{proof}
Now observe that $(2P-1)(2Q-1)W=W^\dag$ implies that the connection form is $A=(1+4PQ)W=W^\dag+2QW$ by Proposition \ref{prop:suff-for-exist}. Since $Q$ is just symmetrisation in the last two variables we readily compute the connection form, and so we find
\begin{thm}
\label{thm:classy-LC}
Let $(M,g)$ be a compact oriented Riemannian manifold. For the usual first order calculus $\Omega^1(M)\ox\C$,
the connection form for the Levi-Civita connection expressed in the frame \eqref{eq:real-classy-frame} is
\begin{equation*}
A=\frac{1}{2}(-g_{\mu\rho,\sigma}+g_{\rho\sigma,\mu}+g_{\sigma\mu,\rho})\dee x^{\rho}\ox \dee x^\mu\ox(\dee x^\sigma)^\dag-B^j_{\rho,\mu}B^j_\sigma \dee x^{\rho}\ox \dee x^\mu\ox(\dee x^\sigma)^\dag
\end{equation*}
and so for a one-form $f_\nu dx^\nu$ with support in a single chart we have
\begin{equation*}
\nablar^v(f_\nu \dee x^\nu)-\alphar((1+4PQ)(W))(f_\nu \dee x^\nu)
=f_{\nu,\rho}\, \dee x^\nu\ox \dee x^\rho
-f_\nu \Gamma^\nu_{\mu\rho}\,\dee x^{\rho}\ox \dee x^{\mu},
\end{equation*}
which is the local expression for the classical Levi-Civita connection on one-forms.
\end{thm}

\section{Nonunital algebras and indefinite metrics}
\label{sec:optimism}

\subsection{Noncompact/nonunital examples}

An important simplifying assumption of this paper is the existence of a finite frame for the module of one-forms. In the unital case this (together with an Hermitian inner product) is equivalent to the module being finite projective. In the nonunital case the situation is more difficult. The correct definition, roughly corresponding to a good cover of a manifold or finite topological dimension \cite{R1,RSi}, is that there is a unitisation $\B_b$ of $\B$ such that
\[
\Omega^1_\D(\B)\cong p\B^N,\qquad p=p^*=p^2\in M_N(B_b).
\]
The main point is that all arguments with finite frames can be emulated, in particular the two projection problem can be considered on $p\B_b^N\ox_{\B_b}p\B_b^N\ox_{\B_b}p\B_b^N$ and then restricted to $p\B^N\ox_{\B}p\B^N\ox_{\B}p\B^N$. 

\subsection{Indefinite metrics}
For manifolds we have shown that we can access the one and two forms, and the differential, from a Dirac-type operator and some additional data (the symmetrisation $\Psi$). We have implicitly used that classically $\D$ comes from a Riemannian metric, but pseudo-Riemannian manifolds have the same underlying differential topology. 

In Example \ref{eg:Minky} we indicated that we need the extra data of a bimodule map $\chi:\Omega^1_\dee(\B)\to\Omega^1_\dee(\B)$ such that 
$\chi^2=1$, $\chi(\omega^\dag)=(\chi\omega)^\dag$ for all one-forms $\omega$, and for all $b\in\B$ we have $\chi(\dee(b)^\dag)=-\dee(b^*)$. Given this we can define a new dagger-structure $\dag_\chi=\chi\circ\dag=\dag\circ\chi$ so that we obtain a $\dag$-bimodule map $\pi:(\Omega^1_u(\B),*)\to(\Omega^1_\dee(\B),\dag_\chi)$.

Then we may make all subsequent definitions for this new dagger structure. Provided that we can obtain a $\dag$-concordant Hermitian differential structure, we will then obtain an Hermitian torsion-free connection, unique if there is a compatible braiding. 

When comparing to the original $\dag$-structure, some items are simple. For instance $G=\omega_j\ox\omega_j^{\dag_\chi}=\omega_j\ox\chi\omega_j^\dag$ and so the associated ``Riemannian'' bimodule inner product is
\[
\pairing{G}{\omega\ox\rho}=\pairing{\chi\omega^\dag}{\rho}
\]
and we can recover an indefinite inner product from $\pairing{\omega^\dag}{\rho}$.
We leave further considerations for the future.

\appendix
\section{The relationship between junk and connections}
\label{sec:j-and-c}

There is an important relationship between junk tensors and connections. This relationship has not been used in this article, but some corollaries serve as a useful tool for both determining junk and finding Hermitian torsion-free connections in examples \cite{MRPods}.

\begin{thm}
\label{thm:where-junk}
Let $(\Omega^1_\dee,\dag)$ be a first order differential structure for a local algebra $\B$.
Let $\nablar:\Omega^{1}_{\dee}\to \Omega^{1}_{\dee}\otimes_{\B}\Omega^{1}_{\dee}$ be any right connection on $\Omega^1_\dee$.
The junk two-tensors $JT^2_\dee$ are contained in 
the 
right $\B$-module generated by
$\nablar(\dee(a))$, $a\in\B$.
Hence
\begin{align}
\label{thm:junkinclusionconnection}
JT^{2}_{\dee}\subset\Big(\bigcap_{\nablar}\nablar(\dee(\B))\Big)\cdot\B,
\end{align}
where the intersection is over all right connections.
Moreover, if $(\Omega^1_\dee,\dag)$ has an inner product $\pairing{\cdot}{\cdot}_\B$ on $\Omega^1_\dee$  with a finite frame for the one-forms $(\omega_j)$, there exists a connection $\nablar_T$ on $\Omega^1_\dee$ such that $JT^2_\dee={\rm Im}(\nablar_T\circ\dee)\cdot\B$, and so
\begin{align}
\label{thm:junkequalityconnection}
JT^{2}_{\dee}=\B\cdot \Big(\bigcap_{\nablar}\nablar(\dee(\B))\Big)\cdot\B.
\end{align}
\end{thm}
\begin{proof}
Let $\nablar$ be any right connection on $\Omega^1_\dee$,
and suppose that $\sum_ja_j\delta(b_j)\in\ker(\pi)$, so that 
$\sum_j\dee(a_j)\ox\dee(b_j)$ is junk. 
Then since $\sum_ja_jb_j=0$ (this is true for all one forms $\sum_ja_j\delta(b_j)$),
the Leibniz rule gives $0=\sum_j a_j\dee(b_j)=-\sum_j\dee(a_j)b_j$. 
So we find
$$
0=\nablar\Big(\sum_ja_j\dee(b_j)\Big)\ \iff\ 0=\nablar\Big(\sum_j\dee(a_j)b_j\Big)
$$
and then by the connection property we have
$$
0=\nablar\Big(\sum_j\dee(a_j)b_j\Big)
=\sum_j\nablar(\dee(a_j))b_j+\dee(a_j)\ox\dee(b_j).
$$
Thus a junk two-tensor can be expressed as
\[
\sum_j\dee(a_j)\ox\dee(b_j)=-\sum_j\nablar(\dee(a_j))b_j.
\]
Composing with $m:T^2_\dee\to\Omega^2_\dee$ gives 
$J^2_\dee\subset {\rm Im}(m\circ\nablar\circ\dee)\cdot\B$.
The arbitrariness of the connection $\nablar$ gives
\begin{align}
\label{cor: junkinclusionconnection}
J^{2}_{\dee}&\subset \Big(\bigcap_{\nabla}(m\circ\nabla)(\dee(\B))\Big)\cdot\B,\quad\mbox{and}\quad JT^{2}_{\dee}\subset \Big(\bigcap_{\nabla}\nabla(\dee(\B))\Big)\cdot\B
\end{align}
where the intersections run over all connections.
 
For the reverse containment we define a connection whose image on exact one-forms is contained in the junk.  
Let $v=(\omega_j)$ be a right frame for the inner product $\pairing{\cdot}{\cdot}$ on $\Omega^1_\dee$. Let 
$\alpha_j\in \Omega^1_u(\B)$ be universal forms with $\pi_\dee(\alpha_j)=\omega_j$. Define $T_j\in T^2_\dee$ by $T_j=\pi_\dee(\delta(\alpha_j))$. With $\nablar^v$ the Grassmann connection of the frame, define a connection by
\begin{equation}
\nablar_T(\eta)=\nablar^v(\eta)-T_j\pairing{\omega_j}{\eta}_\B.
\label{eq:nabla-tee}
\end{equation}
Then for $b\in\B$, write $\dee(b)=\omega_j\pairing{\omega_j}{\dee(b)}_\B$, and observe that
\[
\delta(b)-\alpha_j\pairing{\omega_j}{\dee(b)}_\B\in\ker(\pi_d).
\]
Hence we obtain a junk two-tensor
\[
\pi_d(\delta(\delta(b)-\alpha_j\pairing{\omega_j}{\dee(b)}_\B))
=-T_j\pairing{\omega_j}{\dee(b)}_\B+\omega_j\ox\dee(\pairing{\omega_j}{\dee(b)}_\B)=\nablar_T(\dee(b)).
\]
As $b\in\B$ was arbitrary, we see that ${\rm Im}(\nablar_T\circ\dee)\subset JT^2_\dee$, and as junk is a $\B$-bimodule, ${\rm Im}(\nablar_T\circ\dee)\cdot\B\subset JT^2_\dee$. Since $JT^{2}_{\dee}$ is a $\B$-bimodule, we obtain the equality \eqref{thm:junkequalityconnection}.
\end{proof}

\begin{rmk}
For calculi arising from spectral triples (and similar) we can also define the represented 2-tensors, and composing connections with the multiplication $m:T^2_\dee\to\Omega^2_\dee$ gives 
\[
J^{2}_{\dee}\subset \Big(\bigcap_{\nabla}(m\circ\nabla)(\dee(\B))\Big)\cdot\B
\]
with equality when we have an inner product on $\Omega^1_\dee$ with finite frame.
\end{rmk}

\begin{corl}
Suppose the first order differential structure $(\Omega^1_\dee,\dag)$ has an inner product $\pairing{\cdot}{\cdot}_\B$ on $\Omega^1_\dee$  with a finite frame $(\omega_j)$ for the one-forms  such that 
 $\d\omega_j=0$  (so the frame consists of closed one-forms). With $\nablar$ the Grassmann connection of the frame $(\omega_j)$ we have $JT^2_\dee=\B\cdot\nablar(\dee(\B))\cdot\B$ as bimodules.
\end{corl}
\begin{proof}
We may choose $T_j=\pi_\dee(\delta(\alpha_j))=0$ in the definition of the connection $\nablar_T$ from Theorem \ref{thm:where-junk}. Hence $\nablar_T$ is the Grassmann connection.
\end{proof}

\begin{corl}
Suppose the first order differential structure $(\Omega^1_\dee,\dag)$ has an inner product $\pairing{\cdot}{\cdot}_\B$ on $\Omega^1_\dee$  with a finite frame $(\omega_j)$ for the one-forms and  that there exists a projection $\Psi$ on $T^2_\dee$ with range containing the junk. Suppose that $\d\omega_j=0$  (so the frame consists of closed one-forms).  If $\nablar$ the Grassmann connection of the frame $(\omega_j)$, then $\nablar$ is Hermitian and torsion-free.
\end{corl}
\begin{proof}
We note first that $T_j=0$, and so $\nablar_T$ is the Grassmann connection of the closed frame.
Then we readily compute that for a one-form $\eta$
\[
(1-\Psi)\circ\nablar_T(\eta)=(1-\Psi)\big(\sum\omega_j\ox\dee(\pairing{\omega_j}{\eta})\big)=\d(\omega_j)\pairing{\omega_j}{\eta}-\d(\omega_j\pairing{\omega_j}{\eta})=\d(\eta).
\] 
Hence $\nablar_T$ is torsion-free, and as Grassmann connections  are Hermitian, we are done. 
\end{proof}


\end{document}